\documentclass[10pt,a4paper,utf8]{article}
\usepackage{preambule}
\title{Global well-posedness of partially periodic {KP-I} equation in the energy space and application}
\date{}
\author{Tristan Robert\\ \emph{Université de Cergy-Pontoise}\\ \emph{Laboratoire AGM}\\ \emph{2 av. Adolphe Chauvin, 95302 Cergy-Pontoise Cedex, France}\\ \emph{tristan.robert@u-cergy.fr}}

\begin{document}
\maketitle
\begin{abstract}
In this article, we address the Cauchy problem for the KP-I equation \[\drt u + \drx^3 u -\drx^{-1}\dry^2u + u\drx u = 0\] for functions periodic in $y$. We prove global well-posedness of this problem for any data in the energy space $\E = \{u\in L^2,~\drx u \in L^2,~\drx^{-1}\dry u \in L^2\}$. We then prove that the KdV line soliton, seen as a special solution of KP-I equation, is orbitally stable under this flow, as long as its speed is small enough.
\end{abstract}
\section{Introduction}
\subsection{Motivations}
The Kadomtsev-Petviashvili equations
\begin{equation}\label{KP equations}
\drt u + \drx^3 u + \epsilon \drx^{-1}\dry^2 u + u\drx u = 0
\end{equation}
were first introduced in \cite{KP1970} as two-dimensional generalizations of the Korteweig-de Vries equation
\begin{equation}\label{KdV equation}
\drt u + \drx^3 u + u\drx u = 0
\end{equation} They model long, weakly nonlinear waves propagating essentially along the $x$ direction with a small dependence in the $y$ variable. The coefficient $\epsilon \in \{-1;1\}$ takes into account the surface tension. When this latter is strong ($\epsilon = -1$), (\ref{KP equations}) is then called KP-I equation, whereas KP-II equation refers to a small surface tension ($\epsilon = +1$).

The KdV equation (\ref{KdV equation}) admits a particular family of traveling waves solutions, the so-called solitons $Q_c(x-ct)$ with speed $c>0$ :
\[Q_c(x) := 3c\cdot\cosh\left(\frac{\sqrt{c}}{2}x\right)^{-2}\]
From the work of Benjamin \cite{Benjamin}, we know that these solutions are orbitally stable in $H^1(\R)$ under the flow generated by the KdV equation (\ref{KdV equation}), meaning that every solution of (\ref{KdV equation}) with initial data close to $Q_c$ in $H^1(\R)$ remains close in $H^1(\R)$ to the $Q_c$-orbit (under the action of translations) at any time $t>0$.

Looking at (\ref{KP equations}), we see that every solution of the KdV equation (\ref{KdV equation}) is a solution of the KP equations (\ref{KP equations}), seen as a function independent of $y$. It is then a natural question to ask whether $Q_c$ is orbitally stable or unstable under the flow generated by (\ref{KP equations}). In order to do so, we first need a global well-posedness theory for (\ref{KP equations}) in a space containing $Q_c$. In particular, this rules out any well-posedness result in an anisotropic Sobolev space $H^{s_1,s_2}(\R^2)$. A more suited space to look for is the energy space for functions periodic in $y$ : 
\begin{equation}\label{definition espace donnée initiale}
\E(\R\times\T) := \left\{u_0(x,y)\in L^2(\R\times\T),~~\drx u_0\in L^2(\R\times\T),~~\drx^{-1}\dry u_0 \in L^2(\R\times\T)\right\}
\end{equation}
where $\T = \R/2\pi\Z$. Indeed, due to the Hamiltonian structure of (\ref{KP equations}), the mass
\begin{equation}\label{definition masse KP}
\mathcal{M}(u)(t) := \int_{\R\times\T}u^2(t,x,y)\dx\dy
\end{equation}
and the energy
\begin{equation}\label{definition energie KP}
\mathcal{E}(u)(t) := \int_{\R\times\T}\left\{(\drx u)^2(t,x,y) + (\drx^{-1}\dry u)^2(t,x,y) - \frac{1}{3} u^3(t,x,y)\right\}\dx\dy
\end{equation}
are (at least formally) conserved by the flow, i.e. $\mathcal{M}(u)(t) = \mathcal{M}(u)(0)$ and\\ $\mathcal{E}(u)(t) = \mathcal{E}(u)(0)$, for any time $t$ and any solution $u$ of the KP-I equation defined on $[0,t]$.  
The conservation of the energy allows one to extend local solutions in $\mathcal{C}([-T,T],\E)$ into solutions globally defined. In this article, we thus focus on the following  Cauchy problem for the KP-I equation set on $\R\times \T$ :
\begin{equation}\label{equation KP1}
\begin{cases}\drt u + \drx^3 u -\drx^{-1}\dry^2 u + u\drx u = 0,~~(t,x,y)\in \R^2\times\T\\ u(t=0)=u_0\in \E(\R\times\T)\end{cases}
\end{equation}
\subsection{Well-posedness results}
The KP equations (\ref{KP equations}) have been extensively studied in the past few decades. Using a standard energy method, I\'{o}rio and Nunes \cite{IorioNunes} proved existence and uniqueness of zero mean value solutions in $H^s$, $s>2$, for both KP equations on $\R^2$ and $\T^2$. From the point of view of well-posedness, the KP-II equation is much better understood. Indeed, since the pioneering work of Bourgain \cite{bourgain1993kp}, we know that the KP-II equation is globally well-posed on both $L^2(\R^2)$ and $L^2(\T^2)$. On $\R^2$, Takaoka and Tzvetkov \cite{Takaoka} and Isaza and Mejia \cite{IsazaMejia} pushed the low regularity local well-posedness theory down to the anistropic Sobolev space $H^{s_1,s_2}(\R^2)$ with $s_1>-1/3$, $s_2\supeg 0$. Later, Hadac \cite{Hadac} and then Hadac, Herr and Koch \cite{HadacHerrKoch} reached the threshold $s_1\supeg -1/2$, $s_2\supeg 0$ which is the scaling critical regularity for the KP-II equation. As for the initial value problem on $\R\times\T$, in order to study the stability of the KdV soliton under the flow of the KP-II equation, Molinet, Saut and Tzvetkov \cite{MST2011} proved global well-posedness on $L^2(\R\times\T)$.

The situation is radically different regarding the Cauchy theory for the KP-I equation. From the work of Molinet, Saut and Tzvetkov \cite{MST2002}, we know that this equation badly behaves with respect to pertubation methods. In particular, it is not possible to get well-posedness of (\ref{equation KP1}) using the standard Fourier restriction norm method of Bourgain, nor any method using a fixed point argument on the Duhamel formula associated with (\ref{equation KP1}) since Koch and Tzvetkov \cite{Koch2008} proved that on $\R^2$, the flow map even fails to be uniformly continuous on bounded sets of $\mathcal{C}([-T,T],\E)$. It is thus expected to have the same ill-posedness result on $\R\times\T$. Using the refined energy method introduced in \cite{KochTzvetkov2003BO}, Kenig \cite{Kenig2004}, and then Ionescu and Kenig \cite{ionescu2009} proved global well-posedness in the "second energy space"
\[Z^2 =: \left\{u\in L^2,~\drx^2 u \in L^2,~\drx^{-2}\dry^2 u \in L^2\right\}\]
for functions on $\R^2$, and both $\R\times\T$ and $\T^2$, respectively. Lately, Ionescu, Kenig and Tataru \cite{IonescuKenigTataru2008} introduced the so-called short time Fourier restriction norm method and were able to prove global well-posedness of the KP-I equation in the energy space $\E(\R^2)$. Zhang \cite{Zhang2015} adapted this method in the periodic setting and got local well-posedness in the Besov space $\B^1_{2,1}(\T^2)$, which is almost the energy space but still strictly embedded in it. Overcoming the logarithmic divergence that appears in \cite{Zhang2015} to reach the energy space $\E(\T^2)$ is still an important open problem. In our case, we prove the following theorem, which answers the global well-posedness issue in the partially periodic setting :
\begin{theoreme}\label{theoreme principal}
\begin{enumerate}[label=(\alph*)]
\item\label{theoreme principal partie 1} \emph{Global well-posedness for smooth data}

Take $u_0\in \E^{\infty}(\R\times\T)$. Then, (\ref{equation KP1}) admits a unique global solution \[u = \Phi^{\infty}(u_0)\in\mathcal{C}\left(\R,\E^{\infty}(\R\times\T)\right)\]
which defines a flow map
\[\Phi^{\infty} : \E^{\infty}(\R\times\T) \rightarrow \mathcal{C}\left(\R,\E^{\infty}(\R\times\T)\right)\]
In addition, for any $T>0$ and $\alpha\in\N^*$,
\begin{equation}\label{equation estimation norme}
\norme{L^{\infty}_T\E^{\alpha}}{\Phi^{\infty}(u_0)}\infeg C(T,\alpha,\norme{\E^{\alpha}}{u_0})
\end{equation}
\item\label{theoreme principal partie 2} \emph{Global well-posedness in the energy space} 

For any $u_0\in \E(\R\times\T)$ and $T>0$, there exists a unique solution $u$ to (\ref{equation KP1}) in the class
\begin{equation}\label{classe unicite}
\mathcal{C}([-T;T],\E)\cap \mathbf{F}(T)\cap \B(T)
\end{equation}
Moreover, the corresponding global flow
\[\Phi^1 : \E \rightarrow \mathcal{C}(\R,\E)\]
is continuous and leaves $\mathcal{M}$ and $\mathcal{E}$ invariants.
\end{enumerate}
\end{theoreme}
The function spaces $\El^{\alpha}$, $\El^{\infty}$, $\mathbf{F}$ and $\B$ are defined in section~\ref{section espaces fonctions} below.
\subsection{Stability results}
As far as stability issues are concerned, Mizumachi and Tzvetkov \cite{Mizumachi2012} proved that the KdV line soliton is stable under the flow generated by the KP-II equation on $L^2(\R\times\T)$ for any speed $c>0$. Regarding the KP-I equation, Rousset and Tzvetkov \cite{Rousset2012} proved that $Q_c$ is orbitally unstable in $\E^1(\R\times\T)$ under the KP-I flow constructed on $Z^2(\R\times\T)$ in \cite{ionescu2009}, whenever $c> c^* = 4/\sqrt{3}$, and that it is orbitally stable if $c<c^*$. Thus, as a byproduct of \cite{Rousset2012} and of our theorem~\ref{theoreme principal}, we can extend the range of admissible perturbations in \cite[Theorem 1.4]{Rousset2012} to get
\begin{corollaire}\label{theoreme instabilite}
Assume $c<4/\sqrt{3}$, then $Q_c$ is orbitally stable in $\E$.

More precisely, for every $\varepsilon>0$, there exists $\delta>0$ such that for every $u_0\in \E(\R\times\T)$ satisfying
\[\norme{\E(\R\times\T)}{u_0-Q_c}<\delta\]
we have
\[\sup_{t\in\R}\inf_{a\in\R}\norme{\E(\R\times\T)}{\Phi^1(u_0)(t,x-a,y)-Q_c(x-ct)}<\varepsilon\]
\end{corollaire}
The proof of corollary~\ref{theoreme instabilite}~is a straightforward adaptation of the argument in \cite{Rousset2012}. Indeed, the proof of \cite[Theorem 1.4]{Rousset2012} only uses the extra conditions $\drx^2 u \in L^2$, $\drx^{-2}\dry^2 u \in L^2$ to have the global solutions from \cite{ionescu2009}. For the sake of completeness, we present the outlines of the proof in section~\ref{section stabilite}.
\subsection{Strategy of the proof}
Let us now briefly discuss the main ingredients in the proof of theorem~\ref{theoreme principal}.

As pointed out above, it is irrelevant to look for functions spaces $\mathbf{F}(T)\hookrightarrow \mathcal{C}([-T,T],\E)$ and $\mathbf{N}(T)$ such that any solution to (\ref{equation KP1}) satisfies
\begin{enumerate}
\item a linear estimate
\begin{equation}\label{estimation lineaire Bourgain}
\norme{\mathbf{F}(T)}{u} \lesssim \norme{\E}{u_0}+\norme{\mathbf{N}(T)}{\drx(u^2)}
\end{equation}
\item a bilinear estimate
\begin{equation}\label{estimation bilineaire Bourgain}
\norme{\mathbf{N}(T)}{\drx(uv)}\lesssim \norme{\mathbf{F}(T)}{u}\norme{\mathbf{F}(T)}{v}
\end{equation}
\end{enumerate}
In order to construct solutions in $\E$, we will thus use the functions spaces $\mathbf{F}(T)$, $\mathbf{N}(T)$ and $\B(T)$ introduced in \cite{IonescuKenigTataru2008}. Those spaces are built to combine the idea introduced in \cite{KochTzvetkov2003BO} of a priori estimates on short times (depending on the frequency) for frequency localized solutions, with the standard Bourgain spaces $X^{s,b}$ of \cite{bourgain1993kp}. Thus, we will replace (\ref{estimation lineaire Bourgain})-(\ref{estimation bilineaire Bourgain}) with
\begin{enumerate}
\item a linear estimate
\begin{equation}\label{estimation lineaire}
\norme{\mathbf{F}(T)}{u}\lesssim \norme{\B(T)}{u}+\norme{\mathbf{N}(T)}{\drx(u^2)}
\end{equation}
\item a bilinear estimate
\begin{equation}\label{estimation bilineaire}
\norme{\mathbf{N}(T)}{\drx(uv)}\lesssim \norme{\mathbf{F}(T)}{u}\norme{\mathbf{F}(T)}{v}
\end{equation}
\item an energy estimate
\begin{equation}\label{estimation d'energie}
\norme{\B(T)}{u}^2\lesssim \norme{\E}{u_0}^2+\norme{\mathbf{F}(T)}{u}^{3}
\end{equation}
\end{enumerate}
With (\ref{estimation lineaire})-(\ref{estimation bilineaire})-(\ref{estimation d'energie}) at hand, we will get the existence part of theorem~\ref{theoreme principal} from a standard continuity argument.

To get uniqueness, we will prove the analogous of (\ref{estimation lineaire})-(\ref{estimation bilineaire})-(\ref{estimation d'energie}) for the difference equation, at the $L^2$ level :
\begin{eqnarray}
\norme{\overline{\mathbf{F}}(T)}{u-v}\lesssim \norme{\overline{\B}(T)}{u-v}+\norme{\overline{\mathbf{N}}(T)}{\drx\{(u-v)(u+v)\}}\label{estimation lineaire difference}\\
\norme{\overline{\mathbf{N}}(T)}{\drx\{(u-v)(u+v)\}}\lesssim \norme{\overline{\mathbf{F}}(T)}{u-v}\norme{\mathbf{F}(T)}{u+v}\label{estimation bilineaire difference}\\
\norme{\overline{\B}(T)}{u-v}^2\lesssim \norme{L^2}{u_0-v_0}^2+\norme{\mathbf{F}(T)}{u+v}\norme{\overline{\mathbf{F}}(T)}{u-v}^{2}\label{estimation d'energie difference}
\end{eqnarray}
The main technical difficulties, compared to the case of $\R^2$, are the lack of a scale-invariant Strichartz estimate, and the impossibility to make the change of variables as in the proof of \cite[Lemma 5.1 (a)]{IonescuKenigTataru2008} to estimate the volume of the resonant set. The first one is handled with frequency localized Strichartz estimates in the spirit of \cite{bourgain1993kp,MST2011}. For the second one, we follow Zhang \cite[Lemma 3.1]{Zhang2015}, but looking closely on the computations we are able to take advantage of the smallness of the intervals in which the frequency for the $x$ variable varies (note that this is not possible in \cite{Zhang2015} since this frequency lives in $\Z$) and to recover the same estimate as in \cite{IonescuKenigTataru2008} in this case. We also use a weighted Bourgain type space to deal with the logarithmic divergence in the energy estimate.

\subsection{Organization of the paper}
Sections~\ref{section notations} and~\ref{section espaces fonctions} introduce general notations as well as functions spaces. We begin the proof of theorem~\ref{theoreme principal} in section~\ref{section estimation lineaire} by proving estimate (\ref{estimation lineaire}). After establishing some general dyadic estimates in section~\ref{section estimation dyadique}, sections~\ref{section estimation bilineaire} and~\ref{section estimation energie} deal with (\ref{estimation bilineaire}) and (\ref{estimation d'energie}) respectively. The proof of theorem~\ref{theoreme principal} is then completed in section~\ref{section preuve}. Finally, in the last section~\ref{section stabilite} we recall the arguments to obtain corollary~\ref{theoreme instabilite}.
\section{Notations}\label{section notations}
\begin{itemize}
\item We use the notations of \cite{Molinet2007} to deal with Fourier transform of periodic functions with a large period $2\pi\lambda >0$. Let $\lambda\supeg 1$ be fixed. We define $\dq$ to be the renormalized counting measure on $\lambda^{-1}\Z$ :
\[\int u(q)\dq := \frac{1}{\lambda}\sum_{q\in\lambda^{-1}\Z}u(q)\]
In the sequel, all the Lebesgue norms in $q$ will be with respect to $\dq$. Moreover, the space-time Lebesgue norms are defined as
\[\norme{L^p_{\xi,q} L^r_{\tau}}{f}:= \left\{\int_{\R\times\Zl}\left(\int_{\R}|f|^r\dtau\right)^{p/r}\dxi\dq{}\right\}^{1/p}\]

For a $2\pi\lambda$-periodic function $f$, we define its Fourier transform as 
\[\widehat{f}(q):=\int_0^{2\pi\lambda}\e^{-i qx}f(y)\dy,~q\in \lambda^{-1}\Z\]
and we have the inversion formula
\[f(y)= \int \e^{i qy}\widehat{f}(q)\dq\]
 We write $\T_{\lambda} := \R/2\pi\lambda\Z$. Whenever $\lambda =1$ we drop the lambda.
\item The Fourier transform of a function $u_0(x,y)$ on $\R\times\T_{\lambda}$ or $u(t,x,y)$ on $\R^2 \times \T_{\lambda}$ is denoted $\widehat{u}$ or $\F u$ :
\[\widehat{u_0}(\xi,q) := \int_{\R\times\T_{\lambda}} \e^{-i(\xi x+qy)}u_0(x,y)\dx\dy,~(\xi,q)\in\R\times \lambda^{-1}\Z\]
and
\[\widehat{u}(\tau,\xi,q) := \int_{\R^2\times\T_{\lambda}}\e^{-i(\tau t+\xi x+qy)}u(t,x,y)\dt\dx\dy,~(\tau,\xi,q)\in\R^2\times \lambda^{-1}\Z\]
$\F_t u$ stands for the partial Fourier transform of $u(t,x,y)$ with respect to $t$, whereas $\F_{xy} u$ means the partial Fourier transform of $u$ with respect to space variables $x,y$, and similarly for $\F_x$, $\F_y$.\\
We always note $(\tau,\xi,q)\in \R^2\times \lambda^{-1}\Z$ the Fourier variables associated with $(t,x,y)\in\R^2\times\T_{\lambda}$.\\
We note eventually $\zeta = (\xi,q) \in \R\times \lambda^{-1}\Z$.
\item We denote $\star$ the convolution product for functions on $\R$ or $\lambda^{-1}\Z$ : to specify the variables,
\[ f(x')\star_x g(x') \text{ means } (f\star g)(x) = \int_{\R \text{ or }\lambda^{-1}\Z}f(x-x')g(x')\dx'\]
\item We use the "bracket" notation $\crochet{\cdot}$ for the weight in the definition of inhomogeneous Sobolev spaces, i.e
\[\crochet{\xi}^s:= \left(1+\xi^2\right)^{s/2}\]
\item $U(t)$ is the unitary group defined by the linear evolution equation associated with (\ref{equation KP1}) : 
\[\forall u_0\in L^2(\R\times\T_{\lambda}),~~\widehat{U(t)u_0}(\xi,q)= \e^{it\omega(\xi,q)}\widehat{u_0}(\xi,q)\]where
\[\omega(\xi,q):= \xi^3+q^2/\xi\]
We also note
\[\sigma(\tau,\xi,q) := \tau - \omega(\xi,q) = \tau - \xi^3 - \frac{q^2}{\xi}\]
the modulation associated with (\ref{equation KP1}).
\item For positive reals $a$ and $b$, $a\lesssim b$ means that there exists a positive constant $c>0$ (independent of the various parameters, including $\lambda$) such that $a \infeg c\cdot b$. \\
The notation $a\sim b$ stands for $a\lesssim b$ and $b\lesssim a$.
\item We note $M\in \R_+^*$ (respectively $K\supeg 1$) the dyadic frequency decomposition of $|\xi|$ (respectively of $\crochet{\sigma}$), i.e $M\in 2^{\Z}$ and $K\in 2^{\N}$.\\
We define then
\[D_{\lambda,M,K} := \left\{(\tau,\xi,q)\in \R^2\times\lambda^{-1}\Z,~~|\xi| \sim M, \crochet{\sigma(\tau,\xi,q)}\sim K \right\}\]
and
\[D_{\lambda,M,\infeg K} := \left\{(\tau,\xi,q)\in \R^2\times \lambda^{-1}\Z,~~|\xi| \sim M, \crochet{\sigma(\tau,\xi,q)}\lesssim K \right\} = \bigcup_{K'\infeg K} D_{\lambda,M,K'}\]
We note also
\[\mathfrak{I}_M := \left\{M/2\infeg |\xi|\infeg 3M/2\right\}\]
and
\[\mathfrak{I}_{\infeg M} := \left\{|\xi|\infeg 3M/2\right\} = \bigcup_{M'\infeg M}\mathfrak{I}_{M'}\]
\item We use the notations $M_1\et M_2 := \min(M_1,M_2)$ and $M_1 \ou M_2 := \max(M_1,M_2)$.\\
For $M_1,M_2,M_3 \in \R_+^*$, $M_{\min}\infeg M_{med}\infeg M_{max}$ denotes the increasing rearrangement of $M_1,M_2,M_3$, i.e 
\begin{multline*}
M_{min} := M_1\et M_2 \et M_3,~~ M_{max} = M_1 \ou M_2 \ou M_3 \\ \text{ and } M_{med} = M_1+M_2+M_3 - M_{max} - M_{min}
\end{multline*}
\item We use two different Littlewood-Paley decompositions : the first one is homogeneous (on $2^{\Z}$) for $|\xi|$, the last one is inhomogeneous for $\crochet{\sigma} \in 2^{\N}$.

Let $\chi\in\mathcal{C}^{\infty}_0(\R)$ with $0\infeg \chi \infeg 1$, $\supp \chi \subset [-8/5;8/5]$ and $\chi \equiv 1$ on $[-5/4;5/4]$.
\begin{itemize}
\item For $M\in 2^{\Z}$, we then define $\eta_M(\xi) := \chi(\xi/M) - \chi(2\xi/M)$, such that $\supp \eta_M \subset \{5/8 M \infeg |\xi|\infeg 8/5 M\}$ and $\eta_M \equiv 1$ on $\{4/5 M\infeg |\xi|\infeg 5/4 M\}$. Thus $\xi \in \supp \eta_M \Rightarrow \xi \in \mathfrak{I}_M$ and $|\xi| \sim M$.
\item For $K\in 2^{\N}$, we also define $\rho_1(\sigma) := \chi(\sigma)$ and $\rho_K(\sigma) := \chi(\sigma/K)-\chi(2\sigma/K)$, $K>1$, such that $\supp \rho_K \subset \{5/8 K \infeg |\sigma|\infeg 8/5 K\}$ and $\rho_K \equiv 1$ on $\{4/5 M\infeg |\sigma|\infeg 5/4 K\}$, $K>1$. Thus $\sigma \in \supp \rho_K \Rightarrow \crochet{\sigma} \sim K$.
\item When needed, we may use other decompositions $\widetilde{\chi}$, $\widetilde{\eta}$ and $\widetilde{\rho}$ with the similar properties as $\chi$, $\eta$, $\rho$ and satisfying $\widetilde{\chi}\equiv 1$ on $\supp \chi$, $\widetilde{\eta}\equiv 1$ on $\supp\eta$ and $\widetilde{\rho}\equiv 1$ on $\supp\rho$.
\item Finally, for $\kappa\in\R_+^*$, we note $\chi_{\kappa}(x) := \chi(x/\kappa)$.
\end{itemize}
\item We also define the Littlewood-Paley projectors associated with the previous decompositions :
\[P_M u := \F^{-1}\left(\eta_M(\xi)\widehat{u}\right) \text{ and }P_{\infeg M} u := \sum_{M'\infeg M}P_M u = \F^{-1}\left(\chi_M(\xi)\widehat{u}\right)\]
Moreover, we define
\[P_{Low} := P_{\infeg 2^{-5}} \text{ and }P_{High}:=1-P_{Low}\]
\item The energy space $\E_{\lambda}$ is defined as in (\ref{definition espace donnée initiale}) for any period $2\pi\lambda$ :
\[\E(\R\times\T_{\lambda}) := \left\{u_0\in L^2(\R\times\T_{\lambda}),~~\drx u_0\in L^2(\R\times\Tl),~\drx^{-1}\dry u_0 \in L^2(\R\times\Tl)\right\}\]
It is endowed with the norm
\[\norme{\El}{u_0}:=\normL{2}{\crochet{\xi}\cdot p(\xi,q)\cdot\widehat{u_0}}\]
i.e $\El$ is a weighted Sobolev space, with the weight defined as
\begin{equation}\label{definition p}
p(\xi,q):= \crochet{\crochet{\xi}^{-1}q/\xi},~(\xi,q)\in\R\times\Zl
\end{equation}
so that
\begin{equation}\label{propriete p}
\left|\crochet{\xi}\cdot p(\xi,q)\right|^2 = 1+\xi^2 +  \frac{q^2}{\xi^2}
\end{equation}
i.e
\[\norme{\El}{u_0}^2 = \normL{2}{u_0}^2+\normL{2}{\drx u_0}^2 + \normL{2}{\drx^{-1}\dry u_0}^2\]
More generally, for $\alpha\in\N$, we define
\begin{equation}\label{definition E sigma}
\El^{\alpha} := \left\{u_0(x,y)\in L^2(\R\times\Tl),~~\norme{\El^{\alpha}}{u_0}:= \normL{2}{\crochet{\xi}^{\alpha}\cdot p(\xi,q)\cdot \widehat{u_0}}<+\infty\right\}
\end{equation}
and
\begin{equation}\label{definition E infini}
\El^{\infty} = \bigcap_{\alpha\in\N^*} \El^{\alpha}
\end{equation}
\item For a real $\xi$, we define 
\[\lpart{\xi} := \lambda^{-1}\lfloor \lambda\xi\rfloor \in \Zl\]
\item For a set $A\subset\R^d$, $\mathbb{1}_A$ is the characteristic function of $A$ and if $A$ is Lebesgue-measurable, $|A|$ means its measure. Similarly, if $A \subset \Zl$, its measure with respect to $\dq$ will also be noted $|A|$.
When $A\subset \Z$ is a finite set, its cardinal is denoted $\#A$.
\item For $M>0$ and $s\in\R$, $\lesssim M^{s-}$ means $\infeg C_{\varepsilon}M^{s-\varepsilon}$ for any choice of $\varepsilon>0$ small enough. We define similarly $M^{s+}$.
\end{itemize}
\section{Functions spaces}\label{section espaces fonctions}
\subsection{Definitions}
Let $M\in 2^{\Z}$.\\
First, the dyadic energy space is defined as
\[\el{M} := \left\{u_0\in \El^0,~P_M u_0 = u_0\right\}\]
As in~\cite{IonescuKenigTataru2008}, for $M\in 2^{\Z}$ and $b_1\in[0;1/2[$, the dyadic Bourgain type space is defined as
\begin{multline*}
\X{M}^{b_1} :=\left\{f(\tau,\xi,q)\in L^2(\R^2\times\Zl), \supp f \subset \R \times I_M \times \Zl, \right. \\ \left.\norme{\X{M}^{b_1}}{f}:= \sum_{K\supeg 1}K^{1/2}\beta_{M,K}^{b_1}\normL{2}{\rho_K(\tau-\omega)f}<+\infty\right\}
\end{multline*}
where the extra weight $\beta_{M,K}$ is
\[\beta_{M,K} := 1\ou \frac{K}{(1\ou M)^3}\]
This weight, already used in \cite{bourgain1993kp,MST2011,GPWW}, allows to recover a bit of derivatives in the high modulation regime, thus preventing a logarithmic divergence in the energy estimate.
Then, we use the $\X{M}^{b_1}$ structure uniformly on time intervals of size $(1\ou M)^{-1}$ :
\begin{multline*}
\fl{M}^{b_1} := \left\{u(t,x,y)\in \mathcal{C}\left(\R,\el{M}\right),~P_M u = u,\right.\\ \left.\norme{\fl{M}^{b_1}}{u}:= \sup_{t_M \in \R} \norme{\X{M}^{b_1}}{p\cdot \F\left\{\chi_{(1\ou M)^{-1}}(t-t_M)u\right\}}<+\infty\right\}
\end{multline*}
and
\begin{multline*}
\nl{M}^{b_1} := \left\{u(t,x,y)\in \mathcal{C}\left(\R,\el{M}\right),~P_M u = u,\right.\\ \left. \norme{\nl{M}^{b_1}}{u}:= \sup_{t_M \in \R} \norme{\X{M}^{b_1}}{(\tau-\omega +i(1\ou M))^{-1}p\cdot \F\left\{\chi_{(1\ou M)^{-1}}(t-t_M)u\right\}}<+\infty\right\}
\end{multline*}
For a function space $Y\hookrightarrow \mathcal{C}(\R,\El^{\alpha})$, we set
\[Y(T) := \left\{u\in\mathcal{C}\left([-T,T],\El^{\alpha}\right),~\norme{Y(T)}{u}<+\infty\right\} \]
endowed with
\begin{equation}\label{definition norme localisation T}
\norme{Y(T)}{u}:=\inf\left\{\norme{Y}{\widetilde{u}},~~\widetilde{u}\in Y,~~\widetilde{u}\equiv u \text{ on }[-T,T]\right\}
\end{equation}
Finally, the main function spaces are defined as
\begin{multline}\label{definition F}
\Fl^{\alpha,b_1}(T) := \left\{u\in\mathcal{C}([-T,T],\El^{\alpha}),\right. \\ \left.\norme{\Fl^{\alpha,b_1}(T)}{u}:= \left(\sum_{M>0} (1\ou M)^{2\alpha} \norme{\fl{M}^{b_1}(T)}{P_M u}^2\right)^{1/2}<+\infty\right\}
\end{multline}
and
\begin{multline}\label{definition N}
\Nl^{\alpha,b_1}(T) := \left\{u\in\mathcal{C}([-T,T],\El^{\alpha}),\right. \\ \left.\norme{\Nl^{\alpha,b_1}(T)}{u}:= \left(\sum_{M>0} (1\ou M)^{2\alpha} \norme{\nl{M}^{b_1}(T)}{P_M u}^2\right)^{1/2}<+\infty\right\}
\end{multline}
The last space is the energy space
\begin{multline}\label{definition espace energie}
\Bl^{\alpha}(T) := \left\{u\in \mathcal{C}([-T,T],\El^{\alpha}),\right. \\ \left.\norme{\Bl^{\alpha}(T)}{u}:=\left(\norme{\El^{\alpha}}{P_{\infeg 1}u_0}^2+\sum_{M> 1}\sup_{t_M\in [-T,T]}\norme{\El^{\alpha}}{P_M u(t_M)}^2\right)^{1/2}<+\infty\right\} 
\end{multline}
For $b_1=1/8$, we drop the exponent.\\
If moreover $\alpha = 1$, we simply write $\Fl(T)$, $\Nl(T)$ et $\Bl(T)$.\\

We define similarly the spaces
\[\overline{\el{M}},~\overline{\fl{M}^{b_1}},~\overline{\nl{M}^{b_1}}\]
which are the equivalents of $\el{M}$, $\fl{M}^{b_1}$, $\nl{M}^{b_1}$ but on an $L^2$ level, i.e without the weight $p(\xi,q)$. In particular,
\begin{equation}\label{equation decomposition norme F}
\norme{\Fl(T)}{u}^2\sim \sum_{M>0}(1\ou M)^2\norme{\overline{\fl{M}^{b_1}}(T)}{u}^2+\norme{\overline{\fl{M}^{b_1}}(T)}{\drx^{-1}\dry u}^2
\end{equation}
For the difference equation, we will then use the $L^2$-type energy space
\begin{multline}\label{definition espace energie difference}
\overline{\Bl}(T) := \left\{u\in \mathcal{C}([-T;T],L^2(\R\times\Tl)),\right.\\ 
\left.\norme{\overline{\Bl}(T)}{u}^2:=\norme{L^2}{P_{\infeg 1}u_0}+\sum_{M> 1}\sup_{t_M\in[-T;T]}\norme{L^2}{P_Mu(t_M)}^2<+\infty\right\}
\end{multline}
and the spaces for the difference of solutions and for the nonlinearity are
\begin{multline}\label{definition F difference}
\overline{\Fl}^{b_1}(T) := \left\{u\in\mathcal{C}([-T;T],L^2(\R\times\Tl)),\right.\\
\left.\norme{\overline{\Fl}^{b_1}(T)}{u}^2:=\sum_{M>0} \norme{\overline{\fl{M}^{b_1}}(T)}{P_M u}^2 <+\infty \right\}
\end{multline}
and
\begin{multline}\label{definition N difference}
\overline{\Nl}^{b_1}(T) := \left\{u\in\mathcal{C}([-T;T],L^2(\R\times\Tl)),\right.\\
\left.\norme{\overline{\Nl}^{b_1}(T)}{u}^2:=\sum_{M>0} \norme{\overline{\nl{M}^{b_1}}(T)}{P_M u}^2 <+\infty \right\}
\end{multline}
\subsection{Basic properties}
The following property of dyadic Bourgain type space is fundamental :
\begin{proposition}
Let $M\in 2^{\Z}$, $b_1\in [0;1/2[$, $f_M \in \X{M}^{b_1}$, and $\gamma\in L^2(\R)$ satisfying 
\begin{equation}\label{hypothese localisation en temps X}
|\widehat{\gamma(\tau)}|\lesssim \crochet{\tau}^{-4}
\end{equation}
Then, for any $K_0 \supeg 1$ and $t_0\in\R$ :
\begin{equation}\label{propriete XM}
K_0^{1/2}\beta_{M,K_0}^{b_1}\normL{2}{\chi_{K_0}(\tau-\omega)\F\left\{\gamma(K_0(t-t_0))\F^{-1}f_M\right\}}\lesssim \beta_{M,K_0}^{b_1}\norme{\X{M}^{0}}{f_M}
\end{equation}
and
\begin{equation}\label{estimation cle localisation XM}
\sum_{K\supeg K_0}K^{1/2}\beta_{M,K}^{b_1}\normL{2}{\rho_K(\tau-\omega)\F\left\{\gamma(K_0(t-t_0))\F^{-1}f_M\right\}}\lesssim \norme{\X{M}^{b_1}}{f_M}
\end{equation}
and the implicit constants are independent of $K_0$, $t_0$, $M$ or $\lambda$.
\end{proposition}
We will have several uses of the following estimate
\begin{lemme}
For any $M\in2^{\Z}$ and $f_M\in \X{M}^{0}$, we have
\begin{equation}\label{estimation controle norme L2L1}
\norme{L^2_{\xi,q}L^1_{\tau}}{f_M}\lesssim \norme{\X{M}^{0}}{f_M}
\end{equation}
\end{lemme}
\begin{proof}
We decompose $f_M$ according to its modulations :
\begin{multline*}
\norme{L^2_{\xi,q}L^1_{\tau}}{f_M} \infeg \sum_{K\supeg 1} \norme{L^2_{\xi,q}L^1_{\tau}}{\rho_K(\tau-\omega)f_M}\\
\lesssim \sum_{K\supeg 1} K^{1/2}\norme{L^2_{\xi,q}L^1_{\tau}}{\widetilde{\rho_K}(\tau-\omega)\crochet{\tau-\omega}^{-1/2}\cdot\rho_K(\tau-\omega)f_M}
\end{multline*}
Next, using Cauchy-Schwarz inequality in the $\tau$ variable, we control the previous term with
\[\sum_{K\supeg 1} K^{1/2}\norme{L^2_{\xi,q}}{\norme{L^{2}_{\tau}}{\widetilde{\rho_K}(\tau-\omega)\crochet{\tau-\omega}^{-1/2}}\norme{L^2_{\tau}}{\rho_K(\tau-\omega)f_M}}\]
Now, since for any fixed $(\xi,q)\in\R\times\Zl$, ${\displaystyle \norme{L^{2}_{\tau}}{\widetilde{\rho_K}(\tau-\omega)\crochet{\tau-\omega}^{-1/2}}\lesssim 1}$, the sum above is finally estimated by
\[\sum_{K\supeg 1} K^{1/2}\norme{L^2_{\xi,q,\tau}}{\rho_K(\tau-\omega)f_M}=\norme{\X{M}^{0}}{f_M}\]
\end{proof}
Now we prove the proposition.
\begin{proof}
Let us begin by proving (\ref{propriete XM}). Using that ${\displaystyle \normL{2}{\chi_{K_0}(\tau-\omega)}\lesssim K_0^{1/2}}$, we estimate the term on the left-hand side by
\begin{multline*}
K_0^{1/2}\beta_{M,K_0}^{b_1}\norme{L^2_{\xi,q}}{\norme{L^{2}_{\tau}}{\chi_{K_0}(\tau-\omega)}\norme{L^{\infty}_{\tau}}{\left(K_0^{-1}\e^{i\tau' t_0}\widehat{\gamma}(K_0^{-1}\tau')\right)\star_{\tau}f_M}}\\
\lesssim \beta_{M,K_0}^{b_1}\norme{L^2_{\xi,q}}{\norme{L^{\infty}_{\tau}}{\left(\e^{i\tau' t_0}\widehat{\gamma}(K_0^{-1}\tau')\right)\star_{\tau}f_M}}
\end{multline*}
(\ref{propriete XM}) then follows from using Young's inequality $L^{\infty}\times L^{1}\rightarrow L^{\infty}$ and (\ref{estimation controle norme L2L1}), since $\widehat{\gamma}\in L^{\infty}$ by the assumption (\ref{hypothese localisation en temps X}).

Now we prove (\ref{estimation cle localisation XM}). We decompose $f_M$ according to its modulations and then distinguish two cases depending on the relation between $K$ and $K_1$ :
\begin{multline*}
\sum_{K\supeg K_0}K^{1/2}\beta_{M,K}^{b_1}\normL{2}{\rho_K(\tau-\omega)\F\left\{\gamma(K_0(t-t_0))\F^{-1}f_M\right\}} \\
\infeg \sum_{K\supeg K_0}K^{1/2}\beta_{M,K}^{b_1}\sum_{K_1\supeg 1}\normL{2}{\rho_K(\tau-\omega)\left(\e^{i\tau' t_0}\widehat{\gamma_{K_0^{-1}}}\right)\star_{\tau}\left(\rho_{K_1}(\tau-\omega) f_M\right)}
\\= \sum_{K\supeg K_0}\sum_{K_1\infeg K/10}\left(\right)+ \sum_{K\supeg K_0}\sum_{K_1\gtrsim K}\left(\right)= I+II
\end{multline*}
For the first term, we have $|\tau-\tau'|\sim K$ since $|\tau-\omega|\sim K$ and $|\tau'-\omega|\sim K_1 \infeg K/10$, thus using Young inequality $L^{\infty}\times L^1 \rightarrow L^{\infty}$, the estimate ${\displaystyle \normL{2}{\rho_K}\lesssim K^{1/2}}$ and then summing on $K\supeg K_0$, we get the bound
\begin{multline*}
I \lesssim \sum_{K\supeg K_0}K^{-1}\beta_{M,K}^{b_1}\sum_{K_1\infeg K/10}\normL{2}{\rho_K(\tau-\omega)\left(|\tau'|^{3/2}\widehat{\gamma_{K_0^{-1}}}\right)\star_{\tau}\left(\rho_{K_1}(\tau-\omega) f_M\right)}\\
\lesssim K_0^{-1/2}\sum_{K_1\infeg K/10}\normL{\infty}{|\tau'|^{3/2}\widehat{\gamma_{K_0^{-1}}}(\tau')}\norme{L^2_{\xi,q}L^1_{\tau}}{\rho_{K_1}(\tau-\omega) f_M}
\end{multline*}
This is enough for (\ref{estimation cle localisation XM}) after using (\ref{estimation controle norme L2L1}) and
\begin{equation}\label{estimation norme Lp localisation}
\normL{p}{|\cdot|^s\widehat{\gamma_{K_0^{-1}}}}\lesssim K_0^{s+1/p-1}\normL{p}{|\cdot|^s\widehat{\gamma}}
\end{equation}
and the right-hand side is finite by the assumption on gamma (\ref{hypothese localisation en temps X}).

Finally, $II$ is simply controlled using Young $L^{1}\times L^2 \rightarrow L^2$ and (\ref{estimation norme Lp localisation}) :
\[II \lesssim \sum_{K_1\gtrsim K_0} K_1^{1/2}\beta_{M,K_1}^{b_1}\normL{1}{\widehat{\gamma_{K_0^{-1}}}}\normL{2}{\rho_{K_1}(\tau-\omega)f_M}\lesssim \norme{\X{M}^{b_1}}{f_M}\]
\end{proof}
\begin{remarque}
For the loss in (\ref{propriete XM}) to be nontrivial, we need either $b_1=0$ or $K_0\lesssim (1\ou M)^3$. In particular, in the multilinear estimates we cannot localize the term with the smallest frequency on time intervals of size $M_{max}^{-1}$ when $b_1>0$.
\end{remarque}
The next proposition deals with general time multipliers as in \cite{IonescuKenigTataru2008} :
\begin{proposition}
Let $M>0$, $b_1\in[0;1/2[$, $f_M\in \fl{M}^{b_1}$ (repectively $\nl{M}^{b_1}$) and $m_M \in\mathcal{C}^4(\R)$ bounded along with its derivatives. Then
\begin{equation}\label{estimation multiplicateur F}
\norme{\fl{M}^{b_1}}{m_M(t)f_M}\lesssim \left(\sum_{k=0}^{4}(1\ou M)^{-k}\normL{\infty}{m_M^{(k)}}\right)\norme{\fl{M}^{b_1}}{f_M}
\end{equation}
and
\begin{equation}\label{estimation multiplicateur N}
\norme{\nl{M}^{b_1}}{m_M(t)f_M}\lesssim \left(\sum_{k=0}^{4}(1\ou M)^{-k}\normL{\infty}{m_M^{(k)}}\right)\norme{\nl{M}^{b_1}}{f_M}
\end{equation}
respectively, uniformly in $M>0$ and $\lambda\supeg 1$.
\end{proposition} 
\begin{proof}
Using the definition of $\fl{M}^{b_1}$, we write
\[\norme{\fl{M}^{b_1}}{m_Mf_M} = \sup_{t_M\in\R} \sum_{K\supeg 1}K^{1/2}\beta_{M,K}^{b_1}\normL{2}{p\cdot\rho_K(\tau-\omega)\F\left\{\chi_{(1\ou M)^{-1}}(t-t_M)m_M(t)f_M\right\}}\]
Next we estimate
\[\left|\F\left\{\chi_{(1\ou M)^{-1}}(t-t_M)m_M\right\}\right|(\tau) \infeg \normL{1}{\chi_{(1\ou M)^{-1}}(t-t_M)m_M}\lesssim (1\ou M)^{-1}\normL{\infty}{m_M}\]
and
\begin{multline*}
\left|\F\left\{\chi_{(1\ou M)^{-1}}(t-t_M)m_M\right\}\right|(\tau) = |\tau|^{-4} \left|\F\frac{\mathrm{d}^4}{\dt^4}\left\{\chi_{(1\ou M)^{-1}}(t-t_M)m_M\right\}\right|\\
\lesssim |\tau|^{-4}\sum_{k=0}^4\normL{\infty}{m_M^{(k)}}(1\ou M)^{3-k}\normL{1}{\chi^{(4-k)}} 
\end{multline*}
Thus we obtain
\begin{multline*}
\left|\F\left\{\chi_{(1\ou M)^{-1}}(t-t_M)m_M\right\}\right|(\tau)\\\lesssim \left(\sum_{k=0}^4 (1\ou M)^{-k}\normL{\infty}{m_M^{(k)}}\right)(1\ou M)^{-1}\crochet{(1\ou M)^{-1}\tau}^{-4}
\end{multline*}
Using (\ref{propriete XM}) and (\ref{estimation cle localisation XM}) with $t_0=t_M$, $K_0 = (1\ou M)$ and $\gamma(t)=\F^{-1}\crochet{\tau}^{-4}$ concludes the proof of (\ref{estimation multiplicateur F}). The proof of (\ref{estimation multiplicateur N}) follows similarly.
\end{proof}

The last proposition justifies the use of $\Fl(T)$ as a resolution space :
\begin{proposition}
Let $\alpha\in \N^*$, $T\in ]0;1]$, $b_1\in [0;1/2[$ and $u\in \Fl^{\alpha,b_1}(T)$. Then
\begin{equation}\label{estimation injection continue}
\norme{L^{\infty}_T\El^{\alpha}}{u}\lesssim \norme{\Fl^{\alpha,b_1}(T)}{u}
\end{equation}
and
\begin{equation}\label{estimation injection continue difference}
\norme{L^{\infty}_TL^2_{xy}}{u}\lesssim \norme{\overline{\Fl}^{b_1}(T)}{u}
\end{equation}
\end{proposition}
\begin{proof}
The proof is the same as in \cite[Lemma 3.1]{IonescuKenigTataru2008} : let $M\in 2^{\Z}$, $\widetilde{u_M}$ be an extension of $P_Mu$ to $\R$ with ${\displaystyle \norme{\fl{M}^{b_1}}{\widetilde{u_M}}\infeg 2\norme{\fl{M}^{b_1}(T)}{P_Mu}}$ and $t_M\in [-T;T]$, then it suffices to prove that
\[\norme{L^2_{\xi,q}}{p\cdot\F_{xy}\widetilde{u_M}(t_M)}\lesssim  \norme{\X{M}^{b_1}}{p\cdot \F\left\{\chi_{(1\ou M)^{-1}}(t-t_M)\widetilde{u_M}\right\}}\]
Using the properties of $\chi$ and the inversion formula, we can write
\[\widetilde{u_M}(t_M) = \left\{\chi_{(1\ou M)^{-1}}(\cdot-t_M)\widetilde{u_M}\right\}(t_M) = \int_{\R}\F_t\left\{\chi_{(1\ou M)^{-1}}(t-t_M)\widetilde{u_M}\right\}(\tau)\e^{it_M\tau}\dtau\]
Thus, using (\ref{estimation controle norme L2L1}), we get the final bound
\begin{multline*}
\norme{L^2_{\xi,q}}{p\cdot\F_{xy}\widetilde{u_M}(t_M)} \infeg \norme{L^2_{\xi,q}L^1_{\tau}}{p\cdot\F\left\{\chi_{(1\ou M)^{-1}}(t-t_M)\widetilde{u_M}\right\}}\\ \lesssim \norme{\X{M}^{b_1}}{p\cdot \F\left\{\chi_{(1\ou M)^{-1}}(t-t_M)\widetilde{u_M}\right\}}
\end{multline*}
\end{proof}
\section{Linear estimates}\label{section estimation lineaire}
This section deals with (\ref{estimation lineaire}) and (\ref{estimation lineaire difference}).
\begin{proposition}
Let $T>0$, $b_1\in[0;1/2[$ and $u\in \Bl^{\alpha}(T)$, $f\in \Nl^{\alpha,b_1}(T)$ satisfying
\begin{equation}\label{equation lineaire inhomogene}
\drt u + \drx^3 u -\drx^{-1}\dry^2 u = f
\end{equation}
on $[-T,T]\times\R\times\Tl$.\\
Then $u\in\Fl^{\alpha,b_1}(T)$ and
\begin{equation}\label{equation estimation linéaire}
\norme{\Fl^{\alpha,b_1}(T)}{u}\lesssim \norme{\Bl^{\alpha}(T)}{u}+\norme{\Nl^{\alpha,b_1}(T)}{f}
\end{equation}
\end{proposition}
\begin{proof}
This proposition is proved in \cite{IonescuKenigTataru2008} (see also \cite{KenigPilod}). We recall the proof here for completeness.

Looking at the definition of $\Fl^{\alpha,b_1}(T)$ (\ref{definition F}), $\Nl^{\alpha,b_1}(T)$ (\ref{definition N}) and $\Bl^{\alpha}(T)$ (\ref{definition espace energie}), we have to prove that $\forall M>0$,
\begin{eqnarray*}
\norme{\fl{M}^{b_1}(T)}{P_Mu} \lesssim \norme{\El^0}{P_Mu_0} + \norme{\nl{M}^{b_1}(T)}{P_Mf}\text{ if } 0<M\infeg 1\\
\norme{\fl{M}^{b_1}(T)}{P_Mu} \lesssim \sup_{t_M\in[-T,T]}\norme{\El^0}{P_Mu(t_M)} + \norme{\nl{M}^{b_1}(T)}{P_Mf}\text{ if } M>1
\end{eqnarray*}
Let $M>0$. As in \cite[Proposition 2.9, p.14]{KenigPilod}, we begin by constructing extensions $\widetilde{u}_M$ (respectively $\widetilde{f}_M$) of $P_Mu$ (respectively $P_Mf$) to $\R$, with a control on the boundary terms.\\
To do so, we first define the smooth cutoff function
\[m_M(t) := \begin{cases}
\chi_{(1\ou M)^{-1}/10}(t+T) \text{ if }t<-T\\
1 \text{ if }t\in [-T,T]\\
\chi_{(1\ou M)^{-1}/10}(t-T) \text{ if }t>T
\end{cases}\]
Next, we define $\widetilde{f}_M$ on $\R$ with
\begin{equation}\label{definition fM estimation lineaire}
\widetilde{f_M}(t) := m_M(t)f_M(t)
\end{equation}
where $f_M$ is an extension of $P_Mf$ to $\R$ satisfying ${\displaystyle \norme{\nl{M}^{b_1}}{f_M}\infeg 2\norme{\nl{M}^{b_1}(T)}{P_Mf}}$.

So $\widetilde{f_M}$ is also an extension of $P_Mf$, with $\supp \widetilde{f_M} \subset [-T-(1\ou M)^{-1}/5,T + (1\ou M)^{-1}/5]$.\\
From (\ref{equation lineaire inhomogene}), we have that
\begin{equation}\label{equation Duhamel estimation lineaire}
P_Mu(t)=U(t)P_Mu_0+\int_0^t U(t-t')P_Mf(t')\dt' \text{ on }[-T,T]
\end{equation}
Thus we define $\widetilde{u_M}$ as
\begin{equation}\label{definition uM estimation lineaire}
\widetilde{u_M}(t) :=m_M(t)\left\{U(t)P_Mu_0+\int_0^t U(t-t')\widetilde{f_M}(t')\dt'\right\},~t\in\R
\end{equation}
The choice of $\widetilde{f_M}$ and $\widetilde{u_M}$ is dictated from the necessity to control the boundary term. First using (\ref{estimation multiplicateur N}) with $m_M$ we have 
\[\norme{\nl{M}^{b_1}}{\widetilde{f_M}}\lesssim \norme{\nl{M}^{b_1}(T)}{P_Mf}\]
and $\widetilde{u_M}$ defines an extension of $P_Mu$.\\
Moreover, if $t_M\notin [-T,T]$, from the choice of $m_M$, we can write $\chi_{(1\ou M)^{-1}}(t-t_M)\widetilde{u_M}(t) = \chi_{(1\ou M)^{-1}}(t-\widetilde{t_M})\chi_{(1\ou M)^{-1}}(t-t_M)\widetilde{u_M}(t)$ for a $\widetilde{t_M}\in [-T,T]$. Then, using (\ref{propriete XM}) and (\ref{estimation cle localisation XM})  we get
\[\sup_{t_M\notin [-T,T]}\norme{\X{M}^{b_1}}{\chi_{(1\ou M)^{-1}}(t-t_M)\widetilde{u_M}}\lesssim \sup_{\widetilde{t_M}\in[-T,T]}\norme{\X{M}^{b_1}}{\chi_{(1\ou M)^{-1}}(t-\widetilde{t_M})\widetilde{u_M}}\]
Thus it suffices to prove
\begin{multline*}
\sup_{t_M\in[-T,T]}\norme{\X{M}^{b_1}}{p\cdot\F\left\{\chi(t-t_M)\widetilde{u_M}\right\}} \lesssim \norme{\El^0}{\widetilde{u_M}(0)}+\\ \sup_{\widetilde{t_M}\in \R}\norme{\X{M}^{b_1}}{(\tau-\omega+i)^{-1}p\cdot\F\left\{\chi(t-\widetilde{t_M})\widetilde{f_M}\right\}}\text{ if }M\infeg 1
\end{multline*}
and
\begin{multline*}
\sup_{t_M\in[-T,T]}\norme{\X{M}^{b_1}}{p\cdot\F\left\{\chi_{M^{-1}}(t-t_M)\widetilde{u_M}\right\}} \lesssim \sup_{\widehat{t_M}\in[-T,T]}\norme{\El^0}{\widetilde{u_M}(\widehat{t_M})}+\\ \sup_{\widetilde{t_M}\in \R}\norme{\X{M}^{b_1}}{(\tau-\omega+iM)^{-1}p\cdot\F\left\{\chi_{M^{-1}}(t-\widetilde{t_M})\widetilde{f_M}\right\}}\text{ if }M>1
\end{multline*}
Note that, since $m_M\equiv 1$ on $[-T,T]$ and $u$ is a solution of (\ref{equation lineaire inhomogene}), for $t_M\in[-T,T]$, we have
\[P_Mu(t_M) = U(t_M)P_Mu_0+\int_0^{t_M}U(t-t')\widetilde{f_M}(t')\dt'\]
and thus
\[\widetilde{u_M}(t+t_M) = m_M(t+t_M)\left\{U(t)P_Mu(t_M) + \int_{0}^{t}U(t-t')\widetilde{f_M}(t'+t_M)\dt'\right\}\]
Finally, it suffices to prove that
\begin{equation}\label{estimation lineaire homogene bf}
\norme{\X{M}^{b_1}}{p\cdot\F\left\{\chi(t-t_M)m_M(t)U(t)P_Mu_0\right\}} \lesssim \norme{\El^0}{\widetilde{u_M}(0)}
\end{equation}
and
\begin{multline}\label{estimation lineaire inhomogene bf}
\norme{\X{M}^{b_1}}{p\cdot\F\left\{\chi(t-t_M)m_M(t)\int_0^tU(t-t')\widetilde{f_M}(t')\dt'\right\}}\\
\lesssim \norme{\X{M}^{b_1}}{(\tau-\omega+i)^{-1}p\cdot\F\left\{\chi(t-t_M)\widetilde{f_M}\right\}}
\end{multline}
for the low-frequency part, and
\begin{equation}\label{estimation lineaire homogene hf}
\norme{\X{M}^{b_1}}{p\cdot\F\left\{\chi_{M^{-1}}(t)m_M(t+t_M)U(t)P_Mu(t_M)\right\}} \lesssim \norme{\El^0}{\widetilde{u_M}(t_M)}
\end{equation}
and
\begin{multline}\label{estimation lineaire inhomogene hf}
\norme{\X{M}^{b_1}}{p\cdot\F\left\{\chi_{M^{-1}}(t)m_M(t+t_M)\int_0^tU(t-t')\widetilde{f_M}(t_M+t')\dt'\right\}}\\ \lesssim \norme{\X{M}^{b_1}}{(\tau-\omega+iM)^{-1}p\cdot\F\left\{\chi_{M^{-1}}(t-t_M)\widetilde{f_M}\right\}}
\end{multline}
for the high-frequency part.\\
To prove those estimates, we first notice that, since $t'\in [0;t]$ and $t\in \supp\chi_{(1\ou M)^{-1}}$, we can write $\widetilde{f_M}$ as
\[\widetilde{f_M}(t_M+t') = \sum_{|n|\infeg 100}f_{M,n}(t_M+t'):=\sum_{|n|\infeg 100}\gamma\left((1\ou M)t'-n\right)\widetilde{f_M}(t_M+t')\]
where $\gamma : \R\rightarrow[0;1]$ is a smooth partition of unity, satisfying $\supp\gamma\subset [-1;1]$ and for all $x\in\R$,
\[\sum_{n\in\Z}\gamma(x-n)=1\]
The second observation is that, for a fixed $t_M$, we have for the homogeneous term
\[\norme{\X{M}^{b_1}}{p\cdot\F\left\{\chi_{M^{-1}}(t)m_M(t+t_M)U(t)P_Mu(t_M)\right\}} \lesssim \norme{\fl{M}^{b_1}}{m_MU(t)P_Mu(t_M)}  \]
so we can remove the localization $m_M(t)$ thanks to (\ref{estimation multiplicateur F}), and similarly for the inhomogeneous term.\\
Computing the Fourier transform in the left-hand side of (\ref{estimation lineaire homogene bf}) and using the bound
\[\norme{L^2_{\tau}}{\rho_K(\tau-\omega)\e^{it_M(\tau-\omega)}\widehat{\chi}(\tau-\omega)} \lesssim \normL{2}{\rho_K (\tau) \crochet{\tau}^{-2}}\lesssim K^{-3/2}\]
since $\widehat{\chi}\in\S(\R)$, we then obtain
\begin{multline*}
\norme{\X{M}^{b_1}}{p\cdot\F\left\{\chi(t-t_M)U(t)P_Mu_0\right\}} \\
\lesssim \sum_{K\supeg 1} K^{1/2}\beta_{M,K}^{b_1}\normL{2}{\rho_K(\tau-\omega)p\cdot\e^{it_M(\tau-\omega)}\widehat{\chi}(\tau-\omega)\widehat{P_Mu_0}}\\
\lesssim \normL{2}{P_Mu_0}
\end{multline*}
The proof of (\ref{estimation lineaire homogene hf}) is the same replacing the first bound by
\[\norme{L^2_{\tau}}{\rho_K(\tau-\omega) M^{-1}\crochet{M^{-1}(\tau-\omega)}^{-2}}\lesssim M^{-1}K^{1/2}(1\ou M^{-1}K)^{-2}\]
For (\ref{estimation lineaire inhomogene bf}) and (\ref{estimation lineaire inhomogene hf}), a computation gives first
\begin{multline*}
\F\left\{\chi_{(1\ou M)^{-1}}(t)\int_0^tU(t-t')f_{M,n}(t_M+t')\dt'\right\}(\tau) \\
= (1\ou M)^{-1}\int_{\R}\frac{\widehat{\chi}((1\ou M)^{-1}(\tau-\tau'))-\widehat{\chi}((1\ou M)^{-1}(\tau-\omega))}{i(\tau'-\omega)}\\ \cdot\e^{it_M\tau'}\F\left\{f_{M,n}\right\}(\tau')\dtau'
\end{multline*}
Now, we distinguish between two cases, wether $|\tau'-\omega+i(1\ou M) |\sim |\tau'-\omega|$ or $|\tau'-\omega+i(1\ou M) |\sim (1\ou M)$.\\
First, if $|\tau'-\omega|\gtrsim (1\ou M)$, we have
\begin{multline*}
\left|\frac{\widehat{\chi}((1\ou M)^{-1}(\tau-\tau'))-\widehat{\chi}((1\ou M)^{-1}(\tau-\omega))}{i(\tau'-\omega)}\right|\\ \lesssim \left|\frac{\widehat{\chi}((1\ou M)^{-1}(\tau-\tau'))}{i(\tau'-\omega+i(1\ou M))}\right|+\left|\frac{\widehat{\chi}((1\ou M)^{-1}(\tau-\omega))}{i(\tau'-\omega+i(1\ou M))}\right|
\end{multline*}
Now if $|\tau'-\omega|\lesssim (1\ou M)$ we apply the mean value theorem to $\widehat{\chi}$ so that
\[\widehat{\chi}((1\ou M)^{-1}(\tau-\tau'))-\widehat{\chi}((1\ou M)^{-1}(\tau-\omega)) = (1\ou M)^{-1}\widehat{\chi}'(\theta)\cdot (\tau'-\omega)\]
for a $\theta \in [\tau-\tau';\tau-\omega]$. Thus we have
\begin{multline*}
\left|\frac{\widehat{\chi}((1\ou M)^{-1}(\tau-\tau'))-\widehat{\chi}((1\ou M)^{-1}(\tau-\omega))}{i(\tau'-\omega)}\right| \lesssim (1\ou M)^{-1}\left|\widehat{\chi}'(\theta)\right|\\ \lesssim |\tau'-\omega+i(1\ou M)|^{-1}|\widehat{\chi}'(\theta)|
\end{multline*}
Finally, using the assumption on $\theta$ and that $\widehat{\chi}\in \S(\R)$, we have in both cases
\begin{multline*}
\left|\frac{\widehat{\chi}((1\ou M)^{-1}(\tau-\tau'))-\widehat{\chi}((1\ou M)^{-1}(\tau-\omega))}{i(\tau'-\omega)}\right|\\ \lesssim \left|\frac{\crochet{(1\ou M)^{-1}(\tau-\tau')}^{-4}}{\tau'-\omega+i(1\ou M)}\right| + \left|\frac{\crochet{(1\ou M)^{-1}(\tau-\omega)}^{-4}}{\tau'-\omega+i(1\ou M)}\right|
\end{multline*}
Coming back to (\ref{estimation lineaire inhomogene bf}) and (\ref{estimation lineaire inhomogene hf}), the left-hand side can be split into
\begin{multline*}
\sum_{|n|\infeg 100}\norme{\X{M}^{b_1}}{p\cdot (1\ou M)^{-1}\int_{\R}\left|\frac{\crochet{(1\ou M)^{-1}(\tau-\tau')}^{-4}}{\tau'-\omega+i(1\ou M)}\F\left\{f_{M,n}\right\}(\tau')\right|\dtau'}\\ +\sum_{|n|\infeg 100}\norme{\X{M}^{b_1}}{p\cdot (1\ou M)^{-1}\int_{\R}\left|\frac{\crochet{(1\ou M)^{-1}(\tau-\omega)}^{-4}}{\tau'-\omega+i(1\ou M)}\F\left\{f_{M,n}\right\}(\tau')\right|\dtau'}
\end{multline*}
The first term is handled with (\ref{propriete XM}) and (\ref{estimation cle localisation XM}) with $K_0 = (1\ou M)$ and $\gamma = \F^{-1}\left\{\crochet{\cdot}^{-4}\right\}$. This term is thus controled by
\[\sup_{|n|\infeg 100}\norme{\X{M}^{b_1}}{p\cdot(\tau-\omega+i(1\ou M))^{-1}\F\left\{f_{M,n}\right\}}\lesssim \norme{\fl{M}^{b_1}}{\widetilde{f_M}}\]
where in the last step we have used that 
\[\gamma((1\ou M)t-n)= \gamma((1\ou M)t-n)\chi_{(1\ou M)^{-1}}(t-(1\ou M)^{-1}n)\]
and (\ref{propriete XM})-(\ref{estimation cle localisation XM}) to get rid of $\gamma$.

It remains to treat the second term. By definition of the $\X{M}^{b_1}$ norm, we can write it
\begin{multline*}
\sum_{K\supeg 1} K^{1/2}\beta_{M,K}^{b_1}\left|\left|\rho_K(\tau-\omega)p\cdot (1\ou M)^{-1}\right.\right.\\
\cdot\left.\left.\int_{\R}\left|\frac{\crochet{(1\ou M)^{-1}(\tau-\omega)}^{-4}}{\tau'-\omega+i(1\ou M)}\F\left\{f_{M,n}\right\}(\tau')\right|\dtau'\right|\right|_{L^2}\\
=\sum_{K\supeg 1}K^{1/2}\beta_{M,K}^{b_1}\left|\left|p\cdot\norme{L^1_{\tau'}}{(\tau'-\omega+i(1\ou M))^{-1}\F\left\{f_{M,n}\right\}}\right.\right. \\ \left.\left.\cdot\norme{L^2_{\tau}}{\rho_K(\tau-\omega)(1\ou M)^{-1}\crochet{(1\ou M)^{-1}(\tau-\omega)}^{-4}}\right|\right|_{L^2_{\xi,q}}
\end{multline*}
Now, since
\[\sum_{K\supeg 1}K^{1/2}\beta_{M,K}^{b_1}(1\ou M)^{-1}\crochet{(1\ou M)^{-1}K}^{-4}\normL{2}{\rho_K}\lesssim 1\]
we can use (\ref{estimation controle norme L2L1}) to bound the last term with
\begin{multline*}
\norme{L^2_{\xi,q}}{p\cdot\norme{L^1_{\tau'}}{(\tau'-\omega+i(1\ou M))^{-1}\F\left\{f_{M,n}\right\}}}\\ \lesssim \norme{\X{M}^{b_1}}{p\cdot(\tau'-\omega+i(1\ou M))^{-1}\F\left\{f_{M,n}\right\}}
\end{multline*}
which concludes the proof through the same argument than above.
\end{proof}
Proceeding in the same way at the $L^2$ level, we have also
\begin{proposition}
Let $T>0$, $b_1\in [0;1/2[$ and $u\in \overline{\Bl}^{b_1}(T)$, $f\in\overline{\Nl}^{b_1}(T)$ satisfying (\ref{equation lineaire inhomogene}) on $[-T,T]\times\R\times\Tl$. Then
\begin{equation}\label{equation estimation linéaire difference}
\norme{\overline{\Fl}^{b_1}(T)}{u}\lesssim \norme{\overline{\Bl}(T)}{u}+\norme{\overline{\Nl}^{b_1}(T)}{f}
\end{equation}
\end{proposition}
\section{Dyadic estimates}\label{section estimation dyadique}
As in the standard Bourgain method, we will need some bilinear estimates for functions localized in both their frequency and their modulatation. This section deals with estimating expressions under the form ${\displaystyle \int f_1\star f_2 \cdot f_3}$ which will be usefull to prove the main bilinear estimate (\ref{estimation bilineaire}) as well as the energy estimate (\ref{estimation d'energie}). The following lemma gives a first rough estimate :
\begin{lemme}
Let $f_i\in L^2(\R^2\times\Zl)$ be such that $\supp f_i \subset D_{\lambda,M_i,\infeg K_i}\cap \R^2\times I_i$, with $M_i \in 2^{\Z}$, $K_i\in 2^{\N}$ and $I_i\subset \Zl$, $i=1,2,3$. Then
\begin{equation}\label{estimation forme trilineaire triviale}
\int_{\R^2\times\Zl}f_1\star f_2 \cdot f_3 \lesssim M_{min}^{1/2}K_{min}^{1/2}|I|_{min}^{1/2}\prod_{i=1}^3\normL{2}{f_i}
\end{equation}
\end{lemme}
\begin{proof}
The proof is the same as in \cite[Lemma 5.1 (b)]{IonescuKenigTataru2008}. We just have to expand the convolution product in the left-hand side and then apply Cauchy-Schwarz inequality in the variable coresponding to the min : if, for example, $K_1 = K_{min}$, we have
\begin{multline*}
\int_{\R^2\times \Zl} f_1 \star f_2 \cdot f_3 = \int_{\R^2\times\Zl}\int_{\R^2\times\Zl} f_1(\tau-\tau_2,\zeta-\zeta_2)\\
\cdot f_2(\tau_2,\zeta_2)f_3(\tau,\zeta)\dtau_2\dtau\dzeta_2\dzeta
\end{multline*}
Using Cauchy-Schwarz inequality in $\tau$, the previous term is less than
\[\int_{\R\times\Zl}\int_{\R\times\Zl}\norme{L^2_{\tau}}{f_3(\zeta)}\norme{L^2_{\tau}}{\int_{\R}f_1(\tau-\tau_2,\zeta-\zeta_2)f_2(\tau_2,\zeta_2)\dtau_2}\dzeta_2\dzeta\]
Next, a use of Young's inequality $L^1\times L^2 \rightarrow L^2$ in $\tau$ gives the bound
\[\int_{\R\times\Zl}\int_{\R\times\Zl}\norme{L^2_{\tau}}{f_3(\zeta)}\norme{L^2_{\tau}}{f_2(\zeta_2)}\norme{L^1_{\tau_1}}{f_1(\zeta-\zeta_2)}\dzeta_2\dzeta\]
Finally, using again Cauchy-Schwarz inequality in $\tau_1$, the previous term is controlled with
\[\int_{\R\times\Zl}\int_{\R\times\Zl}\norme{L^2_{\tau}}{f_3(\zeta)}\norme{L^2_{\tau}}{f_2(\zeta_2)}K_1^{1/2}\norme{L^2_{\tau_1}}{f_1(\zeta-\zeta_2)}\dzeta_2\dzeta\]
We get (\ref{estimation forme trilineaire triviale}) when proceeding similarly for the integrals in $\xi$ and $q$.
\end{proof}
\subsection{Localized Strichartz estimates}
The purpose of this subsection is to improve (\ref{estimation forme trilineaire triviale}). All the estimates we need are already used in \cite{MST2011} in the context of the KP-II equation. We briefly recall the outline of the proof here for the sake of completeness.

First, we are going to use the following easy lemmas :
\begin{lemme}\label{lemme mesure ensemble avec projections et sections}
Let $\Lambda \subset \R\times \Zl$. We assume that the projection of $\Lambda$ on the $\xi$ axis is contained in an interval $I\subset \R$. Moreover, we assume that the measure of the $q$-sections of $\Lambda$ (that is the sets $\left\{q \in \Zl, (\xi_0,q)\in \Lambda\right\}$ for a fixed $\xi_0$) is uniformly (in $\xi_0$) bounded by a constant $C$. Then we have
\begin{equation}
\left|\Lambda\right| \infeg C |I|
\end{equation}
\end{lemme}
\begin{proof}
The proof is immediate : by definition
\[\left|\Lambda\right| = \int_{I}\left(\int \mathbb{1}_{\Lambda}(\xi,q)\dq\right)\dxi \infeg \int_I C\dxi = C\left|I\right|\]
\end{proof}
\begin{lemme}\label{lemme mesure ensemble avec fonction}
Let $I$, $J$ be two intervals in $\R$, and let $\varphi : I \rightarrow \R$ be a $\mathcal{C}^1$ function with $\inf_{\xi\in J}\left|\varphi'(\xi)\right|>0$. Then
\begin{equation}\label{estimation mesure cas R}
\left|\left\{x \in J,~\varphi(x) \in I\right\}\right| \infeg \frac{|I|}{\inf_{\xi \in J}|\varphi'(\xi)|}
\end{equation}
and
\begin{equation}\label{estimation mesure cas Z}
\left|\left\{q \in J\cap \Zl,~\varphi(q) \in I\right\}\right| \lesssim \crochet{\frac{|I|}{\inf_{\xi\in J}|\varphi'(\xi)|}}
\end{equation}
\end{lemme}
\begin{proof}
Let us define
\[\mathcal{J}:=\left\{x \in J,~\varphi(x) \in I\right\}\]
and
\[\mathcal{J}_{\lambda} := \left\{q \in J\cap \Zl,~\varphi(q) \in I\right\}\]
We just have to use the mean value theorem and write
\[\left|\mathcal{J}\right| = \sup_{x_1,x_2\in\mathcal{J}}|x_2-x_1| = \sup_{x_1,x_2\in \mathcal{J}}\frac{|\varphi(x_2)-\varphi(x_1)|}{|\varphi'(\theta)|}\]
for a $\theta\in [x_1;x_2]$, and (\ref{estimation mesure cas R}) follows since $\sup_{x_1,x_2\in\mathcal{J}}|\varphi(x_2)-\varphi(x_1)|\infeg |I|$ by definition of $\mathcal{J}$. The proof of (\ref{estimation mesure cas Z}) is the same, using that 
\[|\mathcal{J}_{\lambda}|\infeg \lambda^{-1}+\sup_{q_1,q_2\in \mathcal{J}_{\lambda}}|q_2-q_1|\] by definition of $\dq{}$.
\end{proof}
\begin{lemme}
Let $a\neq 0$ ,$b$, $c$ be real numbers and $I\subset \R$ a bounded interval. Then
\begin{equation}\label{estimation mesure parabole R}
\left|\left\{x \in \R,~ax^2+bx+c \in I\right\}\right|\lesssim \frac{|I|^{1/2}}{|a|^{1/2}}
\end{equation}
and
\begin{equation}\label{estimation mesure parabole Z}
\left|\left\{q \in \Zl,~aq^2+bq+c \in I\right\}\right| \lesssim \crochet{\frac{|I|^{1/2}}{|a|^{1/2}}}
\end{equation}
\end{lemme}
\begin{proof}
We begin by proving (\ref{estimation mesure parabole R}). By the linear change of variable $x \mapsto x+ b/(2a)$ it suffices to evaluate
\[
\left|\left\{y\in\R,~ay^2\in\widetilde{I}\right\}\right| \text{ with }\widetilde{I}= I+b^2/(4a)-c,~|\widetilde{I}|=|I|
\]
Writing $\varepsilon:=\sign(a)$, the measure of the previous set is
\[
\int_{\R}\mathbb{1}_{\widetilde{I}}(ay^2)\dy = |a|^{-1/2}\int_{\R}\mathbb{1}_{\varepsilon \widetilde{I}}(x^2)\dx
\]
\begin{itemize}
\item If $0\notin \varepsilon \widetilde{I}$, by symmetry we may assume $\varepsilon\widetilde{I}\subset \R_+^*$ and write $\varepsilon \widetilde{I} = [x_1;x_2]$ with $0<x_1<x_2$. Then an easy computation gives
\begin{multline*}
\left|\left\{y\in\R,~ay^2\in\widetilde{I}\right\}\right| = |a|^{-1/2}\int_{\R}\mathbb{1}_{[x_1;x_2]}(x^2)\dx = |a|^{-1/2}\int_{\R}\mathbb{1}_{[x_1;x_2]}(y)\frac{\dy}{2\sqrt{y}}\\ = |a|^{-1/2}\left[\sqrt{y}\right]_{x_1}^{x_2} =|a|^{-1/2}(\sqrt{x_2}-\sqrt{x_1})\infeg |a|^{-1/2}|I|^{1/2}
\end{multline*}
\item If $0\in \varepsilon \widetilde{I}$ : defining $I^+ := (\varepsilon\widetilde{I}\cup -\varepsilon\widetilde{I})\cap \R^+ = [0;x_2]$ we have
\[\left|\left\{y\in\R,~ay^2\in\widetilde{I}\right\}\right| \infeg 2|a|^{-1/2} \int_{\R}\mathbb{1}_{I^+}(x^2)\dx = 2|a|^{-1/2}\sqrt{x_2} \lesssim  |a|^{-1/2}|I|^{1/2}\]
\end{itemize}
The proof of (\ref{estimation mesure parabole Z}) follows from (\ref{estimation mesure parabole R}) through the same argument as in the proof of (\ref{estimation mesure cas Z}). 
\end{proof}
\newpage
The main estimates of this section are the following.
\begin{proposition}[Dyadic $L^4-L^2$ Strichartz estimate]\label{estimation Strichartz localisee}
Let $M_1,M_2,M_3\in 2^{\Z}$, $K_1,K_2,K_3\in 2^{\N}$ and let $u_i\in L^2(\R^2\times\Zl)$, $i=1,2$, be such that $\supp (u_i)\subset D_{\lambda,M_i,\infeg K_i}$. Then
\begin{multline}\label{estimation Strichartz grossiere}
\norme{L^2}{\mathbb{1}_{D_{\lambda,M_3,\infeg K_3}}\cdot u_1\star u_2}\lesssim (K_1 \et K_2)^{1/2}M_{min}^{1/2}\\ \cdot \crochet{(K_1\ou K_2)^{1/4}(M_1\et M_2)^{1/4}}\normL{2}{u_1}\normL{2}{u_2}
\end{multline}
Moreover, if we are in the regime $K_{max}\infeg 10^{-10}M_1M_2M_3$ then
\begin{multline}\label{estimation Strichartz basse modulation}
\norme{L^2}{\mathbb{1}_{D_{\lambda,M_3,\infeg K_3}}\cdot u_1\star u_2}\lesssim (K_1 \et K_2)^{1/2}M_{min}^{1/2}\\ \cdot \crochet{(K_1\ou K_2)^{1/2}M_{max}^{-1/2}}\normL{2}{u_1}\normL{2}{u_2}
\end{multline}
\end{proposition}
\begin{proof}
These estimates are proven in \cite[Proposition 2.1 \& Corollaire 2.9]{MST2011} and \cite[Theorem 2.1, p.456-458]{SautTzvetkov2001} for functions $f_i\in L^2(\R^2\times \Z)$ but with a slightly different support condition : the localization with respect to the modulations is done for the symbol of the linear operator associated with the KP-II equation (i.e $\widetilde{\omega}(\xi,q) = \xi^3-q^2/\xi$), and the fifth-order KP-I equation ($\omega^{5\mathrm{th}}(m,\eta) = -m^5-\eta^2/m$) respectively. As a matter of fact, the proof only uses the form of the expression $(q_1/\xi_1-q_2/\xi_2)$ but does not take into account its sign within the resonnant function. Thus we can obtain the similar estimates for the KP-I equation. Let us recall the main steps in proving these estimates : first, split $u_1$ and $u_2$ depending on the value of $\xi_i$ on an $M_3$ scale 
\[
\norme{L^2}{\mathbb{1}_{D_{\lambda,M_3,\infeg K_3}}\cdot u_1\star u_2} \infeg \sum_{k\in\Z}\sum_{\ell\in\Z} \norme{L^2}{\mathbb{1}_{D_{\lambda,M_3,\infeg K_3}}\cdot u_{1,k}\star u_{2,\ell}}
\] 
with
\[
u_{i,j} := \mathbb{1}_{[jM_3,(j+1)M_3]}(\xi_i)u_i
\]
The conditions $|\xi|\sim M_3$, $\xi_1 \in [kM_3,(k+1)M_3]$ and $\xi-\xi_1 \in [\ell M_3;(\ell+1)M_3]$ require $\ell \in [-k-c;-k+c]$ for an absolute constant $c>0$. Thus we have to get estimates for functions $u_i$ supported in $D_{\lambda,M_i,K_i}\cap\left\{\xi_i \in I_i\right\}$ for some intervals $I_i$.

Moreover, we may assume $\xi_i\supeg 0$ on $\supp~u_i$ (see \cite[p.460]{SautTzvetkov2001}). This is crucial as $\xi \sim \xi_1\ou (\xi-\xi_1)$ in this case.\\

Squaring the left-hand side, it then suffices to evaluate
\begin{multline*}
\int_{\R\times\R_+\times\Zl}\left|\int_{\R\times\R_+\times\Zl}\mathbb{1}_{D_{\lambda,M_3,\infeg K_3}}\cdot u_1(\tau_1,\zeta_1)u_2(\tau-\tau_1,\zeta-\zeta_1)\dtau_1\dxi_1\dq[1]\right|^2\\ \dtau\dxi\dq{}
\end{multline*}

 Using Cauchy-Schwarz inequality, the integral above is controled by
\[\sup_{\tau,\xi\supeg 0,q\in D_{\lambda,M_3,\infeg K_3}} \left|A_{\tau,\xi,q}\right|\cdot\normL{2}{u_1}^2\normL{2}{u_2}^2\]
where $A_{\tau,\xi,q}$ is defined as
\begin{multline*}
A_{\tau,\xi,q} := \left\{(\tau_1,\zeta_1)\in\R\times\R_+\times\Zl,\xi_1\in I_1,\xi-\xi_1\in I_2,~0\infeg \xi_1 \sim M_1,\right.\\ \left.0\infeg \xi-\xi_1\sim M_2,~\crochet{\tau_1 - \omega(\zeta_1)}\lesssim K_1,~\crochet{\tau-\tau_1 - \omega(\zeta-\zeta_1)}\lesssim K_2\right\}
\end{multline*}
Using the triangle inequality in $\tau_1$, we get the bound
\[\left|A_{\tau,\xi,q}\right|\lesssim (K_1\et K_2) \left|B_{\tau,\xi,q}\right|\]
where $B_{\tau,\xi,q}$ is defined as
\begin{multline*}
B_{\tau,\xi,q} := \left\{\zeta_1\in\R_+\times\Zl,~\xi_1\in I_1,\xi-\xi_1\in I_2,~0\infeg \xi_1 \sim M_1,\right.\\ 
\left.0\infeg \xi-\xi_1\sim M_2,~\crochet{\tau-\omega(\zeta) -\Omega(\zeta_1,\zeta-\zeta_1,-\zeta)}\lesssim (K_1\ou K_2)\right\}
\end{multline*}
where $\Omega$ is the resonant function for (\ref{equation KP1}), defined on the hyperplane $\zeta_1+\zeta_2+\zeta_3=0$ :
\begin{multline}\label{definition fonction resonance}
\Omega(\zeta_1,\zeta_2,\zeta_3) := \omega(\zeta_1)+\omega(\zeta_2)+\omega(\zeta_3)\\
= -3\xi_1\xi_2\xi_3+\frac{(\xi_1q_2-\xi_2q_1)^2}{\xi_1\xi_2\xi_3}\\=-\frac{\xi_1\xi_2}{\xi_1+\xi_2}\left\{(\sqrt{3}\xi_1+\sqrt{3}\xi_2)^2-\left(\frac{q_1}{\xi_1}-\frac{q_2}{\xi_2}\right)^2\right\}
\end{multline}
Now, (\ref{estimation Strichartz grossiere}) follows directly from applying lemma~\ref{lemme mesure ensemble avec projections et sections} and (\ref{estimation mesure parabole Z}) to $B_{\tau,\xi,q}$ since its projection on the $\xi_1$ axis is controled by $|I_1|\et |I_2|$, whereas for a fixed $\xi_1$, the cardinal of the $q_1$-section is estimated by $\crochet{(K_1\ou K_2)^{1/2}(M_1\et M_2)^{1/2}}$ using (\ref{estimation mesure parabole Z}) as $\tau -\omega(\zeta)-\Omega(\zeta_1,\zeta-\zeta_1,-\zeta)$ is a polynomial of second order in $q_1$, with a dominant coefficient $\sim (M_1\et M_2)^{-1}$. Thus
\[|B_{\tau,\xi,q}| \lesssim (|I_1|\et|I_2|)\crochet{(K_1\ou K_2)^{1/2}(M_1\et M_2)^{1/2}}\]
which gives the estimate (\ref{estimation Strichartz grossiere}) 
when applied with $I_1 = [kM_3;(k+1)M_3]\cap \mathfrak{I}_{M_1}$ and $I_2=[\ell M_3,(\ell+1)M_3]\cap \mathfrak{I}_{M_2}$ and using Cauchy-Schwarz inequality to sum over $k\in\Z$.\\

In the case $K_{max}\infeg 10^{-10}M_1M_2M_3$, we compute
\[\left|\frac{\partial \Omega}{\partial q_1}\right| =2\left|\frac{q_1}{\xi_1}-\frac{q-q_1}{\xi-\xi_1}\right|=2\left\{\frac{\xi}{\xi_1(\xi-\xi_1)}\left(\Omega+3\xi_1(\xi-\xi_1)\xi\right)\right\}^{1/2}\]
Thus, from the condition $|\Omega|\lesssim K_{max}\infeg 10^{-10}M_1M_2M_3$ we get 
\[\left|\frac{\partial \Omega}{\partial q_1}\right|\gtrsim \left|\frac{\xi}{\xi_1(\xi-\xi_1)}\cdot \xi_1(\xi-\xi_1)\xi\right|^{1/2}\sim M_{max}\]
At last, we can estimate $|B_{\tau,\xi,q}|$ in this regime by using (\ref{estimation mesure cas Z}) instead of (\ref{estimation mesure parabole Z}), which gives the final bound
\[|B_{\tau,\xi,q}| \lesssim (|I_1|\et|I_2|)\crochet{(K_1\ou K_2)M_{max}^{-1}}\]
and (\ref{estimation Strichartz basse modulation}) follows through the same argument as for (\ref{estimation Strichartz grossiere}).
\end{proof}
\begin{remarque}
The estimate (\ref{estimation Strichartz grossiere}) is rather crude, yet sufficient for our purpose. (\ref{estimation Strichartz basse modulation}) is better than (\ref{estimation trilineaire dyadique basse modulation}) below in the regime $K_{max}\lesssim M_1M_2M_3$, $M_{min}\infeg 1$. Thus we do not need to use some function spaces with a special low-frequency structure as in \cite{IonescuKenigTataru2008} to deal with the difference equation, therefore we get a stronger uniqueness criterion. Note that we can perform the same argument in $\R^2$. 
\end{remarque}
\subsection{Dyadic bilinear estimates}
We are now looking to improve (\ref{estimation Strichartz basse modulation}) in the case $M_{min}\supeg 1$. We mainly follow \cite[Lemme 3.1]{Zhang2015}. However, in our situation the frequency for the $x$ variable lives in $\R$ and not in $\Z$, and thus the worst case of \cite[Lemma 3.1]{Zhang2015} (when $K_{med} \lesssim M_{max}M_{min}$) is avoided. So, using that this frequency is allowed to vary in very small intervals, we are able to recover the same result as in \cite[Lemme 5.1(a)]{IonescuKenigTataru2008}. Again, we will crucially use lemmas~\ref{lemme mesure ensemble avec projections et sections} and \ref{lemme mesure ensemble avec fonction}.
\begin{proposition}\label{proposition estimation bilineaire dyadique basse modulation}
Let $M_i,K_i\in 2^{\N}$ and $f_i : \R^2\times\Zl\rightarrow \R^+$, $i=1,2,3$, be such that $f_i\in L^2(\R^2\times\Zl)$ with ${\displaystyle \supp f_i \subset D_{\lambda,M_i,\infeg K_i}}$.\\
If $K_{max} \infeg 10^{-10}M_1M_2M_3$ and $K_{med}\gtrsim M_{max}$, then
\begin{equation}\label{estimation trilineaire dyadique basse modulation}
\int_{\R^2\times\Zl}f_1\star f_2 \cdot f_3 \dtau\dxi\dq\lesssim \left(\frac{K_1K_2K_3}{M_1M_2M_3}\right)^{1/2}\normL{2}{f_1}\normL{2}{f_2}\normL{2}{f_3}
\end{equation}
\end{proposition}
\begin{proof}
We begin as in \cite[Lemma 5.1(a)]{IonescuKenigTataru2008}. Defining
\[\I(f_1,f_2,f_3) :=\int_{\R^2\times\Zl}f_1\star f_2 \cdot f_3 \dtau\dxi\dq\]
we observe that
\begin{equation}\label{equation symétrie I}
\I(f_1,f_2,f_3) = \I(\widetilde{f_1},f_3,f_2) = \I(\widetilde{f_2},f_3,f_1)
\end{equation} where we define $\widetilde{f}(x) := f(-x)$. Thus, as ${\displaystyle \normL{2}{\widetilde{f}}=\normL{2}{f}}$, up to replacing $f_i$ by $\widetilde{f_i}$, we may assume $K_1\infeg K_2\infeg K_3$.\\
Moreover, since the expression is symmetrical in $f_1,f_2$ we can assume $M_2\infeg M_1$.\\
We first write
\begin{multline*}
\I(f_1,f_2,f_3)=\int_{\R^2\times\Zl} f_1\star f_2\cdot f_3 \dtau\dxi\dq\\= \int_{\R^2\times\Zl}\int_{\R\times\R^+\times\Zl} f_1(\tau_1,\zeta_1)f_2(\tau_2,\zeta_2)f_3(\tau_1+\tau_2,\zeta_1+\zeta_2)\\
\dtau_1\dtau_2\dzeta_1\dzeta_2
\end{multline*}
Defining ${\displaystyle f_i^{\#}(\theta,\zeta):=f_i(\theta+\omega(\zeta),\zeta)}$ we get ${\displaystyle \normL{2}{f_i^{\#}}=\normL{2}{f_i}}$ and $\supp f_i^{\#} \subset \{|\theta|\lesssim K_i,|\xi|\sim M_i\}$. Changing of variables, we have
\begin{multline*}
\I(f_1,f_2,f_3) = \int_{\R^2\times\Zl}\int_{\R^2\times\Zl} f_1^{\#}(\theta_1,\zeta_1)f_2^{\#}(\theta_2,\zeta_2)\\ \cdot f_3^{\#}(\theta_1+\theta_2+\Omega(\zeta_1,\zeta_2,-\zeta_1-\zeta_2),\zeta_1+\zeta_2)\dtheta_1\dtheta_2\dzeta_1\dzeta_2
\end{multline*}
where the resonant function
\[\Omega(\zeta_1,\zeta_2,-\zeta_1-\zeta_2)=-\frac{\xi_1\xi_2}{\xi_1+\xi_2}\left\{\sqrt{3}|\xi_1+\xi_2|+\left|\frac{q_1}{\xi_1}-\frac{q_2}{\xi_2}\right|\right\}\left\{\sqrt{3}|\xi_1+\xi_2|-\left|\frac{q_1}{\xi_1}-\frac{q_2}{\xi_2}\right|\right\}\] has been defined in (\ref{definition fonction resonance}) in the previous proposition.

Thus
\begin{multline*}
\I(f_1,f_2,f_3) = \int_A f_1^{\#}(\theta_1,\zeta_1)f_2^{\#}(\theta_2,\zeta_2)\\
\cdot f_3^{\#}(\theta_1+\theta_2+\Omega(\zeta_1,\zeta_2,-\zeta_1-\zeta_2),\zeta_1+\zeta_2)\dtheta_1\dtheta_2\dzeta_1\dzeta_2
\end{multline*}
with
\begin{align*}
A := \left\{(\theta_1,\zeta_1,\theta_2,\zeta_2)\in (\R^2\times \Zl)^2,~|\xi_i|\sim M_i,~|\xi_1+\xi_2|\sim M_3,~|\theta_i|\lesssim K_i,\right.\\ \left.~|\theta_1+\theta_2+\Omega(\zeta_1,\zeta_2,-\zeta_1-\zeta_2)|\lesssim K_3,~~i=1,2\right\}
\end{align*}
We can decompose $A\subset I_{\infeg K_1}\times I_{\infeg K_2}\times B$ with $B$ defined as
\begin{multline}\label{definition B}
B := \left\{(\zeta_1,\zeta_2)\in(\R\times\Zl)^2,~|\xi_i|\sim M_i,i=1,2,~|\xi_1+\xi_2|\sim M_3,\right.\\ \left.~|\Omega|\lesssim K_3\right\}
\end{multline}
We can further split
\[B = \bigsqcup_{|\ell|\lesssim K_3/K_2} B_{\ell}\]
with
\begin{equation}\label{definition Bgamma}
B_{\ell} := \left\{(\zeta_1,\zeta_2)\in B,~\Omega\in [\ell K_2;(\ell+1)K_2]\right\}
\end{equation}
and as well for $f_3$ :
\begin{equation}\label{definition fgamma}
f_3^{\#} = \sum_{|\ell|\lesssim K_3/K_2} f_{3,\ell}^{\#} \text{ with } f_{3,\ell}^{\#}(\theta,\xi,q) := \mathbb{1}_{[\ell K_2,(\ell+1)K_2]}(\theta) f_3^{\#}(\theta,\xi,q)
\end{equation}
Next, using Cauchy-Schwarz inequality in $\theta_2$ then $\theta_1$, we obtain
\begin{multline*}
\I(f_1,f_2,f_3) \infeg \sum_{|\ell|\lesssim K_3/K_2}\int_{I_{\infeg K_1}\times B_{\ell}} |f_1^{\#}(\theta_1,\xi_1,q_1)|\norme{L^2_{\theta_2}}{f_2^{\#}(\theta_2,\xi_2,q_2)} \\ \cdot\norme{L^2_{\theta_2}}{f_{3,\ell}^{\#}(\theta_1+\theta_2+\Omega,\xi_1+\xi_2,q_1+q_2)}\dtheta_1\dxi_1\dxi_2\dq[1]\dq[2] \\
\lesssim K_1^{1/2}\sum_{|\ell|\lesssim K_3/K_2} \int_{B_{\ell}} \norme{L^2_{\theta_1}}{f_1^{\#}(\theta_1,\xi_1,q_1)}\norme{L^2_{\theta_2}}{f_2^{\#}(\theta_2,\xi_2,q_2)}\\ \cdot\norme{L^2_{\theta}}{f_{3,\ell}^{\#}(\theta,\xi_1+\xi_2,q_1+q_2)}\dxi_1\dxi_2\dq[1]\dq[2]
\end{multline*}
This allows us to work with functions depending on $(\xi_1,q_1),(\xi_2,q_2)$ only, loosing just a factor $K_1^{1/2}$ in the process. The informations $|\Omega|\lesssim K_3$ and $\supp f_3 \subset I_{\infeg K_3}\times \mathfrak{I}_{M_3}\times \Zl$ have been kept in the decomposition on $\ell$ of $B$ and $f_3^{\#}$.

Finally, defining 
\[g_i(\xi_i,q_i) := \norme{L^2_{\theta_i}}{f_i^{\#}(\theta_i,\xi_i,q_i)},~i=1,2\text{ and } g_{3,\ell}(\xi,q) := \norme{L^2_{\theta}}{f_{3,\ell}^{\#}(\theta,\xi,q)}\]
and writing
\begin{equation}\label{definition J}
J_{\ell}(g_1,g_2,g_{3,\ell}) := \int_{B_{\ell}}g_1(\xi_1,q_1)g_2(\xi_2,q_2)g_{3,\ell}(\xi_1+\xi_2,q_1+q_2)\dxi_1\dxi_2\dq[1]\dq[2]
\end{equation}
it suffices to prove that
\begin{equation}\label{estimation objectif gi}
J := \sum_{\ell}J_{\ell}(g_1,g_2,g_{3,\ell})\lesssim \left(\frac{K_2K_3}{M_1M_2M_3}\right)^{1/2}\norme{L^2_{\xi_1,q_1}}{g_1}\norme{L^2_{\xi_2,q_2}}{g_2}\left\{\sum_{\ell}\norme{L^2_{\xi,q}}{g_{3,\ell}}^2\right\}^{1/2}
\end{equation}
As we are in the regime $K_{max}\lesssim M_1M_2M_3$, $\Omega$ is close to zero. Since $q_i\in\Zl$, we cannot just make a change of variables as in \cite[Lemme 5.1(a)]{IonescuKenigTataru2008}. Thus, to take into account that ${\displaystyle(\sqrt{3}\xi_1+\sqrt{3}\xi_2)^2 \sim \left(\frac{q_1}{\xi_1}-\frac{q_2}{\xi_2}\right)^2}$, we split $B_{\ell}$ depending on the values of $q_1$ and $q_2$.

First, as in \cite[Lemma 5.1(a)]{IonescuKenigTataru2008}, we can split
\[B_{\ell}:=B_{\ell}^{++}\sqcup B_{\ell}^{+-}\sqcup B_{\ell}^{-+}\sqcup B_{\ell}^{--}\]
with
\[B_{\ell}^{\varepsilon_1,\varepsilon_2}:=\left\{(\xi_1,q_1),(\xi_2,q_2) \in B_{\ell},~\sign(\xi_1+\xi_2)=\varepsilon_1,~\sign\left(\frac{q_1}{\xi_1}-\frac{q_2}{\xi_2}\right)=\varepsilon_2\right\}\]
where $\varepsilon_i\in\{\pm 1\}$.

Since the transformations $(\xi_1,q_1),(\xi_2,q_2)\mapsto (\varepsilon_1\xi_1,\varepsilon_2 q_2),(\varepsilon_1 \xi_1,\varepsilon_2 q_2)$ maps $B_{\ell}^{\varepsilon_1,\varepsilon_2}$ to $B_{\ell}^{++}$, it suffices to estimate
\[J_{\ell}^{++}(g_1,g_2,g_{3,\ell}) := \int_{B_{\ell}^{++}}g_1(\xi_1,q_1)g_2(\xi_2,q_2)g_{3,\ell}(\xi_1+\xi_2,q_1+q_2)\dxi_1\dxi_2\dq[1]\dq[2]\]
Moreover, the definition of $\Omega$ and the condition $|\Omega|\lesssim K_3$ give
\begin{equation}\label{equation localisation Q2}
\left|\sqrt{3}(\xi_1+\xi_2)-\left(\frac{q_1}{\xi_1}-\frac{q_2}{\xi_2}\right)\right|\infeg \frac{|\Omega|}{|\xi_1\xi_2|}\lesssim \frac{K_3}{M_1M_2}
\end{equation}
on $B_{\ell}^{++}$.

Now, we can define
\begin{equation}\label{definition Q1}
\mathcal{Q}_1(\xi_1,q_1,\xi_2,q_2) := \left\lfloor\frac{M_1M_2}{K_2}(q_1-\sqrt{3}\xi_1^2)/\xi_1\right\rfloor\in\Z
\end{equation}
and
\begin{equation}\label{definition Q2}
\mathcal{Q}_2(\xi_1,q_1,\xi_2,q_2) := \mathcal{Q}_1(\xi_1,q_1,\xi_2,q_2) - \left\lfloor\frac{M_1M_2}{K_2}( q_2 + \sqrt{3}\xi_2^2)/\xi_2\right\rfloor\in\Z
\end{equation}
So we can split $B_{\ell}^{++}$ according to the level sets of $\mathcal{Q}_1$ and $\mathcal{Q}_2$ :
\begin{align*}
B_{\ell}^{++} = \bigsqcup_{Q_1,Q_2\in\Z} B_{\ell,Q_1,Q_2}
\end{align*}
where $B_{\ell,Q_1,Q_2}$ is defined as
\[B_{\ell,Q_1,Q_2} := \left\{(\xi_1,q_1),(\xi_2,q_2)\in B_{\ell}^{++},~\mathcal{Q}_1(\xi_1,q_1,\xi_2,q_2) = Q_1,~\mathcal{Q}_2(\xi_1,q_1,\xi_2,q_2) = Q_2 \right\}\]
From definitions (\ref{definition Q1}) and (\ref{definition Q2}), for $(\xi_1,q_1),(\xi_2,q_2)\in B_{\ell,Q_1,Q_2}$, $Q_2$ is such that 
\begin{multline*}
Q_2 = \left\lfloor\frac{M_1M_2}{K_2}\left(\frac{q_1}{\xi_1}-\frac{q_2}{\xi_2}-\sqrt{3}(\xi_1+\xi_2)\right)\right\rfloor \\ \text{ or } Q_2 = \left\lfloor\frac{M_1M_2}{K_2}\left(\frac{q_1}{\xi_1}-\frac{q_2}{\xi_2}-\sqrt{3}(\xi_1+\xi_2)\right)\right\rfloor+1
\end{multline*}
Thus,
\begin{equation}\label{estimation zero omega}
\frac{q_1}{\xi_1}-\frac{q_2}{\xi_2}-\sqrt{3}(\xi_1+\xi_2) \in \left[\frac{K_2}{M_1M_2}(Q_2-1)~;~\frac{K_2}{M_1M_2}(Q_2+1)\right]
\end{equation}
Finally, if ${\displaystyle (\xi_1,q_1),(\xi_2,q_2)\in B_{\ell,Q_1,Q_2}}$ we obtain from (\ref{definition fonction resonance}) and (\ref{estimation zero omega}) that
\begin{multline}\label{calcul Omega Q1Q2}
\Omega(\xi_1,q_1,\xi_2,q_2)=\frac{\xi_1\xi_2}{\xi_1+\xi_2}\frac{K_2}{M_1M_2}\left(Q_2 + \nu\right)\left(\left|\frac{q_1}{\xi_1}-\frac{q_2}{\xi_2}\right|+\sqrt{3}\left|\xi_1+\xi_2\right|\right)
\\= \frac{\xi_1\xi_2K_2}{M_1M_2}\left(Q_2 + \nu\right)\left(2\sqrt{3}+\frac{K_2}{M_1M_2}\frac{Q_2 + \nu}{(\xi_1+\xi_2)}\right)
\end{multline}
with 
\[|\nu|\infeg 1\]
The choice of the parameter ${\displaystyle \frac{K_2}{M_1M_2}}$ in the definitions of $Q_i$ allows us to have ${\displaystyle \frac{q_1}{\xi_1}}$ and ${\displaystyle \frac{q_2}{\xi_2}}$ of the same order, and thus to keep an error $\nu$ of size $O(1)$ in this "change of variables". The measure of the $q_i$-sections of $B_{\ell,Q_1,Q_2}$ is then controled with ${\displaystyle \frac{K_2 M_i}{M_1M_2}\gtrsim 1}$ (as $K_2 \gtrsim M_{max}$), $i=1,2$.\\
Using (\ref{equation localisation Q2}), we get
\[|Q_2|\lesssim \frac{K_3}{K_2}\]
Moreover, by definition
\[\forall (\xi_1,q_1),(\xi_2,q_2)\in B_{\ell},~\ell = \left\lfloor\frac{\Omega(\xi_1,q_1,\xi_2,q_2)}{K_2}\right\rfloor\]
and so a key remark is that if ${\displaystyle (\xi_1,q_1),(\xi_2,q_2)\in B_{\ell,Q_1,Q_2}}$ :
\begin{equation}\label{estimation ell}
\ell = \ell(\xi_1,\xi_2,Q_2) = \left\lfloor\frac{\xi_1\xi_2}{M_1M_2}\left(Q_2+\nu\right)\left(2\sqrt{3}+\frac{K_2}{M_1M_2}\frac{Q_2 + \nu}{\xi_1+\xi_2}\right)\right\rfloor
\end{equation}
Using that $|\xi_i|\sim M_i$, $\xi_1+\xi_2\sim M_3$, $|Q_2|\lesssim K_3/K_2$ and that we assumed $K_3 \infeg 10^{-10}M_1M_2M_3$, we get that 
\[\left|\frac{K_2}{M_1M_2}\frac{Q_2 + \nu}{\xi_1+\xi_2}\right|\infeg 10^{-5}\]
which means that for any fixed $Q_1,Q_2$ there is at most 10 possible values for $\ell$ such that $B_{\ell,Q_,Q_2}$ is non empty.

Let us write $J_{\ell,Q_1,Q_2}$ the contribution of the region $B_{\ell,Q_1,Q_2}$ in the integral $J_{\ell}^{++}$. To control $J_{\ell,Q_1,Q_2}$ we first use Cauchy-Schwarz inequality in $q_1,q_2,\xi_1,\xi_2$ :
\begin{multline*}
J_{\ell,Q_1,Q_2} \lesssim \norme{L^2(B_{Q_1}^1)}{g_1}\norme{L^2(B_{Q_1,Q_2}^2)}{g_2}\\ \cdot \left\{\int_{B_{\ell,Q_1,Q_2}} g_{3,\ell}^2 (\xi_1+\xi_2,q_1+q_2)\dxi_1\dxi_2\dq[1]\dq[2]\right\}^{1/2}
\end{multline*}
where we define
\begin{multline}\label{definiton BQ1}
B_{Q_1}^1 := \left\{(\xi_1,q_1)\in \mathfrak{I}_{M_1}\times \Zl,\right.\\
\left.\sqrt{3}\xi_1^2 + Q_1 \frac{K_2}{M_1M_2}\xi_1 \infeg  q_1 < \sqrt{3}\xi_1^2 + (Q_1+1)\frac{K_2}{M_1M_2}\xi_1\right\}
\end{multline}
and
\begin{multline}\label{definition BQ2}
B_{Q_1,Q_2}^2 := \left\{(\xi_2,q_2)\in \mathfrak{I}_{M_2}\times \Zl,\right.\\ \left.-\sqrt{3}\xi_2^2 + (Q_1-Q_2)\frac{K_2}{M_1M_2}\xi_2\infeg  q_2 < -\sqrt{3}\xi_2^2 + (Q_1-Q_2+1)\frac{K_2}{M_1M_2}\xi_2 \right\}
\end{multline}

Let us start by treating the integral over $B_{\ell,Q_1,Q_2}$. 

If $(\xi_1,q_1),(\xi_2,q_2)\in B_{\ell,Q_1,Q_2}$, we can parametrize the $q_i$-sections with
\[r_1 := q_1 - \lpart{\sqrt{3}\xi_1^2 + Q_1\frac{K_2}{M_1M_2}\xi_1}\in\Zl\]
and
\[r_2 :=  q_2 - \lpart{-\sqrt{3}\xi_2^2 + (Q_1-Q_2)\frac{K_2}{M_1M_2}\xi_2}\in\Zl\]
such that ${\displaystyle 0\infeg r_i \lesssim \frac{K_2 M_i}{M_1M_2}}$.\\
As we assumed $M_2\infeg M_1$, the $q_2$-sections of $B_{\ell,Q_1,Q_2}$ are then smaller than the $q_1$-sections, and thus $0\infeg r_1+r_2 \lesssim r_1$. So if $\xi_1,\xi_2$ are fixed, we obtain :
\begin{multline*}
\int\int \mathbb{1}_{B_{\ell,Q_1,Q_2}}(\xi_1,q_1,\xi_2,q_2) g_{3,\ell}^2 (\xi_1+\xi_2,q_1+q_2)\dq[1]\dq[2] \\= \int \int \mathbb{1}_{[0;K_2/M_2]}(r_1)\mathbb{1}_{[0; K_2/M_1]}(r_2) g_{3,\ell}^2 \left(\xi_1+\xi_2,\right.\\ \left.\lpart{\sqrt{3}\xi_1^2 + Q_1\frac{K_2}{M_1M_2}\xi_1}+r_1+\lpart{-\sqrt{3}\xi_2^2 + (Q_1-Q_2)\frac{K_2}{M_1M_2}\xi_2}+r_2\right)\dr[1]\dr[2]
\\ \lesssim \frac{K_2}{M_1} \int \mathbb{1}_{[0;K_2/M_2]}(|r|) g_{3,\ell}^2\left(\xi_1+\xi_2,\right.\\ \left.\lpart{\sqrt{3}\xi_1^2 + Q_1 \frac{K_2}{M_1M_2}\xi_1-\sqrt{3}\xi_2^2 + (Q_1-Q_2)\frac{K_2}{M_1M_2}\xi_2}+r\right)\dr 
\end{multline*}
The integral over $B_{\ell,Q_1,Q_2}$ is thus controled by
\begin{multline*}
J_{\ell,Q_1,Q_2} \lesssim \left(\frac{K_2}{M_1}\right)^{1/2} \norme{L^2(B_{Q_1}^1)}{g_1}\norme{L^2(B_{Q_1,Q_2}^2)}{g_2}\left\{\int_{\R^2}\int \mathbb{1}_{[0;K_2/M_2}(|r|) g_{3,\ell}^2\left(\xi_1+\xi_2,\right.\right.\\ \left.\left.\lpart{\sqrt{3}\xi_1^2 + Q_1 \frac{K_2}{M_1M_2}\xi_1-\sqrt{3}\xi_2^2 + (Q_1-Q_2)\frac{K_2}{M_1M_2}\xi_2}+r\right)\dxi_1\dxi_2\dr\right\}^{1/2}
\end{multline*}
It remains to sum those contributions : using the previous estimate,
\begin{multline*}
J = \sum_{|\ell|\lesssim K_3/K_2} \sum_{Q_1\in\Z} \sum_{|Q_2|\lesssim K_3/K_2} J_{\ell,Q_1,Q_2}\\
 \lesssim \sum_{Q_1\in\Z} \sum_{|Q_2|\lesssim K_3/K_2} \left(\frac{K_2}{M_1}\right)^{1/2} \norme{L^2(B_{Q_1}^1)}{g_1}\norme{L^2(B_{Q_1,Q_2}^2)}{g_2}\\
 \cdot\left\{\sum_{|\ell|\lesssim K_3/K_2}\int_{\R^2}\int \mathbb{1}_{[0;K_2/M_2]}(|r|) g_{3,\ell}^2\left(\xi_1+\xi_2,\right.\right.\\ \left.\left.\lpart{\sqrt{3}\xi_1^2 + Q_1 \frac{K_2}{M_1M_2}\xi_1-\sqrt{3}\xi_2^2 + (Q_1-Q_2)\frac{K_2}{M_1M_2}\xi_2}+r\right)\dxi_1\dxi_2\dr\right\}^{1/2} 
\end{multline*}
Next, a use of Cauchy-Schwarz inequality in $Q_2$ then $Q_1$ gives
\begin{multline*}
J\lesssim \left(\frac{K_2}{M_1}\right)^{1/2} \left(\sum_{Q_1\in\Z}\norme{L^2(B_{Q_1}^1)}{g_1}^2\right)^{1/2}\left(\sum_{Q_1\in\Z}\sum_{|Q_2|\lesssim K_3/K_2}\norme{L^2(B_{Q_1,Q_2}^2)}{g_2}^2\right)^{1/2}\\ \cdot\left\{\sup_{Q_1}\sum_{|Q_2|\lesssim K_3/K_2}\sum_{|\ell|\lesssim K_3/K_2}\int_{\R^2}\int \mathbb{1}_{[0;K_2/M_2]}(|r|)g_{3,\ell}^2\left(\xi_1+\xi_2,\right.\right. \\  \left.\left.\lpart{\sqrt{3}\xi_1^2 + Q_1 \frac{K_2}{M_1M_2}\xi_1-\sqrt{3}\xi_2^2 + (Q_1-Q_2)\frac{K_2}{M_1M_2}\xi_2}+r\right)\dxi_1\dxi_2\dr\right\}^{1/2} 
\end{multline*}
Now, from the definitions of $B_{Q_1}^1$ (\ref{definiton BQ1}) and $B_{Q_1,Q_2}^2$ (\ref{definition BQ2}) :
\[\left(\sum_{Q_1\in\Z}\norme{L^2(B_{Q_1}^1)}{g_1}^2\right)^{1/2} = \norme{L^2_{\xi_1,q_1}}{g_1} = \normL{2}{f_1}\]
and
\begin{multline*}
\left(\sum_{Q_1\in\Z}\sum_{|Q_2|\lesssim K_3/K_2}\norme{L^2(B_{Q_1,Q_2}^2)}{g_2}^2\right)^{1/2} \lesssim \left(\frac{K_3}{K_2}\right)^{1/2}\left(\sup_{Q_2}\sum_{Q_1\in\Z}\norme{L^2(B_{Q_1,Q_2})}{g_2}^2\right)^{1/2}\\ = \left(\frac{K_3}{K_2}\right)^{1/2}\norme{L^2_{\xi_2,q_2}}{g_2} = \left(\frac{K_3}{K_2}\right)^{1/2}\normL{2}{f_2}
\end{multline*}
To conclude, it suffices to prove
\begin{multline}\label{estimation finale basse modulation}
\sup_{Q_1}\sum_{|Q_2|\lesssim K_3/K_2}\sum_{|\ell|\lesssim K_3/K_2}\int_{\R^2}\int \mathbb{1}_{[0;K_2/M_2]}(|r|)g_{3,\ell}^2\left(\xi_1+\xi_2,\right.\\ \left.\lpart{\sqrt{3}\xi_1^2 + Q_1 \frac{K_2}{M_1M_2}\xi_1-\sqrt{3}\xi_2^2 + (Q_1-Q_2)\frac{K_2}{M_1M_2}\xi_2}+r\right)\dxi_1\dxi_2\dr \\ \lesssim \frac{K_2}{M_2M_3}\normL{2}{f_3}
\end{multline}
Here, we can see the interest of splitting $f_3^{\#}$ over $\ell$ : the sum over $\ell$ is controled by the sum over $Q_2$ thanks to (\ref{estimation ell}), whereas a direct estimate on this sum would lose an additional factor $K_3/K_2$ (or in other words, when $\xi_1,\xi_2,Q_2$ are fixed, we do not have the contribution of the full $L^2$ norm of $f_3^{\#}$ in the $\theta$ variable, which allows us to sum those contributions without loosing an additionnal factor).

We begin the proof of (\ref{estimation finale basse modulation}) with the change of variables $\xi_1 \mapsto \xi := \xi_1 + \xi_2$ : the left-hand side now reads
\begin{multline*}
\sup_{Q_1}\sum_{|Q_2|\lesssim K_3/K_2}\sum_{|\ell|\lesssim K_3/K_2}\int_{\R^2}\int \mathbb{1}_{[0;K_2/M_2]}(|r|)g_{3,\ell}^2\left(\xi,\right.\\ \left.\left[\sqrt{3}\xi(\xi-2\xi_2) + Q_1\frac{K_2}{M_1M_2}\xi-Q_2\frac{K_2}{M_1M_2}\xi_2\right]_{\lambda}+r\right)\dxi_2\dxi\dr
\end{multline*}
Now, using (\ref{estimation ell}) and the definition of $g_{3,\ell}$, we have that for fixed $\xi,Q_1,\xi_2,Q_2,r$ :
\begin{multline*}
\sum_{|\ell|\lesssim K_3/K_2} g_{3,\ell}^2\left(\xi,\left[\sqrt{3}\xi(\xi-2\xi_2) + Q_1\frac{K_2}{M_1M_2}\xi-Q_2\frac{K_2}{M_1M_2}\xi_2\right]_{\lambda}+r\right) \\ \lesssim \int_{\R} \mathbb{1}\left(\theta \in \left[\frac{(\xi-\xi_2)\xi_2K_2}{M_1M_2}\left(Q_2-2\right)\left(2\sqrt{3}+\frac{K_2}{M_1M_2}\frac{Q_2 -2}{\xi}\right)~;\right.\right.\\\left.\left.\frac{(\xi-\xi_2)\xi_2K_2}{M_1M_2}\left(Q_2+2\right)\left(2\sqrt{3}+\frac{K_2}{M_1M_2}\frac{Q_2 + 2}{\xi}\right)\right]\right)\\ \cdot (f_3^{\#})^2\left(\theta,\xi,\left[\sqrt{3}\xi(\xi-2\xi_2) + Q_1\frac{K_2}{M_1M_2}\xi-Q_2\frac{K_2}{M_1M_2}\xi_2\right]_{\lambda}+r\right)\dtheta
\end{multline*}

Now, fixing only $\xi$, and $Q_1$, integrating in $\xi_2$ and $r$ and summing over $Q_2$, we can write the previous term as
\begin{multline*}
\sum_{|Q_2|\lesssim K_3/K_2}\int_{\mathfrak{I}_{M_2}}\int \mathbb{1}_{[0;K_2/M_2]}(|r|)\int_{\R} \mathbb{1}\left\{\theta\in I(\xi,\xi_2,Q_2)\right\}\\\cdot (f_3^{\#})^2\left(\theta,\xi,\lpart{\varphi(\xi,Q_1,\xi_2,Q_2)}+r\right)\dtheta\dr\dxi_2
\end{multline*}
where the interval $I(\xi,\xi_2,Q_2)$ is defined as
\begin{multline*}
I(\xi,\xi_2,Q_2) := \left[\frac{(\xi-\xi_2)\xi_2K_2}{M_1M_2}\left(Q_2-2\right)\left(2\sqrt{3}+\frac{K_2}{M_1M_2}\frac{Q_2 -2}{\xi}\right)~;\right. \\ \left.\frac{(\xi-\xi_2)\xi_2K_2}{M_1M_2}\left(Q_2+2\right)\left(2\sqrt{3}+\frac{K_2}{M_1M_2}\frac{Q_2 + 2}{\xi}\right)\right]
\end{multline*}
and the function $\varphi$ is defined as
\[\varphi(\xi,Q_1,\xi_2,Q_2):=\sqrt{3}\xi(\xi-2\xi_2) + Q_1\frac{K_2}{M_1M_2}\xi-Q_2\frac{K_2}{M_1M_2}\xi_2\]
In order to recover the $L^2$ norm of $f_3^{\#}$ in $q$, we decompose the previous term in
\begin{multline*}
\lambda\int_{\Zl}\sum_{|Q_2|\lesssim K_3/K_2}\int\int_{\Lambda_n(\xi,Q_1,Q_2)}\int_{\R} \mathbb{1}\left\{\theta\in I(\xi,\xi_2,Q_2)\right\} \\ \cdot(f_3^{\#})^2\left(\theta,\xi,n\right)
\dtheta(\dxi_2\dr)\dn
\end{multline*}
where the set $\Lambda_n(\xi,Q_1,Q_2)\subset \R\times \Zl$ for $n\in \Zl$ is defined as
\begin{multline*}
\Lambda_n(\xi,Q_1,Q_2) := \left\{(\xi_2,r)\in \mathfrak{I}_{M_2}\times [-\frac{K_2}{M_2};\frac{K_2}{M_2}],\right.\\
\left.\varphi(\xi,Q_1,\xi_2,Q_2) \in [n-r;n+\lambda^{-1}-r[\right\}
\end{multline*}
First, using the localizations $|\xi|\sim M_3$, $|\xi_2|\sim M_2$ and $|\xi-\xi_2|\sim M_1$ and the conditions $|Q_2|\lesssim K_3/K_2$ and $K_3 \infeg 10^{-10} M_1M_2M_3$, we have for any $\xi,\xi_2,Q_2$ :
\[I(\xi,\xi_2,Q_2) \subset \left\{|\theta|\in\left[c^{-1}K_2(Q_2-2),cK_2(Q_2+2)\right]\right\}\]
for an absolute constant $c>0$.

Thus we are left with estimating
\begin{multline}\label{equation dernier terme f3}
\lambda\int_{\Zl}\sum_{|Q_2|\lesssim K_3/K_2}\int_{\R}\mathbb{1}\left\{|\theta|\in [c^{-1}K_2(Q_2-2);cK_2(Q_2+2)]\right\}\\ \cdot\left|\Lambda_n(\xi_,Q_1,Q_2)\right|(f_3^{\#})^2(\theta,\xi,n)\dtheta\dn
\end{multline}

We trivially control the measure of the $r$-sections of $\Lambda_n$ with ${\displaystyle 2\frac{K_2}{M_2}}$. It remains to estimate the measure of the projection of $\Lambda_n$ on the $\xi_2$ axis, uniformly in $n,\xi,Q_1$ and $Q_2$. To do so, we are going to make a good use of lemma~\ref{lemme mesure ensemble avec fonction}. We are then left to compute ${\displaystyle \frac{\partial \varphi}{\partial \xi_2}}$ :
\[\frac{\partial\varphi}{\partial\xi_2} = -2\sqrt{3}\xi - Q_2\frac{K_2}{M_1M_2}\]
Now, as $|Q_2|\lesssim K_3/K_2$ and $K_3\infeg 10^{-10}M_1M_2M_3$, we obtain that
\[\left|\frac{\partial \varphi}{\partial\xi_2}\right| \sim 2\sqrt{3}|\xi|\sim M_3\]

So, applying (\ref{estimation mesure cas R}), we get that the projection of $\Lambda_n(\xi,Q_1,Q_2)$ on the $\xi_2$ axis is controled by ${\displaystyle \lambda^{-1}M_3^{-1}}$. A use of lemma~\ref{lemme mesure ensemble avec projections et sections} finally leads to
\[\left|\Lambda_n(\xi,Q_1,Q_2)\right|\lesssim \lambda^{-1}\frac{K_2}{M_2M_3}\]
uniformly in $n,\xi,Q_1,Q_2$.

Getting back to (\ref{equation dernier terme f3}), we have
\begin{multline*}
(\ref{equation dernier terme f3})\lesssim \frac{K_2}{M_2M_3}\int_{\Zl}\int_{\R} \left(\sum_{|Q_2|\lesssim K_3/K_2}\mathbb{1}\left\{\theta\in [c^{-1}K_2(Q_2-2);cK_2(Q_2+2)]\right\}\right)\\ \cdot(f_3^{\#})^2(\theta,\xi,n)\dtheta\dn \\ \lesssim \frac{K_2}{M_2M_3}\int_{\Zl}\int_{\R} \mathbb{1}\left\{|\theta|\in I_{\infeg K_3}\right\}(f_3^{\#})^2(\theta,\xi,n)\dtheta\dn
\end{multline*} 
Now, neglecting the $\theta$ localization and integrating in $\xi$, we finally get (\ref{estimation finale basse modulation}), which completes the proof of the proposition.
\end{proof}
\begin{remarque}
In the case $(x,y)\in\T^2$ (\cite[Lemma 3.1]{Zhang2015}), we can still use lemma~\ref{lemme mesure ensemble avec fonction}, but since $\xi_2\in\Z$ in that case, we have to use (\ref{estimation mesure cas Z}) instead of (\ref{estimation mesure cas R}), and thus we have the rougher estimate
\[\left|\Lambda_n\right| \lesssim \frac{K_2}{M_2}\left(1+M_3^{-1}\right)\lesssim \frac{K_2}{M_2}\]
as $M_i\supeg 1$ for localized functions on $\T^2$. This is the main obstacle to recover the same estimate as in $\R^2$ or $\R\times \T$, and the cause of the logarithmic divergence in the energy estimate.
\end{remarque}

The following corollary summarizes the estimates on ${\displaystyle \int f_1\star f_2 \cdot f_3}$ according to the relations between the $M$'s and the $K$'s :
\begin{corollaire}\label{proposition estimation bilineaire dyadique synthese}
Let $f_i\in L^2(\R^2\times\Zl)$ be positive functions with the support condition $\supp f_i \subset D_{\lambda,M_i,\infeg K_i}$, $i=1,2,3$. We assume $K_{med}\supeg M_{max} \supeg 1$.
\begin{itemize}
\item[(a)] If $K_{max}\infeg 10^{-10}M_1M_2M_3$ then
\begin{equation}\label{estimation forme trilineaire basse modulation}
\int_{\R^2\times\Zl}f_1\star f_2 \cdot f_3 \lesssim \left(M_{min}\et M_{min}^{-1}\right)^{1/2}M_{max}^{-1}\prod_{i=1}^3K_i^{1/2}\normL{2}{f_i}
\end{equation}
\item[(b)] If $K_{max}\gtrsim M_1M_2M_3$ and $(M_i,K_i)=(M_{min},K_{max})$ for an $i\in\{1,2,3\}$ then
\begin{equation}\label{estimation forme trilineaire grande modulation  cas pire}
\int_{\R^2\times\Zl}f_1\star f_2 \cdot f_3 \lesssim (1\et M_{min})^{1/4}M_{max}^{-1}\prod_{i=1}^3K_i^{1/2}\normL{2}{f_i}
\end{equation}
\item[(c)] If $K_{max}\gtrsim M_1M_2M_3$ but $(M_i,K_i)\neq (M_{min},K_{max})$ for any $i=1,2,3$ then
\begin{equation}\label{estimation forme trilineaire grande modulation}
\int_{\R^2\times\Zl}f_1\star f_2 \cdot f_3 \lesssim (1\ou M_{min})^{1/4}M_{max}^{-5/4}\prod_{i=1}^3K_i^{1/2}\normL{2}{f_i}
\end{equation}
\end{itemize}
\end{corollaire}
\begin{proof}
Using the symmetry property (\ref{equation symétrie I}), we can assume $K_3=K_{max}$. Note that, since $M_{max}\supeg 1$ and in order for the integral to be non zero, we must have $(1\ou M_{min})\lesssim M_{med}\sim M_{max}$. Then we treat the different cases. \\
\textbf{Case (a) :} This has already been proven in the previous proposition in the case $M_{min}\supeg 1$.\\
If $M_{min}\infeg 1$, (\ref{estimation forme trilineaire basse modulation}) follows from (\ref{estimation Strichartz basse modulation}), since $K_3=K_{max}\supeg (K_1\ou K_2)\supeg M_{max}$.\\
\textbf{Case (b) :} $M_3 = M_{min}$. Then, if $M_3\supeg 1$, (\ref{estimation forme trilineaire grande modulation  cas pire}) follows from (\ref{estimation Strichartz grossiere}) since 
\[\crochet{(K_1\ou K_2)^{1/4}(M_1\et M_2)^{1/4}} \lesssim (K_1\ou K_2)^{1/2}\] as $(K_1\ou K_2)\gtrsim M_{max}$.\\
If $M_3\infeg 1$, since this is symmetrical in $f_1$ and $f_2$ we may assume that $K_1=K_1\et K_2$. Then we apply (\ref{estimation Strichartz grossiere}) with $f_1$ and $f_3$ to get (\ref{estimation forme trilineaire grande modulation  cas pire}) since $K_3^{-1/4}\lesssim M_{min}^{-1/4}M_{max}^{-1/2}$ and $K_2^{-1/2}=K_{med}^{-1/2}\lesssim M_{max}^{-1/2}$.\\
\textbf{Case (c) :} Again, (\ref{estimation forme trilineaire grande modulation}) follows from (\ref{estimation Strichartz grossiere}) since \[\crochet{(K_1\ou K_2)^{1/4}(M_1\et M_2)^{1/4}} \lesssim (K_1\ou K_2)^{1/2}M_{max}^{-1/4}(1\ou M_{min})^{1/4}\]
\end{proof}

We conclude this section by stating another estimate which takes into account the weight in the definition of the energy space :\newpage
\begin{proposition}
Let $f_i\in L^2(\R^2\times\Zl)$ be positive functions with the support condition $\supp f_i \subset D_{\lambda,M_i,K_i}$, $i=1,2$ for $M_3>0,~K_3\supeg 1$. Then
\begin{equation}\label{estimation convolution avec poids}
\normL{2}{\mathbb{1}_{D_{\lambda,M_3,K_3}} \cdot f_1\star f_2}\lesssim (1\ou M_1)M_{min}^{1/2} K_{min}^{1/2} \normL{2}{p\cdot f_1}\normL{2}{f_2}
\end{equation}
\end{proposition}
\begin{proof}
We follow \cite[Corollaire 5.3 (b)\&(c)]{IonescuKenigTataru2008} : we split the cases $M_1\lesssim 1$ or $M_1\gtrsim 1$ and we decompose $f_1$ on its $y$ frequency in order to estimate ${\displaystyle p(\xi,q)\sim 1+ \frac{|q|}{|\xi|\crochet{\xi}}}$.\\
\textbf{Case 1 :}If $M_1\supeg 1$.\\
We then have ${\displaystyle p(\xi,q)\sim 1+ \frac{|q|}{|\xi|^2}}$.
We split
\[f_1= \sum_{L\supeg M_1^2}f_1^{L} = \mathbb{1}_{I_{\infeg M_1^2}}(q)f_1 + \sum_{L>M_1^2}\mathbb{1}_{I_L}(q)f_1\]
such that
\[\normL{2}{\mathbb{1}_{D_{\lambda,M_3,K_3}} \cdot f_1\star f_2} \lesssim \sum_{L\supeg M_1^2}L^{1/2}M_{min}^{1/2}K_{min}^{1/2}\normL{2}{f_1^{L}}\normL{2}{f_2}\]
after using (\ref{estimation forme trilineaire triviale}). Now, for $L=M_1^2$ we have $L^{-1/2}p \sim M_1^{-1} (1+ M_1^{-2}|q|) \gtrsim M_1^{-1} = L^{1/2}M_1^{-2}$, and for $L>M_1^2$ we also have $L^{-1/2}p\sim L^{-1/2}(1+LM_1^{-2}) \gtrsim L^{1/2}M_1^{-2}$. Thus, using Cauchy-Schwarz inequality in $L$, we obtain
\begin{multline*}
\normL{2}{\mathbb{1}_{D_{\lambda,M_3,K_3}} \cdot f_1\star f_2} \lesssim M_{min}^{1/2}K_{min}^{1/2}\normL{2}{f_2}\sum_{L\supeg M_1^2}L^{-1/2}M_1^2\normL{2}{p \cdot f_1^L}\\
\lesssim M_1^2M_{min}^{1/2}K_{min}^{1/2}\cdot M_1^{-1}\normL{2}{p\cdot f_1}
\end{multline*}
\textbf{Case 2 :} If $M_1 \infeg 1$.\\
This time, we split the $y$ frequency for $L\supeg 1$ since for $M_1 <\lambda^{-1}$ there is just the frequency $q=0$ :
\[f_1 = \sum_{L\supeg 1}f_1^L = \mathbb{1}_{I_{\infeg 1}}(q)f_1 + \sum_{L>1}\mathbb{1}_{I_L}(q)f_1\]
For $L=1$, we have $L^{-1/2}p \gtrsim 1 = L^{1/2}$, and for $L>1$, we also have $L^{-1/2}p \gtrsim L^{1/2}M_1^{-1}\gtrsim L^{1/2}$. Thus, using again (\ref{estimation forme trilineaire triviale}) and then Cauchy-Schwarz inequality in $L$, we only get in that case
\begin{multline*}
\normL{2}{\mathbb{1}_{D_{\lambda,M_3,K_3}} \cdot f_1\star f_2} \lesssim \sum_{L\supeg 1}L^{1/2}M_{min}^{1/2}K_{min}^{1/2}\normL{2}{f_1^{L}}\normL{2}{f_2}\\
\lesssim M_{min}^{1/2}K_{min}^{1/2}\sum_{L\supeg1}L^{-1/2}\normL{2}{p \cdot f_1^{L}}\normL{2}{f_2}\\
\lesssim M_{min}^{1/2}K_{min}^{1/2}\normL{2}{p \cdot f_1}\normL{2}{f_2}
\end{multline*}
\end{proof}

\section{Bilinear estimates}\label{section estimation bilineaire}
The aim of this section is to prove (\ref{estimation bilineaire}) and (\ref{estimation bilineaire difference}). We will treat separately the interactions $Low \times High \rightarrow High$, $High \times High \rightarrow Low$ and $Low \times Low \rightarrow Low$. Those are the only possible interactions, since for functions $f_i$ localized in $|\xi_i|\sim M_i$, we have \[ \int f_1\star f_2 \cdot f_3 \neq 0 \Rightarrow M_{min} \lesssim M_{med}\sim M_{max}\]
\subsection{For the equation}
We first prove (\ref{estimation bilineaire}).
\begin{lemme}[$Low \times High \rightarrow High$]\label{lemme lhh}
Let $M_1,M_2,M_3\in 2^{\Z}$ with $(1\ou M_1) \lesssim M_2 \sim M_3$ and $b_1\in [0;1/2[$. Then for $u_{M_1}\in \nl{M_1}^{0}$ and $v_{M_2}\in\nl{M_2}^{0}$, we have
\begin{equation}\label{estimation bilineaire lhh}
\norme{\nl{M_3}^{b_1}}{P_{M_3}\drx\left(u_{M_1}\cdot v_{M_2}\right)}\lesssim M_1^{1/2}\norme{\fl{M_1}^{0}}{u_{M_1}}\norme{\fl{M_2}^{0}}{v_{M_2}}
\end{equation}
\end{lemme}
\begin{proof}
By definition, the left-hand side of (\ref{estimation bilineaire lhh}) is
\[\sup_{t_{M_3} \in \R} \norme{\X{M_3}^{b_1}}{(\tau-\omega +iM_3)^{-1}p\cdot \F\left\{\chi_{M_3^{-1}}(t-t_{M_3})P_{M_3}\drx\left(u_{M_1}\cdot v_{M_2}\right)\right\}}\]
Let $\gamma : \R\rightarrow [0;1]$ be a smooth partition of unity, satisfying $\supp\gamma\subset [-1;1]$ and \[\forall x\in\R,~ \sum_{m\in\Z}\gamma(x-m) = 1\]
Since $(1\ou M_1)\lesssim M_2\sim M_3$, we have
\begin{multline*}
\chi_{M_3^{-1}}(t-t_{M_3}) = \sum_{|m|,|n|\infeg 100}\chi_{M_3^{-1}}(t-t_{M_3})\gamma_{M_2^{-1}}(t-t_{M_3}-M_2^{-1}m)\\
\cdot\gamma_{(1\ou M_1)^{-1}}(t-t_{M_3}-M_2^{-1}m-(1\ou M_1)^{-1}n)
\end{multline*}
Since we take the supremum in $m$ and $n$, without loss of generality, we can assume $m=n=0$. Thus, if we define
\begin{multline}\label{definition f1 estimation bilineaire lhh}
f_1^{(1\ou M_1)} := \chi_{(1\ou M_1)}(\tau-\omega) \F\left(\gamma_{(1\ou M_1)^{-1}}(t-t_{M_3})u_{M_1}\right) \text{ and }\\f_1^{K_1} := \rho_{K_1}(\tau-\omega) \F\left(\gamma_{(1\ou M_1)^{-1}}(t-t_{M_3})u_{M_1}\right),\text{ if }K_1>(1\ou M_1)
\end{multline}
and as well for $v$
\begin{multline}\label{definition f2 estimation bilineaire lhh}
f_2^{M_2} := \chi_{M_2}(\tau-\omega) \F\left(\gamma_{M_2^{-1}}(t-t_{M_3})v_{M_2}\right) \text{ and }\\f_2^{K_2} := \rho_{K_2}(\tau-\omega) \F\left(\gamma_{M_2^{-1}}(t-t_{M_3})v_{M_2}\right),\text{ if }K_2>M_2
\end{multline}
by splitting the term in the left-hand side according to its modulations, we then get
\begin{multline*}
\norme{\nl{M_3}^{b_1}}{P_{M_3}\drx\left(u_{M_1}\cdot v_{M_2}\right)} \\
\lesssim  \sup_{t_{M_3} \in \R} \sum_{K_1\supeg (1\ou M_1)}\sum_{K_2\supeg M_2}\norme{\X{M_3}}{(\tau-\omega +iM_3)^{-1}p\cdot \F\left\{P_{M_3}\drx\F^{-1}\left(f_1^{K_1}\star f_2^{K_2}\right)\right\}} \\
= \sup_{t_{M_3} \in \R} \sum_{K_1\supeg (1\ou M_1)}\sum_{K_2\supeg M_2}\sum_{K_3\supeg 1}K_3^{1/2}\beta_{M_3,K_3}^{b_1}\\
\cdot\normL{2}{(\tau-\omega +iM_3)^{-1}p\cdot \rho_{K_3}(\tau-\omega)\F\left\{P_{M_3}\drx\F^{-1}\left(f_1^{K_1}\star f_2^{K_2}\right)\right\}}
\end{multline*}
Let us start with the modulations $K_3<M_3$ : the first factor in the previous norm allows us to gain a factor $(M_3\ou K_3)^{-1}$ which makes up for the derivative, thus
\begin{multline*}
\sum_{1\infeg K_3<M_3}K_3^{1/2}\normL{2}{(\tau-\omega +iM_3)^{-1}p\cdot \rho_{K_3}(\tau-\omega)\F\left\{P_{M_3}\drx\F^{-1}\left(f_1^{K_1}\star f_2^{K_2}\right)\right\}}\\
\lesssim \sum_{1\infeg K_3<M_3}K_3^{1/2}\normL{2}{ \mathbb{1}_{D_{\lambda,M_3,\infeg M_3}}\cdot p\cdot f_1^{K_1}\star f_2^{K_2}}
\end{multline*}
and using that ${\displaystyle \sum_{1\infeg K_3 < M_3}K_3^{1/2}\lesssim M_3^{1/2}}$ we get that the previous sum is controlled with
\[M_3^{1/2} \normL{2}{\mathbb{1}_{D_{\lambda,M_3,\infeg M_3}}\cdot p \cdot f_1^{K_1}\star f_2^{K_2}}\]
Proceeding as well for the modulations $K_3\supeg M_3$ and choosing a factor $K_3^{-1}$ instead of $M_3^{-1}$, we get now
\begin{multline*}
\sum_{ K_3\supeg M_3}K_3^{1/2}\beta_{M_3,K_3}^{b_1}\normL{2}{(\tau-\omega +iM_3)^{-1}p\cdot \rho_{K_3}(\tau-\omega)\F\left\{P_{M_3}\drx\F^{-1}\left(f_1^{K_1}\star f_2^{K_2}\right)\right\}}\\
\lesssim M_3\sum_{ K_3\supeg M_3}K_3^{-1/2}\beta_{M_3,K_3}^{b_1}\normL{2}{ \mathbb{1}_{D_{\lambda,M_3,\infeg K_3}}\cdot p\cdot f_1^{K_1}\star f_2^{K_2}}\\
\end{multline*}
In particular, the first sum over the modulations $K_3 < M_3$ is controlled by the first term in the second sum over the modulations $K_3 \supeg M_3$.

Finally, it suffices to show that $\forall K_i\supeg (1\ou M_i)$, $i=1,2$,
\begin{multline}\label{estimation bilineaire lhh cle}
M_3\sum_{K_3\supeg M_3}K_3^{-1/2}\beta_{M_3,K_3}^{b_1}\normL{2}{\mathbb{1}_{D_{\lambda,M_3,\infeg K_3}} \cdot p \cdot f_1^{K_1}\star f_2^{K_2}}\\ \lesssim M_1^{1/2}\left(K_1^{1/2} \normL{2}{p\cdot f_1^{K_1}}\right)\left(K_2^{1/2}\normL{2}{p\cdot f_2^{K_2}}\right)
\end{multline}
Indeed, combining all the previous estimates, summing over $K_i\supeg (1\ou M_i)$ and using the definitions of $f_i^{K_i}$ (\ref{definition f1 estimation bilineaire lhh}), (\ref{definition f2 estimation bilineaire lhh}), the left-hand side of (\ref{estimation bilineaire lhh}) is controled by
\[M_1^{1/2}\left(\sum_{K_1\supeg (1\ou M_1)}K_1^{1/2} \normL{2}{p\cdot f_1^{K_1}}\right)\left(\sum_{K_2\supeg M_2}K_2^{1/2}\normL{2}{p\cdot f_2^{K_2}}\right)\]
The first sum is
\begin{multline*}
(1\ou M_1)^{1/2}\normL{2}{\chi_{(1\ou M_1)}(\tau-\omega) \F\left(\gamma_{(1\ou M_1)^{-1}}(t-t_{M_3})u_{M_1}\right)}\\
+\sum_{K_1> (1\ou M_1)}K_1^{1/2}\normL{2}{p\cdot \rho_{K_1}(\tau-\omega) \F\left\{\gamma_{(1\ou M_1)^{-1}}(t-t_{M_3})u_{M_1}\right\}}
\end{multline*}
As $\chi\equiv 1$ on $\supp \gamma$, we have 
\[\gamma_{(1\ou M_1)^{-1}}(t-t_{M_3}) = \gamma_{(1\ou M_1)^{-1}}(t-t_{M_3})\chi_{(1\ou M_1)^{-1}}(t-t_{M_3})\]
so this term is controlled by ${\displaystyle \norme{\fl{M_1}^{0}}{u_{M_1}}}$ thanks to (\ref{propriete XM}) and (\ref{estimation cle localisation XM}) with \[f = \F\left\{\chi_{(1\ou M_1)^{-1}}(t-t_{M_3})u_{M_1}\right\}\] and $K_0=(1\ou M_1)$.

We can similarly bound the second sum by ${\displaystyle \norme{\fl{M_2}^{0}}{v_{M_2}}}$.

~~\\

For now, we have established some estimates on expressions in the form\\ ${\displaystyle \int f_1\star f_2 \cdot f_3}$. Thus we first have to express $p\cdot f_1\star f_2$ according to $(p\cdot f_1)$ and $(p\cdot f_2)$. So, using the localizations in $|\xi_i|$ and the relation between the $M_i$, we can estimate
\begin{multline}\label{estimation bilineaire lhh poids 1}
p(\xi_1+\xi_2,q_1+q_2) \sim 1+\frac{|q_1+q_2|}{(\xi_1+\xi_2)^2}\\
\lesssim 1+\frac{|q_2|}{\xi_2^2}+\frac{|\xi_1|\crochet{\xi_1}}{(\xi_1+\xi_2)^2}\cdot\frac{|q_1|}{|\xi_1|\crochet{\xi_1}}\\
\lesssim p(\xi_2,q_2) + \frac{M_1(1\ou M_1)}{M_3^2}p(\xi_1,q_1)
\end{multline}
We then treat separately the low and high frequency cases.\\
\textbf{Case 1 :} If $M_1\infeg 1$.\\
We use the previous estimate to get
\begin{multline*}
\normL{2}{\mathbb{1}_{D_{M_3,\infeg K_3}} \cdot p \cdot f_1^{K_1}\star f_2^{K_2}} \\ \lesssim \normL{2}{\mathbb{1}_{D_{\lambda,M_3,\infeg K_3}} \cdot f_1^{K_1}\star (p\cdot f_2^{K_2})}+\frac{M_1(1\ou M_1)}{M_3^2}\normL{2}{\mathbb{1}_{D_{\lambda,M_3,\infeg K_3}} \cdot (p \cdot f_1^{K_1})\star f_2^{K_2}}\\ = I + II
\end{multline*}
To treat $I$, we use (\ref{estimation convolution avec poids}) :
\[I \lesssim (1\ou M_1)M_1^{1/2}K_{min}^{1/2}\normL{2}{p\cdot f_1^{K_1}}\normL{2}{p\cdot f_2^{K_2}}\]
Using that $K_2\supeg M_2\sim M_3$, we obtain
\[I \lesssim (K_1K_2)^{1/2}M_1^{1/2}M_3^{-1/2}\normL{2}{p\cdot f_1^{K_1}}\normL{2}{p\cdot f_2^{K_2}}\]

Next, as we can exchange the roles played by $f_1$ and $f_2$ in (\ref{estimation convolution avec poids}), we can also apply this estimate to control $II$ :
\[II \lesssim \frac{M_1(1\ou M_1)}{M_3^2}(1\ou M_2)M_1^{1/2}K_{min}^{1/2}\normL{2}{p\cdot f_1^{K_1}}\normL{2}{p\cdot f_2^{K_2}}\]
Using that $M_1\infeg 1 \infeg M_3\sim M_2$, we directly get
\[II \lesssim M_1^{3/2}M_3^{-3/2}(K_1K_2)^{1/2}\normL{2}{p\cdot f_1^{K_1}}\normL{2}{p\cdot f_2^{K_2}}\]
Finally
\[\normL{2}{\mathbb{1}_{D_{M_3,\infeg K_3}} \cdot p \cdot f_1^{K_1}\star f_2^{K_2}}\lesssim M_1^{1/2}M_3^{-1/2}\cdot K_1^{1/2}\normL{2}{p\cdot f_1^{K_1}}K_2^{1/2}\normL{2}{p\cdot f_2^{K_2}}\]
so after summing
\begin{multline*}
M_3\sum_{K_3\supeg M_3}K_3^{-1/2}\beta_{M_3,K_3}^{b_1}\normL{2}{\mathbb{1}_{D_{M_3,\infeg K_3}} \cdot p \cdot f_1^{K_1}\star f_2^{K_2}}\\
\lesssim M_1^{1/2} \cdot K_1^{1/2}\normL{2}{p\cdot f_1^{K_1}}K_2^{1/2}\normL{2}{p\cdot f_2^{K_2}}
\end{multline*}
since
\[\sum_{K_3\supeg M_3}K_3^{-1/2}\beta_{M_3,K_3}^{b_1}\lesssim M_3^{-1/2}\]
This is (\ref{estimation bilineaire lhh cle}) in that case.\\
\textbf{Case 2 :} If $M_1 > 1$.\\
It is still sufficient to use (\ref{estimation bilineaire lhh poids 1}) if $K_3$ is large enough .

Indeed, let us split the sum over $K_3$ in two parts, depending on wether $K_3\supeg M_1^2M_3$ or $M_3\infeg K_3\infeg M_1^2M_3$.\\
\textbf{Case 2.1 :} If $K_3\supeg M_1^2M_3$.\\
We proceed as in the case $M_1\infeg 1$ to get
\begin{multline*}
\normL{2}{\mathbb{1}_{D_{M_3,\infeg K_3}} \cdot p \cdot f_1^{K_1}\star f_2^{K_2}} \\ \lesssim \normL{2}{\mathbb{1}_{D_{\lambda,M_3,\infeg K_3}} \cdot f_1^{K_1}\star (p\cdot f_2^{K_2})}+\frac{M_1(1\ou M_1)}{M_3^2}\normL{2}{\mathbb{1}_{D_{\lambda,M_3,\infeg K_3}} \cdot (p \cdot f_1^{K_1})\star f_2^{K_2}}\\ = I + II
\end{multline*}
As previously,
\begin{multline*}
I \lesssim M_1^{3/2}K_{min}^{1/2}\normL{2}{p\cdot f_1^{K_1}}\normL{2}{p\cdot f_2^{K_2}}\\
\lesssim M_1^{3/2}(K_1K_2)^{1/2}M_2^{-1/2}\normL{2}{p\cdot f_1^{K_1}}\normL{2}{p\cdot f_2^{K_2}}
\end{multline*}
As for $II$, we have again
\begin{multline*}
II \lesssim M_1^{5/2}M_2M_3^{-2}K_{min}^{1/2}\normL{2}{p\cdot f_1^{K_1}}\normL{2}{p\cdot f_2^{K_2}}\\
\lesssim M_1^{3/2}(K_1K_2)^{1/2}M_2^{-1/2}\normL{2}{p\cdot f_1^{K_1}}\normL{2}{p\cdot f_2^{K_2}}
\end{multline*}
It remains to sum for the modulations $K_3\supeg M_1^2M_3$ :
\begin{multline}\label{estimation lhh finale haute modulation}
M_3\sum_{K_3\supeg M_1^2M_3}K_3^{-1/2}\beta_{M_3,K_3}^{b_1}\normL{2}{\mathbb{1}_{D_{\lambda,M_3,\infeg K_3}} \cdot p \cdot f_1^{K_1}\star f_2^{K_2}}\\ \lesssim M_1^{1/2}\left(K_1^{1/2}\normL{2}{p\cdot f_1^{K_1}}\right)\left(K_2^{1/2}\normL{2}{p\cdot f_2^{K_2}}\right)
\end{multline}
since
\[\sum_{K_3\supeg M_1^2M_3}K_3^{-1/2}\beta_{M_3,K_3}^{b_1} \lesssim M_1^{-1}M_3^{-1/2}\beta_{M_3,M_1^2M_3}^{b_1}\]
and for $M_1>1$, we have $M_1^2M_3< M_3^3$ so ${\displaystyle \beta_{M_3,M_1^2M_3}=1}$.\\
\textbf{Case 2.2 :} If $M_3\infeg K_3\infeg M_1^2M_3$.\\
We improve (\ref{estimation bilineaire lhh poids 1}) using the resonant function (cf. (\ref{definition fonction resonance})). Observe that, since $\Omega(\zeta_1,\zeta_2,\zeta_3)$ and the hyperplane $\zeta_1+\zeta_2+\zeta_3=0$ are invariant under permutation, we have
\[\left|\frac{q_1+q_2}{\xi_1+\xi_2}-\frac{q_2}{\xi_2}\right|=\left|\frac{\xi_2}{\xi_1(\xi_1+\xi_2)}\Omega(-\zeta_1-\zeta_2,\zeta_2,\zeta_1)+3\xi_1^2\right|^{1/2}\]
Since $\supp f_i\subset D_{\lambda,M_i,\infeg K_i}$ and $\int f_1\star f_2\cdot f_3 \neq 0 \Rightarrow |\Omega|\lesssim K_{max}$, we deduce the bound
\begin{equation}\label{estimation bilineaire lhh poids 2}
p(\xi_1+\xi_2,q_1+q_2) \lesssim 1+\frac{|q_1+q_2|}{|\xi_1+\xi_2|^2}\lesssim
p(\xi_2,q_2) + M_1^{1/2}M_3^{-2}K_{max}^{1/2}
\end{equation}
Therefore, we have the bound
\begin{multline*}
\normL{2}{\mathbb{1}_{D_{\lambda,M_3,\infeg K_3}} \cdot p \cdot f_1^{K_1}\star f_2^{K_2}} \\ \lesssim \normL{2}{\mathbb{1}_{D_{\lambda,M_3,\infeg K_3}} \cdot f_1^{K_1}\star (p\cdot f_2^{K_2})}+M_1^{-1/2}M_3^{-1}K_{max}^{1/2}\normL{2}{\mathbb{1}_{D_{\lambda,M_3,\infeg K_3}} \cdot f_1^{K_1}\star f_2^{K_2}}
\end{multline*}
as $M_1\lesssim M_3$.\\
To treat those terms, we distinguish the cases of corollary~\ref{proposition estimation bilineaire dyadique synthese}.\\
\underline{Case 2.1 (a) :} If $K_{max}\lesssim M_1M_2M_3$. In that case we estimate $K_{max}^{1/2}$ in the second term and then apply (\ref{estimation forme trilineaire basse modulation}) to both terms to get the bound
\begin{multline*}
M_3\sum_{K_3=M_3}^{M_1^2M_3}K_3^{-1/2}\beta_{M_3,K_3}^{b_1}\normL{2}{\mathbb{1}_{D_{\lambda,M_3,\infeg K_3}} \cdot p \cdot f_1^{K_1}\star f_2^{K_2}}\\ \lesssim \ln\left(M_1\right)M_1^{-1/2}\cdot(K_1K_2)^{1/2}\normL{2}{p\cdot f_1^{K_1}}\normL{2}{p\cdot f_2^{K_2}}
\end{multline*}
\underline{Case 2.2 (b)\&(c) :} If $K_{max}\gtrsim M_1M_2M_3$. Then we lose the factor $K_{max}^{1/2}$ in the first term and use (\ref{estimation Strichartz grossiere}) for both terms with the indices corresponding to $K_{min}$ and $K_{med}$, getting the final bound
\begin{multline*}
M_3\sum_{K_3=M_3}^{M_1^2M_3}K_3^{-1/2}\beta_{M_3,K_3}^{b_1}\normL{2}{\mathbb{1}_{D_{\lambda,M_3,\infeg K_3}} \cdot p \cdot f_1^{K_1}\star f_2^{K_2}}\\ \lesssim \ln\left(M_1\right)\cdot(K_1K_2)^{1/2}\normL{2}{p\cdot f_1^{K_1}}\normL{2}{p\cdot f_2^{K_2}}
\end{multline*}
\end{proof}
\newpage
\begin{lemme}[$High\times High \rightarrow Low$]\label{lemme hhl}
Let $M_1,M_2,M_3\in 2^{\Z}$ with $M_1 \sim M_2 \gtrsim (1\ou M_3)$, and $b_1\in [0;1/2[$. Then for $u_{M_1}\in\nl{M_1}^{0}$ and $v_{M_2}\in\nl{M_2}^{0}$, we have
\begin{equation}\label{estimation bilineaire hhl}
\norme{\nl{M_3}^{b_1}}{P_{M_3}\drx\left(u_{M_1}\cdot v_{M_2}\right)}\lesssim M_2^{3/2+4b_1}(1\ou M_3)^{-1}\norme{\fl{M_1}^{0}}{u_{M_1}}\norme{\fl{M_2}^{0}}{v_{M_2}}
\end{equation}
\end{lemme}

\begin{proof}
We proceed similarly to the previous lemma, but this time the norm on the left-hand side only controls functions on time intervals of size $(1\ou M_3)^{-1}$ whereas the norms on the right-hand side require a control for time intervals of size $M_2^{-1}$. Thus will cut the time intervals in smaller pieces. 

To do so, we take $\gamma$ as in the previous lemma. Since now $M_1\sim M_2 \gtrsim (1\ou M_3)$, we can write
\begin{multline*}
\chi_{(1\ou M_3)^{-1}}(t-t_{M_3}) = \sum_{|m|\lesssim M_2(1\ou M_3)^{-1}}\sum_{|n|\lesssim 100} \chi_{(1\ou M_3)^{-1}}(t-t_{M_3})\gamma_{M_2}(t-t_{M_3}- M_2^{-1}m)\\
\cdot\gamma_{M_1}(t-t_{M_3}-M_2^{-1}m-M_1^{-1}n)
\end{multline*}
As previously, without loss of generality, we can assume $m=n=0$, and defining
\[f_{1} := \F\left\{\gamma\left(M_1(t-t_{M_3})\right)u_{M_1}\right\}\]
and
\[
f_{2} := \F\left\{\gamma\left(M_2(t-t_{M_3})\right)v_{M_2}\right\}
\]
it then suffices to prove that $\forall K_i \supeg (1\ou M_i)$ :
\begin{multline}\label{estimation bilineaire hhl cle}
M_2(1\ou M_3)^{-1}\cdot M_3\sum_{K_3\supeg (1\ou M_3)}K_3^{-1/2}\beta_{M_3,K_3}^{b_1}\normL{2}{\mathbb{1}_{D_{\lambda,M_3,\infeg K_3}} \cdot p \cdot f_{1}^{K_1}\star f_{2}^{K_2}}\\
\lesssim M_2^2 (1\ou M_3)^{-1} K_1^{1/2} \normL{2}{p\cdot f_{1}^{K_1}}K_2^{1/2}\normL{2}{p\cdot f_{2}^{K_2}}
\end{multline}
where we have denoted
\[f_{i}^{M_2} := \chi_{M_i}(\tau-\omega)f_{i} \text{ and }f_{i}^{K_i}:=\rho_{K_i}(\tau-\omega)f_{i},~K_i>M_i\]
As previously, we need to estime $p(\xi_1+\xi_2,q_1+q_2)$ with respect to $p(\xi_1,q_1)$ and $p(\xi_2,q_2)$ :
\begin{multline}\label{estimation p grossiere hhl}
p(\xi_1+\xi_2,q_1+q_2) \lesssim 1+ \frac{|q_1+q_2|}{|\xi_1+\xi_2|\crochet{\xi_1+\xi_2}}\\
\lesssim 1 + \frac{|\xi_1|\crochet{\xi_1}}{|\xi_1+\xi_2|\crochet{\xi_1+\xi_2}}\frac{|q_1|}{|\xi_1|\crochet{\xi_1}}+\frac{|\xi_2|\crochet{\xi_2}}{|\xi_1+\xi_2|\crochet{\xi_1+\xi_2}}\frac{|q_2|}{|\xi_2|\crochet{\xi_2}}\\
\lesssim M_2^2M_3^{-1}(1\ou M_3)^{-1}\left(p(\xi_1,q_1)+p(\xi_2,q_2)\right)
\end{multline}
Just as before, we distinguish several cases.\\
\textbf{Case 1.1 :} If $M_3 \infeg 1$ and $K_3\supeg M_2^5$ :\\
We use (\ref{estimation p grossiere hhl}), so that the left-hand side of (\ref{estimation bilineaire hhl cle}) is controled with
\begin{multline*}
M_2^3 \sum_{K_3\supeg M_2^5} K_3^{b_1-1/2}\left\{\normL{2}{\mathbb{1}_{D_{\lambda,M_3,\infeg K_3}} \cdot  (p \cdot f_1^{K_1})\star  f_2^{K_2}}\right.\\
\left.+\normL{2}{\mathbb{1}_{D_{\lambda,M_3,\infeg K_3}} \cdot  f_1^{K_1}\star (p \cdot f_2^{K_2})}\right\}
\end{multline*}
Using (\ref{estimation convolution avec poids}) and that $M_1\sim M_2\supeg 1$ and $K_1,K_2\gtrsim M_2$, we get the bound
\begin{multline*}
\sum_{K_3\supeg M_2^5}K_3^{b_1-1/2}M_2^3\cdot M_2 M_3^{1/2}K_{min}^{1/2} \normL{2}{p\cdot f_1^{K_1}}\normL{2}{p\cdot f_2^{K_2}}\\
\lesssim M_2^{1+5b_1}M_3^{1/2} \cdot (K_1K_2)^{1/2}\normL{2}{p\cdot f_1^{K_1}}\normL{2}{p\cdot f_2^{K_2}}
\end{multline*}
which suffices for (\ref{estimation bilineaire hhl cle}).\\
\textbf{Case 1.2 :} If $M_3\infeg 1$ and $1\infeg K_3\infeg M_2^5$ :\\
We improve the control on $p$ in this regime by using $\Omega$ as in (\ref{estimation bilineaire lhh poids 2}). We get in this case
\[\left|\frac{q_1+q_2}{\xi_1+\xi_2}-\frac{q_1}{\xi_1}\right| =\left| \frac{\xi_2}{\xi_1(\xi_1+\xi_2)}\Omega(\zeta_1,-\zeta_1-\zeta_2,\zeta_2)+3\xi_2^2\right|^{1/2}
\lesssim M_2 + M_3^{-1/2}K_{max}^{1/2}\]
from which we deduce
\begin{equation}\label{estimation p precise hhl m3<1}
p(\xi_1+\xi_2,q_1+q_2) \lesssim M_2p(\xi_1,q_1) + M_3^{-1/2}K_{max}^{1/2}
\end{equation}
Using this estimate, we get the bound
\begin{multline*}
M_3M_2^2 \sum_{K_3=1}^{M_2^5}K_3^{b_1-1/2}\left\{\normL{2}{\mathbb{1}_{D_{\lambda,M_3,\infeg K_3}} \cdot  (p \cdot f_1^{K_1})\star  f_2^{K_2}}\right.\\
\left. + M_3^{-1/2}M_2^{-1}K_{max}^{1/2}\normL{2}{\mathbb{1}_{D_{\lambda,M_3,\infeg K_3}} \cdot f_1^{K_1}\star  f_2^{K_2}}\right\}
\end{multline*}
Observe that the term within the braces is the same as in case 2.2 of lemma~\ref{lemme lhh}, so we control it the exact same way to get the final bound
\[M_3^{1/2}M_2^{1+5b_1}\cdot(K_1K_2)^{1/2}\normL{2}{p\cdot f_1^{K_1}}\normL{2}{p\cdot f_2^{K_2}}\]

\textbf{Case 2.1 :} If $M_3 \supeg 1$ and $K_3\supeg M_2^4M_3^{-1}$.\\
We use again (\ref{estimation p grossiere hhl}) so that the left-hand side of (\ref{estimation bilineaire hhl cle}) is controled with
\begin{multline*}
M_2 \sum_{K_3\supeg M_2^4M_3^{-1}} K_3^{-1/2}\beta_{M_3,K_3}^{b_1} M_2^2M_3^{-2} \left\{\normL{2}{\mathbb{1}_{D_{\lambda,M_3,\infeg K_3}} \cdot  (p \cdot f_1^{K_1})\star  f_2^{K_2}}\right.\\
\left.+\normL{2}{\mathbb{1}_{D_{\lambda,M_3,\infeg K_3}} \cdot  f_1^{K_1}\star (p \cdot f_2^{K_2})}\right\}
\end{multline*}
With (\ref{estimation convolution avec poids}) again, we obtain the bound
\begin{multline*}
\sum_{K_3\supeg M_2^4M_3^{-1}}K_3^{-1/2}\beta_{M_3,K_3}^{b_1}M_2^3M_3^{-2}\cdot M_2 M_3^{1/2}K_{min}^{1/2} \normL{2}{p\cdot f_1^{K_1}}\normL{2}{p\cdot f_2^{K_2}}\\
\lesssim M_2^{3/2+4b_1}M_3^{-1-4b_1} \cdot (K_1K_2)^{1/2}\normL{2}{p\cdot f_1^{K_1}}K_2^{1/2}\normL{2}{p\cdot f_2^{K_2}}
\end{multline*}
\textbf{Case 2.2 :} If $M_3\supeg 1$ and $M_3\infeg K_3\infeg M_2^4M_3^{-1}$.\\
(\ref{estimation p precise hhl m3<1}) becomes in this case
\begin{equation}\label{estimation p precise hhl}
p(\xi_1+\xi_2,q_1+q_2) \lesssim M_3^{-1}M_2p(\xi_1,q_1) + M_3^{-3/2}K_{max}^{1/2}
\end{equation}
So the use of (\ref{estimation p precise hhl}) allows us to bound the left-hand side of (\ref{estimation bilineaire hhl cle}) with
\begin{multline*}
M_2^2M_3^{-1} \sum_{K_3=M_3}^{M_2^4M_3^{-1}}K_3^{-1/2}\beta_{M_3,K_3}^{b_1}\left\{\normL{2}{\mathbb{1}_{D_{\lambda,M_3,\infeg K_3}} \cdot  (p \cdot f_1^{K_1})\star  f_2^{K_2}}\right.\\
\left. + M_3^{-1/2}M_2^{-1}K_{max}^{1/2}\normL{2}{\mathbb{1}_{D_{\lambda,M_3,\infeg K_3}} \cdot f_1^{K_1}\star  f_2^{K_2}}\right\}
\end{multline*}
Proceeding similarly to the previous cases, we finally obtain the bound
\[M_2^{1+4b_1}M_3^{-1-4b_1}\cdot (K_1K_2)^{1/2}\normL{2}{p\cdot f_1^{K_1}}K_2^{1/2}\normL{2}{p\cdot f_2^{K_2}}\]
\end{proof}

\begin{lemme}[$Low\times Low \rightarrow Low$]\label{lemme lll}
Let $M_1, M_2, M_3\in 2^{-\Z}$ and $b_1\in [0;1/2[$. Then for $u_{M_1}\in\fl{M_1}^{0}$ and $v_{M_2}\in\fl{M_2}^{0}$ we have
\begin{equation}\label{estimation bilineaire lll}
\norme{\nl{M_3}^{b_1}}{P_{M_3}\drx\left(u_{M_1}\cdot v_{M_2}\right)}\lesssim (M_1M_2M_3)^{1/2}\norme{\fl{M_1}^{0}}{u_{M_1}}\norme{\fl{M_2}^{0}}{v_{M_2}}
\end{equation}
\end{lemme}
\begin{proof}
As in the previous lemmas, it is enough to prove that $\forall K_1,K_2 \supeg 1$,
\begin{multline}\label{estimation bilineaire lll cle}
M_3\sum_{K_3\supeg 1}K_3^{-1/2}\beta_{M_3,K_3}^{b_1}\normL{2}{\mathbb{1}_{D_{\lambda,M_3,\infeg K_3}} \cdot p \cdot f_1^{K_1}\star f_2^{K_2}}\\
\lesssim K_1^{1/2} \normL{2}{p\cdot f_1^{K_1}}K_2^{1/2}\normL{2}{p\cdot f_2^{K_2}}
\end{multline}
By symmetry, we may assume $M_1\infeg M_2$, so similarly to (\ref{estimation p grossiere hhl}), we have in this case
\[
p(\xi_1+\xi_2,q_1+q_2)  \lesssim M_2M_3^{-1}\left(p(\xi_1,q_1)+p(\xi_2,q_2)\right)
\]
It then suffices to use (\ref{estimation convolution avec poids}) along with the previous bound to get (\ref{estimation bilineaire lll cle}) :
\begin{multline*}
M_3\sum_{K_3\supeg 1}K_3^{-1/2}\beta_{M_3,K_3}^{b_1}\normL{2}{\mathbb{1}_{D_{\lambda,M_3,\infeg K_3}} \cdot p \cdot f_1^{K_1}\star f_2^{K_2}}\\
\lesssim M_2\sum_{K_3\supeg 1}K_3^{b_1-1/2}M_{min}^{1/2}K_{min}^{1/2}\normL{2}{p\cdot f_1^{K_1}}\normL{2}{p\cdot f_2^{K_2}}\\
\lesssim (M_1M_2M_3)^{1/2}\cdot
(K_1K_2)^{1/2}\normL{2}{p\cdot f_1^{K_1}}\normL{2}{p\cdot f_2^{K_2}}
\end{multline*}
\end{proof}

\begin{proposition}\label{proposition bilineaire globale}
Let $T\in]0;1]$, $\alpha \supeg 1$ and $b_1\in [0;1/8]$. Then for $u,v\in \Fl^{\alpha,0}(T)$ we have
\begin{equation}\label{estimation bilineaire globale}
\norme{\Nl^{\alpha,b_1}(T)}{\drx(uv)}\lesssim \norme{\Fl^{\alpha,0}(T)}{u}\norme{\Fl^{1,0}(T)}{v}+\norme{\Fl^{1,0}(T)}{u}\norme{\Fl^{\alpha,0}(T)}{v}
\end{equation}
\end{proposition}
\begin{proof}
For $M_1\in 2^{\Z}$, let us choose an extension $u_{M_1}\in \fl{M_1}^{0}$ of $P_{M_1}u$ satisfying \[ \norme{\fl{M_1}^0}{u_{M_1}}\infeg 2\norme{\fl{M_1}^0(T)}{P_{M_1}u}\] and let us define $v_{M_2}$ analogously.

Using the definition of $\Fl^{\alpha,b_1}(T)$ (\ref{definition F}) and $\Nl^{\alpha,b_1}(T)$ (\ref{definition N}), it then suffices to show that
\begin{multline}\label{estimation cle estimation bilineaire globale}
\sum_{M_1,M_2,M_3}(1\ou M_3)^{2\alpha}\norme{\nl{M_3}^{b_1}}{P_{M_3}\drx (u_{M_1}\cdot v_{M_2})}^2\\
\lesssim \sum_{M_1,M_2}\left\{(1\ou M_1)^{2\alpha}\norme{\fl{M_1}^{0}}{u_{M_1}}^2(1\ou M_2)^2\norme{\fl{M_2}^{0}}{v_{M_2}}^2\right.\\+\left.(1\ou M_1)^2\norme{\fl{M_1}^{0}}{u_{M_1}}^2(1\ou M_2)^{2\alpha}\norme{\fl{M_2}^{0}}{v_{M_2}}^2\right\}
\end{multline}
Since the left-hand side of (\ref{estimation bilineaire globale}) is symmetrical in $u$ and $v$, we can assume $M_1\infeg M_2$.\\
Then we can decompose the left-hand side of (\ref{estimation cle estimation bilineaire globale}) depending on the relation between $M_1,M_2$ and $M_3$ :
\begin{multline*}
\sum_{M_1,M_2,M_3>0}(1\ou M_3)^{2\alpha}\norme{\nl{M_3}^{b_1}}{P_{M_3}\drx (u_{M_1}\cdot v_{M_2})}^2\\
 = \sum_{i=1}^3 \sum_{(M_1,M_2,M_3)\in A_i}(1\ou M_3)^{2\alpha}\norme{\nl{M_3}^{b_1}}{P_{M_3}\drx (u_{M_1}\cdot v_{M_2})}^2
\end{multline*}
where
\[
\begin{cases}
A_1 := \left\{(M_1,M_2,M_3)\in 2^{\Z}, (1\ou M_1)\lesssim M_2\sim M_3\right\}\\
A_2 := \left\{(M_1,M_2,M_3)\in 2^{\Z}, (1\ou M_3)\lesssim M_1\sim M_2\right\}\\
A_3:=\left\{(M_1,M_2,M_3)\in 2^{\Z}, M_{max}\lesssim 1\right\}
\end{cases}
\]
Using lemma~\ref{lemme lhh}, the first term is estimated by
\begin{multline*}
\sum_{(M_1,M_2,M_3)\in A_1}(1\ou M_3)^{2\alpha}\norme{\nl{M_3}^{b_1}}{P_{M_3}\drx (u_{M_1}\cdot v_{M_2})}^2\\
\lesssim \sum_{M_2 \gtrsim 1}\sum_{M_1\lesssim M_2}M_1(1\ou M_2)^{2\alpha}\norme{\fl{M_1}^{0}}{u_{M_1}}^2\norme{\fl{M_2}^{0}}{v_{M_2}}^2
\end{multline*}
which suffices for (\ref{estimation cle estimation bilineaire globale}). For the second term, the use of lemma~\ref{lemme hhl} provides the bound
\begin{multline*}
\sum_{(M_1,M_2,M_3)\in A_2}(1\ou M_3)^{2\alpha}\norme{\nl{M_3}^{b_1}}{P_{M_3}\drx (u_{M_1}\cdot v_{M_2})}^2 \\
\lesssim \sum_{ M_2 \gtrsim 1}\sum_{M_1\sim M_2}M_2^{3+8b_1+2(\alpha-1)}\norme{\fl{M_1}^{0}}{u_{M_1}}^2\norme{\fl{M_2}^{0}}{v_{M_2}}^2
\end{multline*}
which is enough for (\ref{estimation cle estimation bilineaire globale}) since $b_1\in[0;1/8]$. Finally, lemma~\ref{lemme lll} allows us to control the last term by
\begin{multline*}
\sum_{(M_1,M_2,M_3)\in A_3}(1\ou M_3)^{2\alpha}\norme{\nl{M_3}^{b_1}}{P_{M_3}\drx (u_{M_1}\cdot v_{M_2})}^2 \\
\lesssim \sum_{M_1\in 2^{-\N}}\sum_{M_2\in 2^{-\N}}M_1M_2\norme{\fl{M_1}^{0}}{u_{M_1}}^2\norme{\fl{M_2}^{0}}{v_{M_2}}^2
\end{multline*}
which concludes the proof of the bilinear estimate.
\end{proof}
\subsection{For the difference equation}
The end of this section is devoted to treating (\ref{estimation bilineaire difference}). Let $b_1\in [0;1/2[$.

We begin with the low frequency interactions :
\begin{lemme}[$Low\times Low \rightarrow Low$]\label{lemme lll difference}
Let $M_1,M_2,M_3\in 2^{-\Z}$. Then for $u_{M_1}\in \overline{\fl{M_1}^0}$ and $v_{M_2}\in\fl{M_2}^0$, we have
\[
\norme{\overline{\nl{M_3}^{b_1}}}{P_{M_3}\drx\left(u_{M_1}\cdot v_{M_2}\right)}\lesssim M_3M_{min}^{1/2}\norme{\overline{\fl{M_1}^{b_1}}}{u_{M_1}}\norme{\fl{M_2}^{b_1}}{v_{M_2}}
\]
\end{lemme} 
\begin{proof}
Proceeding as for the previous lemmas, it suffices to prove that for all $K_1,K_2\supeg 1$ and $f_i^{K_i} : D_{\lambda,M_i,\infeg K_i}\rightarrow \R_+$,
\begin{multline*}
M_3 \sum_{K_3\supeg 1}K_3^{-1/2}\beta{M_3,K_3}^{b_1} \normL{2}{\mathbb{1}_{D_{\lambda,M_3,\infeg K_3}}\cdot f_1^{K_1}\star f_2^{K_2}}\\
\lesssim M_3M_{min}^{1/2}\cdot (K_1K_2)^{1/2} \normL{2}{f_1^{K_1}}\normL{2}{p\cdot f_2^{K_2}}
\end{multline*}
This follows directly from (\ref{estimation convolution avec poids}).
\end{proof}

\begin{lemme}[$High\times High\rightarrow Low$]\label{lemme hhl difference}
Let $M_1,M_2,M_3\in 2^{\Z}$ with $M_1\sim M_2 \gtrsim (1\ou M_3)$. Then for $u_{M_1}\in \overline{\fl{M_1}^0}$ and $v_{M_2}\in\fl{M_2}^0$, we have
\[
\norme{\overline{\nl{M_3}^{b_1}}}{P_{M_3}\drx\left(u_{M_1}\cdot v_{M_2}\right)}\lesssim (1\et M_3)^{3/2}M_2\norme{\overline{\fl{M_1}^{b_1}}}{u_{M_1}}\norme{\fl{M_2}^{b_1}}{v_{M_2}}
\]
\end{lemme}
\begin{proof}
Following the proof of lemma~\ref{lemme hhl}, it is enough to prove that for all $K_i \supeg (1\ou M_i)$ and $f_i^{K_i} : D_{\lambda,M_i,\infeg K_i}\rightarrow \R_+$,
\begin{multline*}
M_3M_2(1\ou M_3)^{-1}\sum_{K_3\supeg (1\ou M_3)}K_3^{-1/2}\beta{M_3,K_3}^{b_1}\normL{2}{\mathbb{1}_{D_{\lambda,M_3,\infeg K_3}} \cdot f_{1}^{K_1}\star f_{2}^{K_2}}\\
\lesssim (1\et M_3)^{3/2}M_2 \cdot(K_1K_2)^{1/2} \normL{2}{f_{1}^{K_1}}\normL{2}{p\cdot f_{2}^{K_2}}
\end{multline*}
This is a consequence of (\ref{estimation Strichartz grossiere}).
\end{proof}
It remains to treat the interaction between low and high frequencies. Since $u$ and $v$ do not play a symetric role anymore, we have to distinguih which one has the low frequency part.
\begin{lemme}[$Low\times High \rightarrow High$]\label{lemme lhh difference}
Let $(1\ou M_1)\lesssim M_2\sim M_3$ and $u_{M_1}\in\overline{\fl{M_1}^0}$, $v_{M_2}\in\fl{M_2}^0$. Then
\[
\norme{\overline{\nl{M_3}^{b_1}}}{P_{M_3}\drx\left(u_{M_1}\cdot v_{M_2}\right)}\lesssim M_1^{1/2}(1\ou M_1)^{1/4}M_2^{1/4} \norme{\overline{\fl{M_1}^{0}}}{u_{M_1}}\norme{\fl{M_2}^{0}}{v_{M_2}}
\]
\end{lemme}
\begin{proof}
Foolowing the proof of lemma~\ref{lemme lhh}, it suffices to prove that for all $K_i \supeg (1\ou M_i)$ and $f_i^{K_i} : D_{\lambda,M_i,\infeg K_i} \rightarrow \R_+$, 
\begin{multline*}
M_3\sum_{K_3\supeg M_3}K_3^{-1/2}\beta{M_3,K_3}^{b_1}\normL{2}{\mathbb{1}_{D_{\lambda,M_3,\infeg K_3}} \cdot f_1^{K_1}\star f_2^{K_2}}\\ \lesssim M_1^{1/2}(1\ou M_1)^{1/4}M_2^{1/4}\cdot(K_1K_2)^{1/2} \normL{2}{f_{1}^{K_1}}\normL{2}{p\cdot f_{2}^{K_2}}
\end{multline*}
Again, this follows from using (\ref{estimation Strichartz grossiere}).
\end{proof}
\begin{lemme}[$High\times Low\rightarrow High$]\label{lemme hlh difference}
Let $(1\ou M_2)\lesssim M_1\sim M_3$ and $u_{M_1}\in\overline{\fl{M_1}^0}$, $v_{M_2}\in\fl{M_2}^0$. Then
\[
\norme{\overline{\nl{M_3}^0}}{P_{M_3}\drx\left(u_{M_1}\cdot v_{M_2}\right)}\lesssim (1\ou M_2)\norme{\overline{\fl{M_1}^{0}}}{u_{M_1}}\norme{\fl{M_2}^{0}}{v_{M_2}}
\]
\end{lemme}
\begin{proof}
As previously, it is enough to prove
\begin{multline*}
M_3\sum_{K_3\supeg M_3}K_3^{-1/2}\normL{2}{\mathbb{1}_{D_{\lambda,M_3,\infeg K_3}} \cdot f_1^{K_1}\star f_2^{K_2}}\\ \lesssim (1\ou M_2)\cdot(K_1K_2)^{1/2} \normL{2}{f_{1}^{K_1}}\normL{2}{p\cdot f_{2}^{K_2}}
\end{multline*}
for $K_i\supeg (1\ou M_i)$ and $f_i^{K_i} : D_{\lambda,M_i,\infeg K_i} \rightarrow \R_+$.\\
Following the proof of lemma~\ref{lemme lhh},we distinguish several cases.\\
\textbf{Case 1 :} If $M_2\infeg 1$.\\
This is a consequence of (\ref{estimation convolution avec poids}).\\
\textbf{Case 2 :} If $M_2\supeg 1$.\\
 We split the sum over $K_3$ into two parts. The high modulations part $K_3\supeg M_2M_3$ is treated again with (\ref{estimation convolution avec poids}), whereas for the sum over the modulations $M_3\infeg K_3\infeg M_2M_3$ is controled by using (\ref{estimation forme trilineaire grande modulation  cas pire}) (which is the worst case of corollary~\ref{proposition estimation bilineaire dyadique synthese}).
\end{proof}
We finally combine the previous estimates to get
\begin{proposition}
Let $T\in]0;1]$, $b_1\in [0;1/2[$ and $u\in \overline{\Fl}^{0}(T)$, $v\in \Fl^{1,0}(T)$. Then 
\begin{equation}\label{estimation bilineaire difference globale}
\norme{\overline{\Nl}^{b_1}(T)}{\drx(uv)}\lesssim \norme{\overline{\Fl}^0(T)}{u}\norme{\Fl^{1,0}(T)}{v}
\end{equation}
\end{proposition}
\begin{proof}
First, for $M_1,M_2\in 2^{\Z}$, we fix an extension $u_{M_1}\in \overline{\fl{M_1}^{0}}$ of $P_{M_1}u$ to $\R$ satisfying \[ \norme{\overline{\fl{M_1}^{0}}}{u_{M_1}} \infeg 2\norme{\overline{\fl{M_1}^{0}}(T)}{P_{M_1}u}\] and similarly for $v_{M_2}$.\\
Using the definition of $\overline{\Fl}^0$ (\ref{definition F difference}) and $\overline{\Nl}^{b_1}$ (\ref{definition N difference}), it then suffices to show that 
\begin{multline}\label{estimation cle bilineaire difference}
\sum_{M_1,M_2,M_3\in 2^{\Z}}\norme{\overline{\nl{M_3}^{b_1}}(T)}{\drx (u_{M_1}\cdot v_{M_2})}^2\\
\lesssim \sum_{M_1,M_2\in 2^{\Z}}\norme{\overline{\fl{M_1}^{0}}(T)}{u_{M_1}}^2(1\ou M_2)^2\norme{\fl{M_2}^{0}(T)}{v_{M_2}}^2
\end{multline}
As in the proof of proposition~\ref{proposition bilineaire globale}, we separate 4 cases, so it suffices to show that for $i\in \{1,2,3,4\}$,
\begin{multline*}
\sum_{(M_1,M_2,M_3)\in B_i}\norme{\overline{\nl{M_3}^{b_1}}(T)}{\drx (u_{M_1}\cdot v_{M_2})}^2\\
\lesssim \sum_{M_1,M_2\in 2^{\Z}}\norme{\overline{\fl{M_1}^{0}}(T)}{u_{M_1}}^2(1\ou M_2)^2\norme{\fl{M_2}^{0}(T)}{v_{M_2}}^2
\end{multline*}
with
\[
\begin{cases}
B_1:=\left\{(M_1,M_2,M_3)\in 2^{-\Z}\right\}\\
B_2:=\left\{(M_1,M_2,M_3)\in 2^{\Z},~M_1\sim M_2\gtrsim (1\ou M_3)\right\}\\
B_3:=\left\{(M_1,M_2,M_3)\in 2^{\Z},~M_2\sim M_3\gtrsim (1\ou M_1)\right\}\\
B_4:=\left\{(M_1,M_2,M_3)\in 2^{\Z},~M_1\sim M_3\gtrsim (1\ou M_2)\right\}
\end{cases}
\]
This follows from lemmas~\ref{lemme lll difference},~\ref{lemme hhl difference},~\ref{lemme lhh difference}, and~\ref{lemme hlh difference} respectively.
\end{proof}
\section{Energy estimates}\label{section estimation energie}
In this section we prove the energy estimates~(\ref{estimation d'energie}) and (\ref{estimation d'energie difference}). As the nonlinear term is expressed as a bilinear form, we will need some control on trilinear form to deal with the energy estimate :
\begin{lemme}\label{lemme estimation trilineaire}
Let $T\in [0;1[$, $M_1,M_2,M_3\in 2^{\Z}$ with $M_{max}\supeg 1$, and $b_1\in[0;1/8]$. Then for $u_i\in \overline{\fl{M_i}^{b_1}}(T)$, $i\in\{1,2,3\}$, with one of them in $\fl{M_i}^{b_1}(T)$ (in order for the integral to converge), we have
\begin{equation}\label{equation energie forme trilineaire}
\left|\int_{[0,T]\times\R\times\Tl}u_1u_2u_3 \dt\dx\dy \right|\lesssim \Lambda_{b_1}(M_{min},M_{max})\prod_{i=1}^3 \norme{\overline{\fl{M_i}^{b_1}}(T)}{u_i}
\end{equation}
where
\begin{equation}\label{definition Lambda estimation trilineaire}
\Lambda_{b_1}(X,Y)=\left(X\et X^{-1}\right)^{1/2}+\left(\frac{(1\ou X)}{Y}\right)^{2b_1}
\end{equation}
\end{lemme}
\begin{proof}
Using the symmetry property (\ref{equation symétrie I}), we may assume $M_1\infeg M_2\infeg M_3$. We begin by fixing some extensions $u_{M_i}\in \overline{\fl{M_i}^{b_1}}$ of $u_i$ to $\R$ satisfying ${\displaystyle \norme{\overline{\fl{M_i}^{b_1}}}{u_{M_i}}\infeg 2 \norme{\overline{\fl{M_i}^{b_1}}(T)}{u_i}}$.\\
Let $\gamma:\R\rightarrow [0;1]$ be a smooth partition of unity as in the proof of lemma~\ref{lemme lhh}, satisfying now $\supp \gamma \subset [-1;1]$ and 
\begin{equation}\label{definition gamma estimation energie}
\forall t\in\R,~\sum_{n\in\Z}\gamma^3(t-n) = 1
\end{equation}
We then use $\gamma$ to chop the time interval in pieces of size $M_3^{-1}$ : 
\begin{equation}\label{estimation energie decoupage}
\int_{[0;T]\times\R\times\Tl}u_1u_2u_3\dt\dx\dy  \lesssim  \sum_{n\in\Z} \int_{\R^2\times\Zl}f_{1,n}\star f_{2,n}\cdot f_{3,n}\dtau\dxi\dq
\end{equation}
where we define
\[
f_{i,n} := \F\left(\gamma(M_3t-n)\mathbb{1}_{[0,T]} u_{M_i}\right)
\]
We can divide the set of integers such that the trilinear form is not zero into two subsets
\begin{multline*}
\mathcal{A}:=\left\{n\in\Z,~\gamma(M_3t-n)\mathbb{1}_{[0,T]}=\gamma(M_3t-n)\right\}\\ \text{ and }\mathcal{B}=\left\{n\in\Z\prive\mathcal{A},~\int f_{1,n}\star f_{2,n}\cdot f_{3,n} \neq 0\right\}
\end{multline*}
Let us notice that $\#\mathcal{A}\lesssim M_3$ and $\#\mathcal{B}\infeg 4$. 

Let us start by dealing with the sum over ${\displaystyle \mathcal{A}}$ :
\[\sum_{n\in\mathcal{A}}\int_{\R^2\times\Zl} f_{1,n}\star f_{2,n}\cdot f_{3,n}\lesssim M_3\sup_{n\in\mathcal{A}}\sum_{K_1,K_2,K_3\supeg M_3}\int_{\R^2\times\Zl} f_{1,n}^{K_1}\star f_{2,n}^{K_2}\cdot f_{3,n}^{K_3}\]
where $f_{i,n}^{K_i}$ is defined as
\begin{equation}\label{definition fi estimation energie}
f_{i,n}^{K_i} (\tau,\xi,q) := \rho_{K_i}(\tau-\omega(\xi,q))f_{i,n}(\tau,\xi,q),~i=1,2,3 \text{ if }K_i>M_3
\end{equation}
and
\[
f_{i,n}^{M_3} (\tau,\xi,q) := \chi_{M_3}(\tau-\omega(\xi,q))f_{i,n}(\tau,\xi,q),~i=1,2,3
\]
Then, we separate the sum into three parts depending on the relations between the $M$'s and the $K$'s as in corollary~\ref{proposition estimation bilineaire dyadique synthese} :
\[\sum_{K_1,K_2,K_3}\int_{\R^2\times\Zl} f_{1,n}^{K_1}\star f_{2,n}^{K_2}\cdot f_{3,n}^{K_3}= \sum_{i=1}^3\sum_{(K_1,K_2,K_3)\in A_i}\int_{\R^2\times\Zl} f_{1,n}^{K_1}\star f_{2,n}^{K_2}\cdot f_{3,n}^{K_3}\]
with
\[\begin{cases}
A_1 := \left\{(K_1,K_2,K_3)\in 2^{\N},~K_i\supeg M_3,~K_{max}\infeg 10^{-10}M_1M_2M_3\right\}\\
A_2:=\left\{(K_1,K_2,K_3)\in 2^{\N},~K_i\supeg M_3,~K_1 = K_{max}\gtrsim M_1M_2M_3\right\}\\
A_3:=\left\{(K_1,K_2,K_3)\in 2^{\N},~K_i\supeg M_3,~K_{max}=(K_2\ou K_3)\gtrsim M_1M_2M_3\right\}
\end{cases}\]
We treat those terms separately, using the estimates of corollary~\ref{proposition estimation bilineaire dyadique synthese}. Denoting $J_i$ the contributioin of the region $A_i$ in the sum, we have
\begin{multline*}J_1 \lesssim M_3\sup_{n\in\mathcal{A}}\sum_{K_1,K_2,K_3\supeg M_3} \left(M_{min}\et M_{min}^{-1}\right)^{1/2}M_{max}^{-1} \prod_{i=1}^3K_i^{1/2}\normL{2}{f_{i,n}^{K_i}}\\
 \lesssim \left(M_{min}\et M_{min}^{-1}\right)^{1/2}\prod_{i=1}^3\norme{\overline{\fl{M_i}^0}}{u_{M_i}}
 \end{multline*}
after using (\ref{estimation forme trilineaire basse modulation}) and 
\begin{equation}\label{estimation controle terme interieure estimation energie}
\sup_{n\in\mathcal{A}}\sum_{K_i\supeg M_3}K_i^{1/2}\normL{2}{f_{i,n}^{K_i}}\lesssim \norme{\overline{\fl{M_i}^0}}{u_{M_i}}
\end{equation}
Indeed, (\ref{estimation controle terme interieure estimation energie}) follows from the definition of ${\displaystyle f_{i,n}^{K_i}}$ (\ref{definition fi estimation energie}), the fact that $\chi_{(1\ou M_i)^{-1}}\equiv 1$ on the support of $\gamma_{M_3^{-1}}$, and the use of (\ref{propriete XM}) and (\ref{estimation cle localisation XM}).

Proceeding analogously, we get
\begin{multline*}
J_3 \lesssim M_3 \sup_{n\in\mathcal{A}}\sum_{K_1,K_2,K_3\supeg M_3} M_{max}^{-1}\left(\frac{(1\ou M_{min})}{M_{max}}\right)^{1/4}\prod_{i=1}^3K_i^{1/2}\normL{2}{f_{i,n}^{K_i}}\\
\lesssim \left(\frac{(1\ou M_{min})}{M_{max}}\right)^{2b_1}\prod_{i=1}^3\norme{\overline{\fl{M_i}^0}}{u_i}
\end{multline*}
by using (\ref{estimation forme trilineaire grande modulation}) and that $b_1\in [0;1/8]$.

Finally, the last contribution is controled thanks to (\ref{estimation forme trilineaire grande modulation  cas pire}), (\ref{estimation controle terme interieure estimation energie}) and the weight $\beta_{M_1,K_1}^{b_1}$ :
\begin{multline*}
J_2 \lesssim\sup_{n\in\mathcal{A}}\sum_{K_1\gtrsim M_1M_2M_3}\sum_{K_2,K_3\supeg M_3}(1\et M_{min})^{1/4}\prod_{i=1}^3 K_i^{1/2}\normL{2}{f_{i,n}^{K_i}}\\
\lesssim (1\et M_{min})^{1/4}\left(\frac{(1\ou M_{min})^{3}}{M_{min}M_{max}^2}\right)^{b_1}\prod_{i=1}^2\norme{\overline{\fl{M_i}^0}}{u_{M_i}}\\
\cdot\left(\sup_{n\in\mathcal{A}}\sum_{K_1\gtrsim M_1M_2M_3} \beta_{M_1,K_1}^{b_1}K_1^{1/2}\normL{2}{f_{1,n}^{K_1}}\right)
\end{multline*}
This suffices for (\ref{equation energie forme trilineaire}) since
\[\sup_{n\in\mathcal{A}}\sum_{K_1\gtrsim M_1M_2M_3} \beta_{M_1,K_1}^{b_1}K_1^{1/2}\normL{2}{f_{1,n}^{K_1}} \lesssim \norme{\overline{\fl{M_1}^{b_1}}}{u_{M_1}}\] as we only need to use (\ref{estimation cle localisation XM}) in this regime.

Let us now come back to (\ref{estimation energie decoupage}). It remains to treat the border terms. We have
\[
\sum_{n\in\mathcal{B}}\int_{\R^2\times\Zl}f_{1,n}\star f_{2,n}\cdot f_{3,n} \infeg \sum_{n\in\mathcal{B}}\sum_{K_1,K_2,K_3} \int_{\R^2\times\Zl}g_{1,n}^{K_1}\star g_{2,n}^{K_2}\cdot g_{3,n}^{K_3}
\]
where $g_{i,n}^{K_i}$ is defined as
\begin{equation}\label{definition g lemme estimation energie}
g_{i,n}^{K_i} := \rho_{K_i}(\tau-\omega)\F\left(\gamma(M_3t-n)\mathbb{1}_{[0,T]}u_{M_i}\right),~i=1,2,3,~K_i \supeg 1
\end{equation}
Once again, we separate the different cases of corollary~\ref{proposition estimation bilineaire dyadique synthese}. Let us define $G_i$ the contribution of the region $A_i$ in the sum above.

Using (\ref{estimation forme trilineaire basse modulation}), we can control the first term :
\begin{equation}\label{estimation G1 estimation energie}
G_1 \lesssim \left(M_{min}\et M_{min}^{-1}\right)^{1/2}M_{max}^{-1}\sup_{n\in\mathcal{B}} \sum_{(K_1,K_2,K_3)\in A_1} \prod_{i=1}^3K_i^{1/2}\normL{2}{g_{i,n}^{K_i}}
\end{equation}
Now, we need to replace (\ref{estimation controle terme interieure estimation energie}) by an analogous estimate on $\mathcal{B}$ :
\begin{equation}\label{estimation controle terme de bord estimation energie}
\sup_{n\in\mathcal{B}}\sup_{K_i\supeg M_3}K_i^{1/2}\normL{2}{g_{i,n}^{K_i}}\lesssim \norme{\overline{\fl{M_i}^0}}{u_{M_i}}
\end{equation}
Let us prove this estimate. Using the definition of ${\displaystyle g_{i,n}^{K_i}}$ (\ref{definition g lemme estimation energie}), if we note $\widetilde{u_{M_i}}:=\gamma(M_3t-n)u_{M_i}$ then we have to estimate
\[\normL{2}{g_{i,n}^{K_i}}=\normL{2}{\rho_{K_i}(\tau-\omega)\cdot\widehat{\mathbb{1}_{[0,T]}}\star\F\left(\widetilde{u_{M_i}}\right)}\]
We then split ${\displaystyle \F\left(\widetilde{u_{M_i}}\right)}$ depending on its modulations :
\begin{multline*}
\normL{2}{g_{i,n}^{K_i}}
\infeg \sum_{K\infeg K_i/10}\normL{2}{\rho_{K_i}(\tau-\omega)\cdot\widehat{\mathbb{1}_{[0,T]}}\star_{\tau}\left(\rho_{K}(\tau'-\omega)\F\left(\widetilde{u_{M_i}}\right)\right)} \\
+\sum_{K\supeg K_i/10}\normL{2}{\rho_{K_i}(\tau-\omega)\F_t\left\{\mathbb{1}_{[0,T]}\F_t^{-1}\left(\rho_{K}(\tau'-\omega)\F\left(\widetilde{u_{M_i}}\right)\right)\right\}}\\
= I + II
\end{multline*}
To treat $I$, we use that ${\displaystyle \left|\widehat{\mathbb{1}_{[0,T]}}(\tau-\tau')\right| \infeg |\tau-\tau'|^{-1}\sim K_i^{-1}}$ since $|\tau-\omega|\sim K_i$ and $|\tau'-\omega|\sim K \infeg K_i/10$. Thus, from Young inequality $L^{\infty}\times L^1 \rightarrow L^{\infty}$ we deduce that
\[K_i^{1/2}\cdot I\lesssim K_i\norme{L^2_{\xi,q}L^{\infty}_{\tau}}{\widehat{\mathbb{1}_{[0,T]}}\star_{\tau}\left(\rho_{K}(\tau'-\omega)\F\left(\widetilde{u_{M_i}}\right)\right)} \lesssim \norme{L^2_{\xi,q}L^1_{\tau}}{\rho_{K}(\tau'-\omega)\F\left(\widetilde{u_{M_i}}\right)}\]
which is enough for (\ref{estimation controle terme de bord estimation energie}) due to (\ref{estimation controle norme L2L1}) and then (\ref{propriete XM})-(\ref{estimation cle localisation XM}).

To deal with $II$, we simply neglect the localization $\rho_{K_i}(\tau-\omega)$, use Plancherel identity, then neglect the localization $\mathbb{1}_{[0;T]}$ and use Plancherel identity again and that $K_i^{1/2}\lesssim K^{1/2}$ to get
\[K_i^{1/2}\cdot II\lesssim \sum_{K\supeg K_i/10}K^{1/2}\norme{L^2_{\xi,q,\tau}}{\rho_{K}(\tau'-\omega)\F\left(\widetilde{u_{M_i}}\right)}\lesssim \norme{\X{M_i}^0}{\F\left(\widetilde{u_{M_i}}\right)}\]
This proves (\ref{estimation controle terme de bord estimation energie}) after using again (\ref{propriete XM})-(\ref{estimation cle localisation XM}).

Coming back to (\ref{estimation G1 estimation energie}) and using (\ref{estimation controle terme de bord estimation energie}) along with $\#\mathcal{B}\infeg 4$, we then infer
\[G_1 \lesssim \crochet{\ln\left(M_1M_2M_3\right)}^3(M_{min}\et M_{min}^{-1})^{1/2}M_{max}^{-1}\prod_{i=1}^3\norme{\overline{\fl{M_i}^0}}{u_i}\]
as ${\displaystyle \sum_{(K_1,K_2,K_3)\in A_1}1 \lesssim \crochet{\ln\left(M_1M_2\right)}^3}$. This is enough for (\ref{equation energie forme trilineaire}).

Let us now turn to $G_2$. We use (\ref{estimation forme trilineaire grande modulation cas pire}) combined with (\ref{estimation controle terme de bord estimation energie}) to get
\begin{multline*}
G_2 \lesssim (1\et M_{min})^{1/4}M_{min}^{0+}M_{max}^{(-1)+}\sum_{(K_1,K_2,K_3)\in A_2}K_{max}^{0-} \prod_{i=1}^3\sup_{n\in\mathcal{B}}\sup_{K_i}K_i^{1/2}\normL{2}{g_{i,n}^{K_i}}\\
\lesssim M_{max}^{(-1)+}\prod_{i=1}^3\norme{\overline{\fl{M_i}^0}}{u_{M_i}}
\end{multline*}
which is sufficient as well.

Finally, we treat $G_3$, using now (\ref{estimation forme trilineaire grande modulation}) and (\ref{estimation controle terme de bord estimation energie}) :
\begin{multline*}
G_3 \lesssim (1\ou M_{min})^{(1/4)+}M_{max}^{(-5/4)+}\sum_{(K_1,K_2,K_3)\in A_3}K_{max}^{0-}\prod_{i=1}^3\sup_{n\in\mathcal{B}}\sup_{K_i}\normL{2}{g_{i,n}^{K_i}}\\
\lesssim M_{max}^{(-1)+}\prod_{i=1}^3\norme{\overline{\fl{M_i}^0}}{u_{M_i}}
\end{multline*}
which concludes the proof of lemma~\ref{lemme estimation trilineaire}.
\end{proof}

Following \cite[Lemme 6.1 (b)]{IonescuKenigTataru2008}, we then use the previous estimate to control the special terms in the energy estimate~\ref{estimation d'energie} :
\begin{lemme}
Let $T\in]0;1]$, $b_1\in [0;1/8]$, $M,M_1\in 2^{\Z}$, with $M\supeg 10(1\ou M_1)$, and $u\in \overline{\fl{M}^{b_1}}(T)$, $v\in \overline{\fl{M_1}^{b_1}}(T)$. Then
\begin{multline}\label{estimation terme specifique estimation energie}
\left|\int_{[0,T]\times\R\times\Tl}P_M u\cdot P_M (P_{M_1}v\cdot\drx u)\dt\dx\dy \right|\\
 \lesssim M_1\Lambda_{b_1}(M_1,M)\norme{\overline{\fl{M_1}^{b_1}}(T)}{P_{M_1}v}\sum_{M_2\sim M}\norme{\overline{\fl{M_2}^{b_1}}(T)}{P_{M_2} u}^2
\end{multline}
\end{lemme}
\begin{proof}
First, we chop the integral in the left-hand side of (\ref{estimation terme specifique estimation energie}) into two terms
\begin{multline*}
\int_{[0,T]\times\R\times\Tl}P_M u\cdot P_M (\drx u P_{M_1}v) \\ = \int_{[0,T]\times\R\times\Tl}P_M u\cdot P_M\drx u\cdot P_{M_1}v + \int_{[0,T]\times\R\times\Tl}P_M u\cdot [P_M (\drx u P_{M_1}v)- P_M\drx u\cdot P_{M_1}v] \\= I + II
\end{multline*}
The first term is easy to control : integrating by parts and using (\ref{equation energie forme trilineaire}), we get the bound
\begin{multline*}
\left|I\right| = \left| \frac{1}{2}\int_{[0,T]\times\R\times\Tl}(P_M u)^2\cdot \drx P_{M_1}v\right|\\
\lesssim M_1\Lambda_{b_1}(M_1,M)\norme{\overline{\fl{M}^{b_1}}(T)}{P_M u}^2\norme{\overline{\fl{M_1}^{b_1}}(T)}{P_{M_1}v}
\end{multline*}
To deal with $II$, we proceed as for the previous lemma : after choosing some extensions (still denoted $u\in \overline{\fl{M}^{b_1}}$ and $v\in \overline{\fl{M_1}^{b_1}}$) of $u$ and $v$ to $\R$, we chop the integral in
\[
II=\sum_{n\in \Z}\int_{\R^2\times\Tl}P_Mu_n \cdot [P_M (\drx u_n P_{M_1}v_n)- P_M\drx u_n\cdot P_{M_1}v_n]
\]
where we define $u_n:= \mathbb{1}_{[0,T]}\gamma(Mt-n)u$ and $v_n:=\mathbb{1}_{[0,T]}\gamma(Mt-n)v$ for a function $\gamma$ as in the previous lemma.

Using Plancherel identity, we can write $II$ as
\[
II = \sum_{n\in\Z} \int_{\R^2\times\Zl} \widehat{P_M u_n}\cdot \int_{\R^2\times\Zl} K(\zeta,\zeta_1)\widehat{u_n}(\zeta-\zeta_1)\widehat{\drx P_{M_1}v_n}(\zeta_1)\dzeta_1\dzeta
\]
where the kernel $K$ is given by
\[
K(\zeta,\zeta_1) = \frac{\xi-\xi_1}{\xi_1}\left[\eta_M(\xi)-\eta_M(\xi-\xi_1)\right]\widetilde{\eta_{M_1}}(\xi_1)\sum_{M_2\sim M}\eta_{M_2}(\xi-\xi_1)
\]
The last sum appears since $|\xi|\sim M$ and $|\xi_1|\sim M_1\infeg M/10$, thus $|\xi-\xi_1|\sim M$.\\
Using the mean value theorem, we can bound the kernel with
\begin{equation}\label{estimation noyau estimation energie}
\left|K(\zeta,\zeta_1)\right|\lesssim\left|\frac{\xi-\xi_1}{\xi_1}\right|M^{-1}|\xi_1|\widetilde{\eta_{M_1}}(\xi_1)\sum_{M_2\sim M}\eta_{M_2}(\xi-\xi_1) \lesssim \widetilde{\eta_{M_1}}(\xi_1)\sum_{M_2\sim M}\eta_{M_2}(\xi-\xi_1)
\end{equation}
Therefore, as in \cite[Lemma 6.1 (b)]{IonescuKenigTataru2008}, (\ref{estimation terme specifique estimation energie}) follows after repeating the proof of (\ref{equation energie forme trilineaire}) and using (\ref{estimation noyau estimation energie}).
\end{proof}
We finally prove (\ref{estimation d'energie}). From now on, we fix $b_1=1/8$ and drop the parameter when writing the main spaces.
\begin{proposition}\label{proposition estimation energie}
Let $T\in ]0;1]$ and $u\in\mathcal{C}([-T,T],\El^{\infty})$ be a solution of
\begin{equation}\label{equation energie KP1}
\begin{cases} \drt u + \drx^3 u - \drx^{-1}\dry u +u\drx u = 0\\u(0,x)=u_0(x)\end{cases}
\end{equation}
on $[-T,T]$.
Then for any $\alpha \supeg 1$,
\begin{equation}\label{estimation energie periode}
\norme{\Bl^{\alpha}(T)}{u}^2\lesssim \norme{\El^{\alpha}}{u_0}^2 + \norme{\Fl(T)}{u}\norme{\Fl^{\alpha}(T)}{u}^2
\end{equation}
\end{proposition}
\begin{proof}
Using the definitions of $\Bl^{\alpha}(T)$ (\ref{definition espace energie})  and $p$ (\ref{definition p}) along with (\ref{equation decomposition norme F}), it suffices to prove
\begin{multline} \label{equation energie dyadique}
\sum_{M_3\supeg 1} \sup_{t_{M_3}\in [-T;T]}M_3^{2\alpha}\normL{2}{P_{M_3}u(t_{M_3})}^2-M_3^{2\alpha}\normL{2}{P_{M_3} u_0}^2\\ \lesssim \norme{\Fl(T)}{u}\sum_{M_3\supeg 1}M_3^{2\alpha}\norme{\fl{M_3}^{b_1}(T)}{P_{M_3}u}^2
\end{multline}
and
\begin{multline}\label{equation energie dyadique poids}
\sum_{M_3\supeg 1} \sup_{t_{M_3}\in [-T;T]}M_3^{2(\alpha-1)}\normL{2}{P_{M_3}\drx^{-1}\dry u(t_{M_3})}^2-M_3^{2(\alpha-1)}\normL{2}{P_{M_3}\drx^{-1}\dry u_0}^2\\ \lesssim \norme{\Fl(T)}{u}\sum_{M_3\supeg 1}M_3^{2\alpha}\norme{\fl{M_3}^{b_1}(T)}{P_{M_3}u}^2
\end{multline}

Let us start with (\ref{equation energie dyadique}).\\
Applying $P_{M_3}$ to (\ref{equation energie KP1}), multiplying by $P_{M_3}u$ and integrating, we get
\begin{multline}\label{equation energie integrale}
\normL{2}{P_{M_3}u(t_{M_3})}^2 - \normL{2}{P_{M_3}u_0}^2 = \int_0^{t_{M_3}}\ddt\normL{2}{P_{M_3} u(t)}^2\dt\\ \lesssim \left|\int_0^{t_{M_3}}\int_{\R\times\Tl} P_{M_3}u \cdot P_{M_3}(u\drx u )\dt'\dx\dy\right|
\end{multline}
since $\drx^3$ and $\drx^{-1}\dry^2$ are skew-adjoint.\\
We separate the right-hand side of (\ref{equation energie integrale}) in
\begin{align}
&\sum_{M_1\infeg M_3/10}\int_{[0,t_{M_3}]\times\R\times\Tl} P_{M_3}u\cdot P_{M_3}\left(P_{M_1}u\cdot\drx u\right)\dt\dx\dy \label{terme 1.1 estimation energie}\\& + \sum_{M_1\gtrsim M_3}\sum_{M_2>0} \int_{[0,t_{M_3}]\times\R\times\Tl}(P_{M_3})^2u\cdot P_{M_1}u\cdot \drx P_{M_2}u \dt\dx\dy
\label{terme 1.2 estimation energie}
\end{align}
Using (\ref{estimation terme specifique estimation energie}) and Cauchy-Schwarz inequality in $M_1$, the first term (\ref{terme 1.1 estimation energie}) is estimated by
\begin{multline*}
(\ref{terme 1.1 estimation energie}) \lesssim \sum_{M_1\infeg M_3/10}M_1\Lambda_{b_1}(M_1,M_3)\norme{\overline{\fl{M_1}^{b_1}}(T)}{P_{M_1}u}\sum_{M_2\sim M_3}\norme{\overline{\fl{M_2}^{b_1}}(T)}{P_{M_2} u}^2\\
\lesssim \norme{\Fl(T)}{u}\sum_{M_2\sim M}\norme{\overline{\fl{M_2}^{b_1}}(T)}{P_{M_2} u}^2
\end{multline*}
since
\[\left[\sum_{0<M_1\infeg M_3/10}(1\ou M_1)^{-2}M_1^2\Lambda_{b_1}(M_1,M_3)^2\right]^{1/2}\lesssim 1\]
Thus
\[\sum_{M_3\supeg 1}M_3^{2\alpha}\cdot(\ref{terme 1.1 estimation energie}) \lesssim \norme{\Fl(T)}{u}\norme{\Fl^{\alpha}(T)}{u}^2\]
To treat (\ref{terme 1.2 estimation energie}), we use (\ref{equation energie forme trilineaire}) and then we separate the sum on $M_2$ depending on whether $M_1\sim M_3\gtrsim M_2$ or $M_1\sim M_2 \gtrsim M_3$ : 
\begin{multline*}
(\ref{terme 1.2 estimation energie}) \lesssim\sum_{M_1\sim M_3}\sum_{M_2\lesssim M_3}M_2\Lambda_{b_1}(M_2,M_3)\prod_{i=1}^3\norme{\overline{\fl{M_i}^{b_1}}(T)}{P_{M_i}u}\\+\sum_{M_1\gtrsim M_3}\sum_{M_2\sim M_1}M_2\Lambda_{b_1}(M_3,M_2)\prod_{i=1}^3\norme{\overline{\fl{M_i}^{b_1}}(T)}{P_{M_i}u}\\ = I + II
\end{multline*}
Applying Cauchy-Schwarz inequality in $M_2$ we get the bounds
\[I \lesssim \norme{\overline{\fl{M_3}^{b_1}}(T)}{P_{M_3}u}^2\norme{\Fl(T)}{u}\]
and
\[II \lesssim \sum_{M_1\gtrsim M_3}M_1\Lambda_{b_1}(M_3,M_1)\norme{\overline{\fl{M_3}^{b_1}}(T)}{P_{M_3}u}\norme{\overline{\fl{M_1}^{b_1}}(T)}{P_{M_1}u}^2\]
Summing on $M_3$ and using Cauchy-Schwarz inequality in $M_3$ and $M_1$ for $II$, we finally get
\begin{multline*}
\sum_{M_3\supeg 1}M_3^{2\alpha}\cdot(\ref{terme 1.2 estimation energie})\lesssim \norme{\Fl(T)}{u}\norme{\Fl^{\alpha}(T)}{u}^2\\ + \sum_{M_3\supeg 1}\sum_{M_1\gtrsim M_3}M_3^{\alpha}\Lambda_{b_1}(M_3,M_1) M_1^{1+\alpha}\norme{\overline{\fl{M_3}^{b_1}}(T)}{P_{M_3}u}\norme{\overline{\fl{M_1}^{b_1}}(T)}{P_{M_1}u}^2\\
\lesssim \norme{\Fl(T)}{u}\norme{\Fl^{\alpha}(T)}{u}^2
\end{multline*}

Now we turn to the proof of (\ref{equation energie dyadique poids}).\\ This time, we apply $P_{M_3}\drx^{-1}\dry$ to (\ref{equation KP1}), we multiply by $P_{M_3}\drx^{-1}\dry u$ and we integrate to get
\begin{multline}\label{equation energie integrale poids}
\normL{2}{P_{M_3}\drx^{-1}\dry u(t_{M_3})}^2-\normL{2}{P_{M_3}\drx^{-1}\dry u_0}^2 \\ \lesssim \left|\int_{[0,t_{M_3}]\times\R\times\Tl}P_{M_3} \drx^{-1}\dry u \cdot P_{M_3}\drx^{-1}\dry (u\drx u)\dt\dx\dy\right|
\end{multline}
using again the skew-adjointness of $\drx^3$ and $\drx^{-1}\dry^2$.\\
The right-hand side of (\ref{equation energie integrale poids}) is similarly split up into
\begin{align}
&\sum_{M_1\infeg M_3/10}\int_{[0,t_{M_3}]\times\R\times\Tl}P_{M_3} \drx^{-1}\dry u \cdot P_{M_3}(P_{M_1}u\cdot\dry u)\dt\dx\dy\label{terme 2.1 estimation energie}\\
&+\sum_{M_1\gtrsim M_3}\sum_{M_2}\int_{[0,t_{M_3}]\times\R\times\Tl}(P_{M_3})^2 \drx^{-1}\dry u \cdot P_{M_1}u\cdot\dry P_{M_2}u\dt\dx\dy\label{terme 2.2 estimation energie}
\end{align}
Writing $v := \drx^{-1}\dry u$, using (\ref{estimation terme specifique estimation energie}) and Cauchy-Schwarz inequality in $M_1$, we obtain
\begin{multline*}
(\ref{terme 2.1 estimation energie})\lesssim \sum_{M_1\infeg M_3/10}M_1\Lambda_{b_1}(M_1,M_3)\norme{\overline{\fl{M_1}^{b_1}}(T)}{P_{M_1}u}\sum_{M_2\sim M_3}\norme{\overline{\fl{M_2}^{b_1}}(T)}{v_{M_2}}^2 \\
\lesssim \norme{\Fl(T)}{u}\sum_{M_2\sim M_3}\norme{\overline{\fl{M_2}^{b_1}}(T)}{P_{M_2}\drx^{-1}\dry u}^2
\end{multline*}
which is enough for (\ref{equation energie dyadique poids}) after summing on $M_3$.

As for (\ref{terme 2.2 estimation energie}), we separate again the sum on $M_2$ :
\begin{multline*}
(\ref{terme 2.2 estimation energie}) =I + II=\sum_{M_1\sim M_3}\sum_{M_2\lesssim M_3}\int_{[0,t_{M_3}]\times\R\times\Tl}(P_{M_3})^2 v \cdot P_{M_1}u\cdot\drx v_{M_2}\dt\dx\dy\\ + \sum_{M_1\gtrsim M_3}\sum_{M_2\sim M_1}\int_{[0,t_{M_3}]\times\R\times\Tl}(P_{M_3})^2 v \cdot P_{M_1}u\cdot\drx P_{M_2}v\dt\dx\dy
\end{multline*}
For the first term, we use again (\ref{equation energie forme trilineaire}) which gives
\[I \lesssim \sum_{M_1\sim M_3}\sum_{M_2\lesssim M_3}M_2\Lambda_{b_1}(M_2,M_3) \norme{\overline{\fl{M_3}^{b_1}}(T)}{P_{M_3}v}\norme{\overline{\fl{M_1}^{b_1}}(T)}{P_{M_1}u}\norme{\overline{\fl{M_2}^{b_1}}(T)}{P_{M_2}v}\]
We first sum on $M_2$ by using Cauchy-Schwarz inequality to get the bound
\[\sum_{M_1\sim M_3}M_3\norme{\overline{\fl{M_3}^{b_1}}(T)}{P_{M_3}v}\norme{\overline{\fl{M_1}^{b_1}}(T)}{P_{M_1}u} \norme{\Fl(T)}{u}\]
and then we can sum on $M_3$ using Cauchy-Schwarz inequality again to get (\ref{equation energie dyadique poids}) for this term.

For the second term, we apply also (\ref{equation energie forme trilineaire}), then we first sum on $M_3$ using Cauchy-Schwarz inequality and $M_3^{2(\alpha-1)}\lesssim M_1^{2(\alpha-1)}$ in this regime, and finally sum on $M_1$ using again Cauchy-Schwarz inequality to get (\ref{equation energie dyadique poids}).
\end{proof}
In the same spirit, following \cite{IonescuKenigTataru2008} we have for the difference equation
\begin{proposition}\label{proposition estimation energie difference}
Let $T\in ]0,1[$ and $u,v\in\Fl(T)$ satisfying
\begin{equation}\label{equation energie difference}
\begin{cases} \drt u + \drx^3 u - \drx^{-1}\dry u +\drx(uv) = 0\\u(0,x)=u_0(x)\end{cases}
\end{equation}
on $[-T,T]\times\R\times\Tl$.
Then
\begin{equation}\label{estimation energie difference periode}
\norme{\overline{\Bl}(T)}{u}^2\lesssim \norme{L^2_{\lambda}}{u_0}^2 + \norme{\Fl(T)}{v}\norme{\overline{\Fl}(T)}{u}^2
\end{equation}
\end{proposition}
and to deal with the equation satisfied by $P_{High}\drx (u_1-u_2)$ we need
\newpage
\begin{proposition}\label{proposition estimation energie difference specifique}
Let $T\in ]0;1]$ and $u\in \overline{\Fl}(T)$ with $u = P_{High}u$. Moreover, let $v\in\Fl(T)$, $w_i\in \Fl(T)$, $i=1,2,3$, and $w_i'\in \overline{\Fl}(T)$, $i=1,2,3$ and finally $h\in\overline{\Fl}(T)$ with $h=P_{\infeg 1}h$. Assume that $u$ satisfies
\begin{equation}\label{equation estimation energie difference specifique}
\drt u + \drx^3u-\drx^{-1}\dry^2 u = P_{High}(v\drx u)+\sum_{i=1}^3P_{High}(w_i w_i')+P_{High}h
\end{equation}
on $[-T;T]\times\R\times\Tl$. Then
\begin{equation}\label{estimation energie difference specifique}
\norme{\overline{\Bl}(T)}{u}^2\lesssim \norme{L^2_{\lambda}}{u_0}^2+\norme{\Fl(T)}{v}\norme{\overline{\Fl}(T)}{u}^2+\norme{\overline{\Fl}(T)}{u}\sum_{i=1}^3\norme{\overline{\Fl}(T)}{w_i}\norme{\overline{\Fl}(T)}{w_i'}
\end{equation}
\end{proposition}
\begin{proof}
(\ref{estimation energie difference periode}) follows from (\ref{estimation energie difference specifique}) after splitting up $u$ into $P_{Low}u$ and $P_{High}u$ and observing that $P_{High}u$ satisfies an equation of type (\ref{equation estimation energie difference specifique}).\\

To prove (\ref{estimation energie difference specifique}), we follow the proof of proposition~\ref{proposition estimation energie}. Using the definitions of $\overline{\Bl}(T)$ (\ref{definition espace energie difference}), it suffices to prove
\begin{multline} \label{equation energie difference dyadique}
\sum_{M_3> 1} \sup_{t_{M_3}\in [-T;T]}\normL{2}{P_{M_3}u(t_{M_3})}^2-\normL{2}{P_{M_3} u_0}^2\\ \lesssim \norme{\Fl(T)}{v}\norme{\overline{\Fl}(T)}{u}^2+\norme{\overline{\Fl}(T)}{u}\sum_{i=1}^3\norme{\overline{\Fl}(T)}{w_i}\norme{\overline{\Fl}(T)}{w_i'}
\end{multline}
Take $M_3> 1$. Applying $P_{M_3}$ to (\ref{equation estimation energie difference specifique}), multiplying by $P_{M_3}u$ and integrating, we get
\begin{multline}\label{equation energie difference specifique integrale}
\normL{2}{P_{M_3}u(t_{M_3})}^2 - \normL{2}{P_{M_3}u_0}^2 = \int_0^{t_{M_3}}\ddt\normL{2}{P_{M_3} u(t)}^2\dt\\ \lesssim \left|\int_0^{t_{M_3}}\int_{\R\times\Tl} P_{M_3}u \cdot P_{M_3}P_{High}(u\drx v)\dt'\dx\dy\right|
\\ +\sum_{i=1}^3\left|\int_0^{t_{M_3}}\int_{\R\times\Tl} P_{M_3}u \cdot P_{M_3}P_{High}(w_iw_i')\dt'\dx\dy\right|
\end{multline}
since $\drx^3$ and $\drx^{-1}\dry^2$ are skew-adjoint. The term in $h$ vanishes after applying $P_{M_3}$, due to its frequency localization.\\

To treat the first term in the right-hand side of (\ref{equation energie difference specifique integrale}) we split it up in
\begin{align}
&\sum_{M_1\infeg M_3/10}\int_{[0,t_{M_3}]\times\R\times\Tl} P_{M_3}u\cdot P_{M_3}\left(P_{M_1}v\cdot\drx u\right)\dt\dx\dy \label{terme 1.1 estimation energie specifique}\\& + \sum_{M_1\gtrsim M_3}\sum_{M_2} \int_{[0,t_{M_3}]\times\R\times\Tl}(P_{M_3})^2u\cdot P_{M_1}v\cdot \drx P_{M_2}u \dt\dx\dy
\label{terme 1.2 estimation energie specifique}
\end{align}
The first term (\ref{terme 1.1 estimation energie specifique}) is estimated similarly to (\ref{terme 2.1 estimation energie}) with $\alpha=1$ and exchanging the roles of $u$ and $v$, whereas for (\ref{terme 1.2 estimation energie specifique}) we proceed as for (\ref{terme 2.2 estimation energie}).

To treat the second term in the right-hand side of (\ref{equation energie difference specifique integrale}), we perform a dyadic decomposition of $w_i$ and $w_i'$. By symmetry we can assume $M_1\infeg M_2$, thus either $M_1 \lesssim M_2\sim M_3$ or $M_3\lesssim M_1\sim M_2$. Then we apply (\ref{equation energie forme trilineaire}) to bound the sum on $M_3$ by
\begin{multline*}
\sum_{M_3\supeg 1}\sum_{M_2\sim M_3}\sum_{M_1\lesssim M_2}\Lambda_{b_1}(M_1,M_2)\\
\cdot\norme{\overline{\fl{M_3}^{b_1}}(T)}{P_{M_3} u}\norme{\overline{\fl{M_1}^{b_1}}(T)}{P_{M_1}w_i}\norme{\overline{\fl{M_2}^{b_1}}(T)}{P_{M_2}w_i'}\\
+\sum_{M_2\supeg 1}\sum_{M_1\sim M_2}\sum_{1\infeg M_3 \lesssim M_2}\Lambda_{b_1}(M_3,M_2)\\
\cdot\norme{\overline{\fl{M_3}^{b_1}}(T)}{P_{M_3} u}\norme{\overline{\fl{M_1}^{b_1}}(T)}{P_{M_1}w_i}\norme{\overline{\fl{M_2}^{b_1}}(T)}{P_{M_2}w_i'}
\end{multline*}
For the second term, we can just use Cauchy-Schwarz inequality in $M_3$ and $M_2$ since $M_1\sim M_2 \gtrsim M_3\supeg 1$. For the first term, we use that
\[\sum_{0<M_1\lesssim M_2}\Lambda_{b_1}(M_1,M_2)\norme{\overline{\fl{M_1}^{b_1}}(T)}{P_{M_1}w_i} \lesssim \norme{\overline{\Fl}(T)}{w_i}\]
Note that this is the only step where we need (\ref{estimation Strichartz basse modulation}) to avoid a logarithmic divergence when summing on very low frequencies, thus we do not need the extra decay for low frequency as in \cite{IonescuKenigTataru2008}.

Thus we finally obtain
\[\sum_{M\supeg 1}(\ref{terme 1.2 estimation energie specifique}) \lesssim \sum_{i=1}^3\norme{\overline{\Fl}(T)}{u}\norme{\overline{\Fl}(T)}{w_i}\norme{\overline{\Fl}(T)}{w_i'}\]
which concludes the proof of (\ref{estimation energie difference specifique}).
\end{proof}
\section{Proof of Theorem~\ref{theoreme principal}}\label{section preuve}
We finally turn to the proof of our main result. We follow the scheme of \cite[Section 6]{KenigPilod}.\\

We begin by recalling a local well-posedness result for smooth data :
\begin{proposition}\label{proposition existence haute regularite}
Assume $u_0\in \El^{\infty}$. Then there exists $T_{\lambda}\in]0;1]$ and a unique solution $u\in\mathcal{C}([-T_{\lambda};T_{\lambda}],\El^{\infty})$ of (\ref{equation KP1}) on $[-T_{\lambda};T_{\lambda}]\times\R\times\Tl$. Moreover, $T_{\lambda}=T(\norme{\El^3}{u_0})$ can be chosen as a nonincreasing function of $\norme{\El^3}{u_0}$.
\end{proposition}
\begin{proof}
This is a straightforward adaptation of \cite{IorioNunes} to the case of partially periodic data. Indeed, proposition~\ref{proposition existence haute regularite} follows from the standard energy estimate (see for example \cite[Lemma 1.3]{Kenig2004})
\begin{equation}\label{estimation energie standard}
\norme{L^{\infty}_T\El^{\alpha}}{u}\infeg C_{\alpha}\norme{\El^{\alpha}}{u_0}\exp\left(\widetilde{C}_{\alpha}\norme{L^1_TL^{\infty}_{xy}}{\drx u}\right)
\end{equation}
along with the Sobolev embedding
\[
\norme{L^1_TL^{\infty}_{xy}}{\drx u}\lesssim T\norme{L^{\infty}_T\El^3}{u}
\]
\end{proof}

\subsection{A priori estimates for smooth solutions}
In this subsection we improve the control on the previous solutions.
\begin{proposition}\label{proposition controle solutions regulieres}
There exists $\epsilon_0\in ]0;1]$ such that for $u_0\in \El^{\infty}$ with
\begin{equation}\label{condition taille donnee initiale}
\norme{\El}{u_0}\infeg \epsilon_0
\end{equation}
then there exists a unique solution $u$ to (\ref{equation KP1}) in $\mathcal{C}([-1;1],\El^{\infty})$, and it satisfies for $\alpha=1,2,3$,
\begin{equation}\label{estimation controle norme avec donnee}
\norme{\Fl^{\alpha}(1)}{u}\infeg C_{\alpha} \norme{\El^{\alpha}}{u_0}
\end{equation}
where $C_{\alpha}>0$ is a constant independent of $\lambda$.
\end{proposition}
\begin{proof}
Let $T=T\left(\norme{\El^3}{u_0}\right)\in ]0;1]$ and $u\in\mathcal{C}([-T;T],\El^{\infty})$ be the solution to (\ref{equation KP1}) given by proposition~\ref{proposition existence haute regularite}. Then, for $T'\in [0;T]$, we define
\begin{equation}\label{definition X argument continuite}
\Gl{\alpha}(T') := \norme{\Bl^{\alpha}(T')}{u}+\norme{\Nl^{\alpha}(T')}{u\drx u}
\end{equation}
Recalling (\ref{equation estimation linéaire})-(\ref{estimation bilineaire globale})-(\ref{estimation energie periode}) for $\alpha \in\N^*$, we get
\begin{equation}\label{estimations synthese}
\begin{cases}
\norme{\Fl^{\alpha}(T)}{u}\lesssim \norme{\Bl^{\alpha}(T)}{u}+\norme{\Nl^{\alpha}(T)}{f}\\
\norme{\Nl^{\alpha}(T)}{\drx(uv)}\lesssim \norme{\Fl^{\alpha}(T)}{u}\norme{\Fl(T)}{v}+\norme{\Fl(T)}{u}\norme{\Fl^{\alpha}(T)}{v}\\
\norme{\Bl^{\alpha}(T)}{u}^2\lesssim \norme{\El^{\alpha}}{u_0}^2 + \norme{\Fl(T)}{u}\norme{\Fl^{\alpha}(T)}{u}^2
\end{cases}
\end{equation}
Thus, combining those estimates first with $\alpha=1$, we deduce that
\begin{equation}\label{estimation argument continuite}
\Gl{1}(T')^2 \infeg c_1\norme{\El}{u_0}^2+c_2\left(\Gl{1}(T')^3+\Gl{1}(T')^4\right)
\end{equation}
Let us remind here that the constants appearing in (\ref{equation estimation linéaire})-(\ref{estimation bilineaire globale})-(\ref{estimation energie periode}) do \emph{not} depend on $\lambda \supeg 1$, so neither does (\ref{estimation argument continuite}). Thus, using lemma~\ref{lemme continuite} below and a continuity argument, we get that there exists $T_0=T_0(\epsilon_0)\in ]0;1]$ such that $\Gl{1}(T)\infeg 2c_0\epsilon_0$ for $T\in [0;T_0]$. Thus, if we choose $\epsilon_0$ small enough such that
\[2c_2c_0\epsilon_0+ 4c_2c_0^2\epsilon_0^2<\frac{1}{2}\]
then
\[\Gl{1}(T)\lesssim \norme{\El}{u_0}\]for $T\in [0;T_0]$.

(\ref{estimation controle norme avec donnee}) for $\alpha=1$ then follows from (\ref{equation estimation linéaire}).

Next, substituting the estimate obtained above in (\ref{estimations synthese}), we infer that for $\alpha= 2,3$
\[\Gl{\alpha}(T)^2\infeg c_{\alpha}\norme{\El^{\alpha}}{u_0}^2+\widetilde{c_{\alpha}}\epsilon_0\Gl{\alpha}(T)^2\]
which in turn, up to chosing $\epsilon_0$ even smaller such that $\widetilde{c_{\alpha}}\epsilon_0<1/2$, gives (\ref{estimation controle norme avec donnee}) for $\alpha=2,3$. 

To reach $T=1$, we just have to use (\ref{estimation controle norme avec donnee}) with $\alpha=3$ along with (\ref{estimation injection continue}), and then extend the lifespan of $u$ by using proposition~\ref{proposition existence haute regularite} a finite number of times.
\end{proof}
Therefore it remains to prove the following lemma :
\begin{lemme}\label{lemme continuite}
Let $T\in ]0;1]$ and $u\in\mathcal{C}\left([-T;T],\El^{\infty}\right)$. Then $\Gl{1} : [0;T]\rightarrow \R$, defined in (\ref{definition X argument continuite}), is continuous and nondecreasing, and furthermore
\[
\lim_{T'\rightarrow 0}\Gl{1}(T')\infeg c_0 \norme{\El}{u_0}
\]
where $c_0>0$ is a constant independent of $\lambda$.
\end{lemme}
\begin{proof}
From the definition of $\Bl(T)$ (\ref{definition espace energie}) it is clear that for $u\in\mathcal{C}([-T;T],\El^{\infty})$,\\ $ T'\mapsto \norme{\Bl(T')}{u}$ is nondecreasing and continuous and satisfies 
\[
\lim_{T'\rightarrow 0}\norme{\Bl(T')}{u}\lesssim \norme{\El}{u_0}
\]
where the constant only depends on the choice of the dyadic partition of unity.

Thus it remains to prove that for all $v\in\mathcal{C}([-T;T],\El^{\infty})$,  ${\displaystyle T'\mapsto \norme{\Nl(T')}{v}}$ is increasing and continuous on $[0;T]$ and satisfies
\begin{equation}\label{limite norme N}
\lim_{T'\rightarrow 0}\norme{\Nl(T')}{v}=0
\end{equation}
The proof is the same as in \cite[Lemma 6.3]{KenigPilod} or \cite[Lemma 8.1]{GuoOh2015} : first, for $M>0$ and $T'\in [0;T]$, take an extension $v_M$ of $P_Mv$ outside of $[-T;T]$, then using the definition of $\nl{M}^{b_1}$ we get
\[\norme{\nl{M}^{b_1}(T')}{P_M v}\lesssim \norme{\nl{M}^{b_1}}{\chi_{T'}(t)v_M}\lesssim \normL{2}{p\cdot\F\left\{\chi_{T'}(t)v_M\right\}}\]
Using the Littlewood-Paley theorem, we obtain the bound
\begin{multline}\label{estimation continuite en 0 norme N}
\norme{\Nl(T')}{v}=\left(\sum_{M>0}(1\ou M)^2\norme{\nl{M}^{b_1}(T')}{P_Mv}^2\right)^{1/2}\\
 \lesssim \left(\sum_{M>0}(1\ou M)^2\normL{2}{p\cdot\F\left\{\chi_{T'}(t)v_M\right\}}^2\right)^{1/2}\\
  \lesssim \norme{L^2_T\El}{\chi_{T'}v}
 \lesssim (T')^{1/2}\norme{L^{\infty}_T\El}{v}
\end{multline}
This proves (\ref{limite norme N}) and the continuity at $T'=0$. The nondecreasing property follows from the definition of $\norme{Y(T')}{\cdot}$ (\ref{definition norme localisation T}). It remains to prove the continuity in $T_0\in ]0;T]$. 

Let $\epsilon >0$. If we define for $u_0\in L^2(\R\times\T)$ and $L>0$,
\[\P_L u_0 := \F^{-1}\left\{\chi_L(\omega(\xi,q))\widehat{u_0}\right\}\]
then by monotone convergence theorem we can take $L$ large enough such that 
\[\norme{\Nl(T_0)}{(\mathrm{Id}-\P_L)v}<\epsilon\]
Then it suffices to show that there exists $\delta_0>0$ such that for $r\in [1-\delta_0;1+\delta_0]$,
\[\left|\norme{\Nl(T_0)}{v_L}-\norme{\Nl(rT_0)}{v_L}\right|<\epsilon\]
Thus we may assume $v= \P_L v$ in the sequel. In particular, $P_Mv =0$ if $M^3\gtrsim L$.

As in \cite{IonescuKenigTataru2008}, we define for $r$ close to 1 the scaling operator
\[D_r(v)(t,x,y):=v(t/r,x,y)\]
Proceeding as in (\ref{estimation continuite en 0 norme N}), we have
\[\norme{\Nl(T')}{v-D_{T'/T_0}(v)}\lesssim (T')^{1/2}\norme{L_{T}^{\infty}\El}{v-D_{T'/T_0}(v)}\underset{T'\rightarrow T_0}{\longrightarrow}0\]
where we use that $v\in\mathcal{C}([-T;T],\El)$ to get the convergence.

Consequently, we are left with proving
\begin{equation}\label{estimation continuite norme liminf}
\norme{\Nl(T_0)}{v}\infeg \liminf_{r\rightarrow 1}\norme{\Nl(rT_0)}{D_r(v)}
\end{equation}
and 
\begin{equation}\label{estimation continuite norme limsup}
 \limsup_{r\rightarrow 1}\norme{\Nl(rT_0)}{D_r(v)}\infeg \norme{\Nl(T_0)}{v}
\end{equation}
Let us begin with (\ref{estimation continuite norme liminf}). Fixing $\widetilde{\epsilon}>0$ and $r\in[1/2;2]$, for any $M\in 2^{\Z}$, $M^3\lesssim L$, we can choose an extension $v_{M,r}$ of $P_MD_{r}(v)$ satisfying $v_{M,r}\equiv P_MD_r(v)$ on $[-rT_0;rT_0]$ and 
\[\norme{\nl{M}^{b_1}}{v_{M,r}}\infeg \norme{\nl{M}^{b_1}(rT_0)}{P_MD_r(v)}+\widetilde{\epsilon}\]
Since $D_{1/r}(v_{M,r})\equiv P_Mv$ on $[-T_0;T_0]$, it defines an extension of $P_Mv$ and thus
\[\norme{\Nl(T_0)}{v}\infeg \left(\sum_{M\lesssim L^{1/3}}(1\ou M)^2\norme{\nl{M}^{b_1}}{D_{1/r}(v_{M,r})}^2\right)^{1/2}\]
Finally, it remains to prove that
\begin{equation}\label{estimation cle liminf}
\norme{\nl{M}^{b_1}}{D_{1/r}(v_{M,r})}\infeg \psi(r)\norme{\nl{M}^{b_1}}{v_{M,r}}
\end{equation}
to get (\ref{estimation continuite norme liminf}), where $\psi$ is a continuous function defined on a neighborhood of $r=1$ and satisfying $\limit{r\rightarrow 1}\psi(r)=1$.

From the definition of $\nl{M}^{b_1}$, we have
\begin{multline*}
\norme{\nl{M}^{b_1}}{D_{1/r}(v_{M,r})}\\=\sup_{t_M\in\R}\norme{\X{M}^{b_1}}{(\tau-\omega+i(1\ou M))^{-1}p\F\left\{\chi_{(1\ou M)^{-1}}(\cdot-t_M)D_{1/r}(v_{M,r})\right\}}
\end{multline*}
 and a computaton gives
\[\chi_{(1\ou M)^{-1}}(\cdot-t_M)D_{1/r}(v_{M,r}) = D_{1/r}\left(\chi_{r(1\ou M)^{-1}}(\cdot -rt_M)v_{M,r}\right)\]
so that
\[\F\left\{\chi_{(1\ou M)^{-1}}(\cdot-t_M)D_{1/r}(v_{M,r})\right\} = r^{-1}D_r\left(\F\left\{\chi_{r(1\ou M)^{-1}}(\cdot -rt_M)v_{M,r}\right\}\right)\]
Thus, using the definition of $\X{M}^{b_1}$, the left-hand side of (\ref{estimation cle liminf}) equals
\begin{multline*}
r^{-1/2}\sup_{\widetilde{t_M}\in\R}\sum_{K\supeg 1}K^{1/2}\beta_{M,K}^{b_1}\\
\normL{2}{(r\tau -\omega+i(1\ou M))^{-1}\rho_K(r\tau - \omega)p\F\left\{\chi_{r(1\ou M)^{-1}}(\cdot - \widetilde{t_M})v_{M,r}\right\}}
\end{multline*}
Now, for $r\sim 1$, we observe that for $K \supeg 10^{10} L$, we have $|\tau|\sim |\tau - \omega|\sim |r\tau - \omega|\sim K$, whereas for $K \lesssim L$ we have $|\tau|$, $|\tau-\omega|$ and $|r\tau - \omega| \lesssim L$.

Thus,
\begin{multline}\label{estimation continuite comutateur 1}
\left|\frac{1}{(r\tau-\omega)^2+(1\ou M)^2}-\frac{1}{(\tau-\omega)^2+(1\ou M)^2}\right|\\ \lesssim |1-r|(1\ou L)^2\cdot\frac{1}{(\tau-\omega)^2+(1\ou M)^2}
\end{multline}
and the use of the mean value theorem provides
\begin{equation}\label{estimation continuite comutateur 2}
\left|\rho_K(r\tau - \omega)-\rho_K(\tau-\omega)\right|
\lesssim |1-r|\begin{cases}
{\displaystyle \sum_{K'\sim K}\rho_{K'}(\tau-\omega) \text{ if }K\supeg 10^{10}L}\\ {\displaystyle K^{-1}L\sum_{K'\lesssim L}\rho_{K'}(\tau-\omega) \text{ if }K\lesssim L}
\end{cases}
\end{equation}
Combining all the estimates above, we get the bound
\begin{multline}\label{estimation continuite comutateur 3}
\norme{\nl{M}^{b_1}}{D_{1/r}(v_{M,r})} \\
\infeg \widetilde{\psi}(r)\sup_{\widetilde{t_M}}\norme{\X{M}^{b_1}}{(\tau-\omega+i(1\ou M))^{-1}p\F\left\{\chi_{r(1\ou M)^{-1}}(\cdot -\widetilde{t_M})v_{M,r}\right\}}
\end{multline}
where ${\displaystyle \widetilde{\psi}(r) = r^{-1/2}\left(1+C(1\ou L)^2|r-1|\right)^{3/2}}\underset{r\rightarrow 1}{\longrightarrow}1$.

It remains to treat the time localization term : using the fundamental theorem of calculus, we have
\[F(t-\widetilde{t_M}) := \chi_{r(1\ou M)^{-1}}(t-\widetilde{t_M}) - \chi_{(1\ou M)^{-1}}(t-\widetilde{t_M}) = \int_1^{r^{-1}}s^{-1}\varphi(s(1\ou M)(t-\widetilde{t_M}))\ds\]
with $\varphi(t) := t\chi'(t)$. In particular, for $r\in[1/2;2]$, from the support property of $\chi$, the support of $F(\cdot -\widetilde{t_M})$ is included in $[\widetilde{t_M}-4(1\ou M);\widetilde{t_M}+4(1\ou M)]$, thus we can represent 
\[F(t-\widetilde{t_M}) = F(t-\widetilde{t_M})\sum_{|\ell|\infeg 4}\gamma((1\ou M)(t-\widetilde{t_M} - \ell)\chi_{(1\ou M)^{-1}}(t-\widetilde{t_M}-\ell(1\ou M)^{-1})\]
where $\gamma$ is a smooth partition of unity with $\supp \gamma\subset [-1;1]$ satisfying $\forall x\in\R$, ${\displaystyle \sum_{\ell\in\Z}\gamma(x-\ell) = 1}$.

Now, using Minkowski's integral inequality to deal with the integral in $s$, the right-hand-side of (\ref{estimation continuite comutateur 3}) is less than
\begin{multline*}
\widetilde{\psi}(r)\left(\norme{\nl{M}^{b_1}}{v_{M,r}}+\int_{I(r)}s^{-1}\sup_{\widetilde{t_M}}\sum_{|\ell|\infeg 4}\left|\left|(\tau-\omega+i(1\ou M))^{-1}p\F\left\{\varphi(s(1\ou M)(t-\widetilde{t_M})\right.\right.\right.\right. \\
\left.\left.\left.\left.\cdot\gamma_{(1\ou M)^{-1}}(t-\widetilde{t_M}-(1\ou M)^{-1}\ell)\chi_{(1\ou M)^{-1}}(t-\widetilde{t_M}-(1\ou M)^{-1}\ell)v_{M,r}\right\}\right|\right|_{\X{M}^{b_1}}\ds\right)
\end{multline*}
with $I(r) = [1;r^{-1}]$ if $r\in [1/2;1]$ and $I(r) = [r^{-1};1]$ if $r\in [1;2]$.

Since $\varphi(t) = t\chi'(t)$ and $\gamma$ are smooth, twice the use of (\ref{propriete XM}) and (\ref{estimation cle localisation XM}) (with $K_0 = s(1\ou M)$ and $K_0=(1\ou M)$ respectively) provides the final bound
\begin{equation}\label{estimation continuite liminf finale}
\norme{\nl{M}^{b_1}}{D_{1/r}(v_{M,r})}\infeg \widetilde{\psi}(r)\left(1+C\left|\ln(r)\right|\right)\norme{\nl{M}^{b_1}}{v_{M,r}}
\end{equation}
(here we used that the implicit constant in (\ref{propriete XM}) and (\ref{estimation cle localisation XM}) are independent of $s$). This concludes the proof of (\ref{estimation continuite norme liminf}).

To prove (\ref{estimation continuite norme limsup}), as before we may assume $v =\P_Lv$. Given $\widetilde{\epsilon}>0$ for any $M>0$ we take an extension $v_M$ of $P_Mv$ outside of $[-T_0;T_0]$ and satisfying ${\displaystyle \norme{\nl{M}^{b_1}}{v_M}\infeg \norme{\nl{M}^{b_1}(T_0)}{P_Mv}+\widetilde{\epsilon}}$. Then for $r\in [-1/2;2]$, $D_r(v_M)$ defines an extension of $P_MD_r(v)$ outside of $[-rT_0;rT_0]$. Then, since in the proof of (\ref{estimation continuite liminf finale}) we did not used the dependence in $r$ of $v_{M,r}$, the same estimate actually holds for $v_M$, and thus
\[\norme{\nl{M}^{b_1}}{D_r(v_M)} \infeg \psi(1/r)\norme{\nl{M}^{b_1}}{v_M}\]
which is enough for (\ref{estimation continuite norme limsup}) and thus concludes the proof of the lemma.
\end{proof}
\subsection{Global well-posedness for smooth data}
In view of the previous proposition, theorem~\ref{theoreme principal}~\ref{theoreme principal partie 1} follows from the conservation of the energy.

Indeed, take $u_0\in\El^{\infty}$ satisfying 
\begin{equation}\label{critere taille donnee existence globale}
\norme{\El}{u_0}\infeg \epsilon_1\infeg \epsilon_0
\end{equation}
and let $T^*:= \sup\{T\supeg 1,~\norme{\El}{u(T)}<+\infty\}$ where $u$ is the unique maximal solution of (\ref{equation KP1}) given by proposition~\ref{proposition controle solutions regulieres}. Then, using the anisotropic Sobolev estimate (see \cite[Lemma 2.5]{Tom})
\begin{equation}\label{estimation Sobolev anisotrope}
\int_{\R\times\Tl}u_0(x,y)^3\dx\dy\infeg 2\normL{2}{u_0}^{3/2}\normL{2}{\drx u_0}\normL{2}{\drx^{-1}\dry u_0}^{1/2}
\end{equation}
we have for $T<T^*$
\begin{multline*}
\norme{\El}{u(T)}^2 = \mathcal{M}(u(T)) + \mathcal{E}(u(T)) + \frac{1}{3}\int_{\R\times\T}u^3(T,x,y)\dx\dy \\
\infeg \mathcal{M}(u(T)) + \mathcal{E}(u(T)) + 2\normL{2}{u(T)}^{3/2}\normL{2}{\drx u(T)}\normL{2}{\drx^{-1}\dry u(T)}^{1/2}
\\
\infeg  \mathcal{M}(u(T)) + \mathcal{E}(u(T)) + 2\mathcal{M}(u(T))\norme{\El}{u(T)}^2
\end{multline*}
Thus, from the conservation of $\mathcal{M}$ and $\mathcal{E}$ (as $u$ is a smooth solution), we finally obtain
\[
\norme{\El}{u(T)}^2\lesssim \mathcal{M}(u_0) + \mathcal{E}(u_0)<+\infty
\]
for any $T<T^*$ provided $\epsilon_1^2<1/4$, from which we get $T^* = +\infty$.\\

Finally, let us notice that equation (\ref{equation KP1}) admits the scaling
\begin{equation}\label{propriete scaling}
u_{\lambda}(t,x,y):= \lambda^{-1}u(\lambda^{-3/2}t,\lambda^{-1/2}x,\lambda^{-1}y),~(x,y)\in\R\times\T_{\lambda\lambda_0}
\end{equation}
meaning that $u_{\lambda}$ is a solution of (\ref{equation KP1}) on $[-\lambda^{3/2}T;\lambda^{3/2}T]\times \R\times \T_{\lambda\lambda_0}$ if and only if $u$ is a solution of (\ref{equation KP1}) on $[-T;T]\times\R\times\T_{\lambda_0}$. Moreover,
\[\norme{\E_{\lambda\lambda_0}}{u_{\lambda}(0)} \lesssim \lambda^{-1/4}\norme{\E_{\lambda_0}}{u(0)}\]
Thus, take $u_0\in \E_{\lambda_0}^{\infty}$. If ${\displaystyle \norme{\E_{\lambda_0}}{u_0}>\epsilon_1}$, then there exists \[\lambda=\lambda\left(\norme{\E_{\lambda_0}}{u_0}\right)\sim \epsilon_1^{-4}\norme{\E_{\lambda_0}}{u_0}^{-4}> 1\]
such that ${\displaystyle \norme{\E_{\lambda\lambda_0}^{\alpha}}{u_{0,\lambda}}\infeg \epsilon_1}$ (since $\epsilon_1>0$ is independent of $\lambda\supeg 1$). Thus, if $u_{\lambda}\in \mathcal{C}(\R,\E_{\lambda\lambda_0}^{\infty})$ is the unique global solution associated with $u_{0,\lambda}$ satisfying (\ref{critere taille donnee existence globale}) , then \[u(t,x,y) := \lambda u_{\lambda}(\lambda^{3/2}t,\lambda^{1/2}x,\lambda y)\in\mathcal{C}\left(\R,\E_{\lambda_0}^{\infty}\right)\] is the unique global solution associated with $u_0$.\\

The rest of the section is devoted to the proof of theorem~\ref{theoreme principal}~\ref{theoreme principal partie 2}.
\subsection{Lipschitz bound for the difference of small data solutions}
Let $T>0$, $u_0,v_0\in \El$ and $u,v$ in the class (\ref{classe unicite}) be the corresponding solutions of the Cauchy problems (\ref{equation KP1}).
As before, up to rescaling and using the conservation of $\mathcal{M}$ and $\mathcal{E}$, it suffices to prove uniqueness for $T=1$ and
\[\norme{\El}{u_0},\norme{\El}{v_0}\infeg \epsilon_2 \infeg \epsilon_0\]
 Set $w:=u-v$. Then $w$ is also in the class (\ref{condition taille donnee initiale}) and solves the equation
\begin{equation}\label{equation unicite}
\drt w + \drx^3w - \drx^{-1}\dry w + \drx\left(w\cdot\frac{u+v}{2}\right)=0
\end{equation} 
on $[-1;1]\times\R\times\Tl$.
Then, since $u_0,v_0$ satisfy (\ref{condition taille donnee initiale}), using (\ref{estimation controle norme avec donnee}) and then (\ref{estimation injection continue difference}), (\ref{equation estimation linéaire difference})-(\ref{estimation bilineaire difference globale})-(\ref{estimation energie difference periode}), we obtain for $\epsilon_2$ small enough
\begin{equation}\label{estimation cle unicite}
\norme{L^{\infty}_{[-1;1]}L^2_{xy}}{w}\lesssim \norme{\overline{\Fl}(1)}{w}\lesssim \normL{2}{u_0-v_0}
\end{equation}
from which we get $u\equiv v$ on $[-1;1]$ if $u_0=v_0$.
\subsection{Global well-posedness in the energy space}
In this subsection we end the proof of theorem~\ref{theoreme principal}~\ref{theoreme principal partie 2}. We proceed as in \cite[Section 4]{IonescuKenigTataru2008}.\\

Take $T>0$, and let $u_0\in\El$ and ${\displaystyle (u_{0,n})\in\left(\El^{\infty}\right)^{\N}}$ such that $(u_{0,n})$ converges to $u_0$ in $\El$. Again, up to rescaling we can assume ${\displaystyle \norme{\El}{u_0}\infeg \epsilon\infeg\epsilon_2}$ and ${\displaystyle \norme{\El}{u_{0,n}}\infeg \epsilon\infeg\epsilon_2}$. Using again the conservation of $\mathcal{M}$ and $\mathcal{E}$, it then suffices to prove that ${\displaystyle \left(\Phi^{\infty}(u_{0,n})\right)\in \left(\mathcal{C}([-1;1],\El^{\infty})\right)^{\N}}$ is a Cauchy sequence in $\mathcal{C}([-1;1],\El)$.\\

For a fixed $M>1$ and $m,n\in\N$, we can split
\begin{multline*}
\norme{L^{\infty}_1\El}{\Phi^{\infty}(u_{0,m})-\Phi^{\infty}(u_{0,n})} \infeg \norme{L^{\infty}_1\El}{\Phi^{\infty}(u_{0,m})-\Phi^{\infty}(P_{\infeg M}u_{0,m})}\\+\norme{L^{\infty}_1\El}{\Phi^{\infty}(P_{\infeg M}u_{0,m})-\Phi^{\infty}(P_{\infeg M}u_{0,n})}+\norme{L^{\infty}_1\El}{\Phi^{\infty}(P_{\infeg M}u_{0,n})-\Phi^{\infty}(u_{0,n})}
\end{multline*}
Since
\[
\norme{L^{\infty}_1\El^{\alpha}}{S^{\infty}_T(P_{\infeg M}u_{0,n})}\infeg C(\alpha,M)
\]
thanks to (\ref{equation estimation norme}), the middle term is then controled with the analogous of (\ref{estimation energie standard}) for the difference equation along with a Sobolev inequality with $\alpha$ large enough, which gives
\[
\norme{L^{\infty}_1\El}{\Phi^{\infty}(P_{\infeg M}u_{0,m})-\Phi^{\infty}(P_{\infeg M}u_{0,n})} \infeg C(M)\norme{\El}{u_{0,m}-u_{0,n}}
\]
Therefore it remains to treat the first and last terms. A use of (\ref{estimation injection continue}) provides
\[
\norme{L^{\infty}_1\El}{\Phi^{\infty}(u_{0,m})-\Phi^{\infty}(P_{\infeg M}u_{0,m})} \lesssim \norme{\Fl(1)}{\Phi^{\infty}(u_{0,m})-\Phi^{\infty}(P_{\infeg M}u_{0,m})}
\]
and thus we have to estimate difference of solutions in $\Fl(1)$. Let us write $u_1:=\Phi^{\infty}(u_{0,m})$, $u_2 := \Phi^{\infty}(P_{\infeg M}u_{0,m})$ and $v:=u_1-u_2$.

Using (\ref{equation estimation linéaire}) and (\ref{estimation bilineaire globale}) combined with (\ref{estimation controle norme avec donnee}) we obtain the bound
\[
\norme{\Fl(1)}{v}\lesssim \norme{\Bl(1)}{v}+\norme{\Fl(1)}{v}\epsilon
\]
Therefore, taking $\epsilon$ small enough, it suffices to control ${\displaystyle \norme{\Bl(1)}{v}}$. Using the definition of $\Bl(1)$ (\ref{definition espace energie}), we see that
\[
\norme{\Bl(1)}{v}\infeg \norme{\El}{P_{\infeg 1}v_0} + \norme{\Bl(1)}{P_{\supeg 2}v}
\]
Now, in view of the definition of $\Bl(1)$ and $\overline{\Bl}(1)$, we have 
\[
\norme{\Bl(1)}{P_{\supeg 2}v} \sim \norme{\overline{\Bl}(1)}{\drx P_{\supeg 2}v}+\norme{\overline{\Bl}(1)}{\drx^{-1}\dry P_{\supeg 2}v}
\]
Combining this remark with the previous estimates, we finally get the bound
\begin{equation}\label{equation decomposition U et V}
\norme{\Fl(1)}{v}\lesssim \norme{\El}{v_0}+\norme{\overline{\Bl}(1)}{P_{\supeg 2}\drx v}+\norme{\overline{\Bl}(1)}{P_{\supeg 2}\drx^{-1}\dry v}
\end{equation}
We now define $U := P_{High}\drx v$ and $V:= P_{High}\drx^{-1}\dry v$. We begin by writing down the equations satisfied by $U$ and $V$ :
\begin{multline}\label{equation U}
\drt U + \drx^3 U - \drx^{-1}\dry^2 U = P_{High}(-u_1\cdot\drx U)+ P_{High}(-P_{Low}u_1\cdot\drx^2P_{Low}v)\\+P_{High}(-P_{High}u_1\cdot\drx^2P_{Low}v)+P_{High}(-\drx v\cdot\drx (u_1+u_2)) + P_{High}(-v\cdot\drx^2u_2)
\end{multline}
and
\begin{multline}\label{equation V}
\drt V + \drx^3 V - \drx^{-1}\dry^2V = P_{High}(-u_1\cdot\drx V)+P_{High}(-P_{Low}u_1\cdot \drx P_{Low}\drx^{-1}\dry v)\\+P_{High}(-P_{High}u_1\cdot \drx P_{Low}\drx^{-1}\dry v)+P_{High}(-v\cdot \dry u_2)
\end{multline}
Let us look at  (\ref{equation U}). We set $h:=-P_{Low}u_1\cdot\drx^2P_{Low}v$, $w_1:=-P_{High}u_1$, $w_1':=\drx^2P_{Low}v$, $w_2:=-\drx v$, $w_2':=\drx (u_1+u_2)$ and $w_3:=-v$, $w_3':=\drx^2u_2$. Since $u_1,u_2\in \Fl(1)$ we have $v\in \Fl(1)$, thus $h$, $w_i$ and $w_i'$ satisfy the assumptions of (\ref{estimation energie difference specifique}). Thence we infer
\begin{multline*}
\norme{\overline{\Bl}(1)}{U}^2\lesssim \norme{L^2_{\lambda}}{\drx v_0}^2+\norme{\Fl(1)}{u_1}\norme{\overline{\Fl}(1)}{U}^2\\ + \norme{\overline{\Fl}(1)}{U}\left(\norme{\overline{\Fl}(1)}{P_{High}u_1}\norme{\overline{\Fl}(1)}{\drx^2P_{Low}v}\right.\\ \left.+\norme{\overline{\Fl}(1)}{\drx v}\norme{\overline{\Fl}(1)}{\drx (u_1+u_2)}+\norme{\overline{\Fl}(1)}{v}\norme{\overline{\Fl}(1)}{\drx^2 u_2}\right)
\end{multline*}
Therefore, using (\ref{estimation controle norme avec donnee}) and (\ref{estimation cle unicite}), the previous estimate reads
\begin{multline*}
\norme{\overline{\Bl}(1)}{U}^2\lesssim \norme{\El}{v_0}^2+\epsilon\norme{\overline{\Fl}(1)}{U}^2 + \norme{\overline{\Fl}(1)}{U}\left(\epsilon\norme{L^2_{\lambda}}{v_0}\right.\\ \left.+\norme{\Fl(1)}{v}\epsilon+\norme{L^2_{\lambda}}{v_0}\norme{\Fl^2(1)}{u_2}\right)
\end{multline*}
Proceeding similarly for $V$, we obtain the estimate
\begin{multline*}
\norme{\overline{\Bl}(1)}{V}^2
\lesssim \norme{L^2_{\lambda}}{\drx^{-1}\dry v_0}^2 + \norme{\Fl(1)}{u_1}\norme{\overline{\Fl}(1)}{V}^2+\norme{\overline{\Fl}(1)}{V}\\
\cdot\left(\norme{\overline{\Fl}(1)}{P_{High}u_1}\norme{\overline{\Fl}(1)}{\drx P_{Low}\drx^{-1}\dry v}+\norme{\overline{\Fl}(1)}{v}\norme{\overline{\Fl}(1)}{\drx \drx^{-1}\dry u_2}\right)
\end{multline*}
after applying (\ref{estimation energie difference specifique}). Again, a use of (\ref{estimation controle norme avec donnee}) and (\ref{estimation cle unicite}) gives
\[\norme{\overline{\Bl}(1)}{V}^2\lesssim \norme{\El}{v_0}^2 + \epsilon\norme{\overline{\Fl}(1)}{V}^2+\norme{\overline{\Fl}(1)}{V}\left(\epsilon\norme{\Fl(1)}{v}+\norme{L^2_{\lambda}}{v_0}\norme{\Fl^2(1)}{u_2}\right)\]
Combining the estimates for $U$ and $V$ along with (\ref{equation decomposition U et V}), we get the final bound
\[
\norme{\Fl(1)}{v}\lesssim \norme{\El}{v_0}+\epsilon\norme{\Fl(1)}{v}+\norme{\El^2}{P_{\infeg M}u_{0,m}}\norme{L^2_{\lambda}}{v_0}
\]
since ${\displaystyle \norme{\Fl^2(1)}{u_2}\lesssim \norme{\El^2}{u_2(0)}}$ by (\ref{estimation controle norme avec donnee}).\\

Taking $\epsilon$ small enough and $M>1$ large enough concludes the proof. 
\section{Orbital stability of the line soliton}\label{section stabilite}
In this last section, we turn to the proof of corollary~\ref{theoreme instabilite}. We briefly recall the main steps of \cite[Section 2]{Rousset2012}.\\

Let us remember that equation (\ref{equation KP1}) has a Hamiltonian structure, with Hamiltonian $E(u)$. To study the orbital stability of $Q_c(x-ct)$, we first make a change of variable to see $Q_c(x)$ as a stationary solution of (\ref{equation KP1}) rewriten in a moving frame :
\begin{equation}\label{equation KP1 moving frame}
\drt u - c\drx u + \drx^3u-\drx^{-1}\dry^2 u + u\drx u = 0
\end{equation}
Equation (\ref{equation KP1 moving frame}) still has a Hamitonian structure, with the new Hamiltonian 
\[
\mathcal{E}_c(u) := \mathcal{E}(u)+c\mathcal{M}(u)
\]
The key idea of the proof is then to show, as for the orbital stability of $Q_c$ under the flow of KdV \cite{Benjamin}, that the Hessian of $\mathcal{E}_c$ about $Q_c$ is strictly positive on the codimension-2 subspace ${\displaystyle H:=\left\{\langle v,Q_c\rangle_{L^2} = \langle v,Q_c'\rangle_{L^2} = 0\right\}}$ to get a lower bound on $\mathcal{E}_c(\Phi^1(u_0)(t)) - \mathcal{E}_c(Q_c)$ in term of $\norme{\E}{\Phi^1(u_0)(t) - Q_c}$. 

To study $D^2\mathcal{E}_c(Q_c)$ on $H$, we begin by computing
\begin{multline*}
\mathcal{E}_c(Q_c+v(t)) = \mathcal{E}_c(Q_c) + \left(\normL{2}{\drx v}^2 + \normL{2}{\drx^{-1}\dry v}^2 + c\normL{2}{v}^2 - \int_{\R\times\T}Q_c\cdot v^2\dx\dy\right)\\ - \int_{\R\times\T}v^3\dx\dy
\end{multline*}
The linear term in $v$ vanishes since $Q_c$ is a stationary solution.

Using the Plancherel identity in the $y$ variable, we can write the Hessian of $\mathcal{E}_c$ about $Q_c$ as the sum of the bilinear forms
\[
\frac{1}{2}D^2\mathcal{E}_c(Q_c)(v,v) = \sum_{k\in \Z} B_c^k(\F_yv(t,x,k),\F_yv(t,x,k))
\]
with
\[
B_c^k\left(\widetilde{v}(x),\widetilde{v}(x)\right) = \normL{2}{\drx \widetilde{v}}^2 + k^2\normL{2}{\drx^{-1}\widetilde{v}}^2 + c\normL{2}{\widetilde{v}}^2 - \int_{\R}Q_c\cdot \widetilde{v}^2\dx
\]
Observe that $B_c^0$ is the Hessian about $Q_c$ of the Hamiltonian associated with the KdV equation in a moving frame, and thus by the study in \cite{Benjamin} $B_c^0$ is $H^1$ bounded from below as desired.

To treat the terms with $k\neq 0$, first make the change of test function \[f (x):= \drx^{-1}\F_y v(t,x,k)\in L^2(\R)\] Then, using that $k^2\supeg 1$, we can write 
\[
B_c^k(\F_y v(t,x,k),\F_y v(t,x,k)) \supeg \langle \L_c f,f\rangle
\]
where the linear operator $\L_c$ is defined as 
\[
\L_c := \drx^4 -c\drx^2 +\drx Q_c\drx +1
\]
Since $Q_c$ is exponentially decreasing, $\drx Q_c \drx$ is compact with respect to $\drx^4-c\drx^2+1$ and thus $\mathrm{Spec}_{ess}\L_c \subset [1,+\infty[$. To get a lower bound on $\langle \L_c f,f\rangle$, it remains to study the existence of negative eigenvalues.
Following the method of \cite{ALEXANDER1997}, a change of variables leads to consider the eigenvalue problem
\begin{equation}\label{equation vp}
g^{(4)}-4\left(1-\frac{3}{\cosh^2}\right)g''+3\nu^2g = 0
\end{equation}
where
\[
3\nu^2 = \frac{16}{c^2}(1-\lambda_0)
\]
and $\lambda_0\infeg 0$ is the possible negative eigenvalue. Using again the exponential decreasing of $Q_c$, $g$ behaves at infinity as a solution of the linear equation
\begin{equation}\label{equation vp lineaire}
h^{(4)}-4h''+3\nu^2h = 0
\end{equation}
For each characteristic value $\mu$ of (\ref{equation vp lineaire}), there is an exact solution
\[
g_{\mu}(x):=\e^{\mu x}\left(\mu ^3 + 2\mu-3\mu^2\tanh(x)\right)
\]
of (\ref{equation vp}). For these solutions to behave as $\e^{\mu x}$ at infinity, this requires
\[
\mu^3+2\mu-3\mu^2 = 0
\]
As $\mu$ is also a characteristic value, this implies $\mu =1$ and thus $\nu^2=1$ from which we finally infer
\[
c^2=\frac{16}{3}(1-\lambda_0)
\]
Consequently, there is no possible negative eigenvalue $\lambda_0$ if $c<c^*=4/\sqrt{3}$.
 
Hence we have a lower $L^2$ bound for the bilinear form associated with $\L_c$, which provides the bound
\[
B_c^k(\widetilde{v},\widetilde{v})\gtrsim \normL{2}{\drx^{-1}\widetilde{v}}^2
\]
Linearly interpolating with the obvious bound (since $Q_c\infeg 3c$)
\[
B_c^k(\widetilde{v},\widetilde{v}) \supeg \normL{2}{\drx \widetilde{v}}^2 + \normL{2}{\drx^{-1} \widetilde{v}}^2 - 2c\normL{2}{\widetilde{v}}^2
\]
yields to an $L^2$ lower bound for $B_c^k$, which in return provides the final bound
\[
B_c^k(\widetilde{v},\widetilde{v}) \gtrsim \normH{1}{\widetilde{v}}^2+k^2\normL{2}{\drx^{-1}\widetilde{v}}^2
\]
uniformly in $k$.

The last trilinear term ${\displaystyle \int v^3}$ is treated with the anisotropic Sobolev inequality (\ref{estimation Sobolev anisotrope}).

Combining all the bounds from below finally provides a control of ${\displaystyle \norme{\E}{w}}$ in term of $\mathcal{E}_c(Q_c+w_0)-\mathcal{E}_c(Q_c)$ for any $w\in H$. The end of the proof is then standard (cf. \cite{Benjamin},\cite[Section 2]{Rousset2012}).

\section*{Acknowledgements.}
The author would like to thank his thesis supervisor Nikolay Tzvetkov for pointing out this problem as well as for many helpfull discussions.

\bibliographystyle{plain}
\bibliography{Article}

\begin{thebibliography}{10}

\bibitem{ALEXANDER1997}
J.C. Alexander, R.L. Pego, and R.L. Sachs.
\newblock On the transverse instability of solitary waves in the
  {K}adomtsev-{P}etviashvili equation.
\newblock {\em Physics Letters A}, 226(3):187 -- 192, 1997.

\bibitem{Benjamin}
T.~B. Benjamin.
\newblock The stability of solitary waves.
\newblock {\em Proceedings of the Royal Society of London A: Mathematical,
  Physical and Engineering Sciences}, 328(1573):153--183, 1972.

\bibitem{bourgain1993kp}
Jean Bourgain.
\newblock On the cauchy problem for the {K}adomstev-{P}etviashvili equation.
\newblock {\em Geometric and Functional Analysis}, 3(4):315--341, 1993.

\bibitem{GuoOh2015}
Zihua Guo and Tadahiro Oh.
\newblock Non-existence of solutions for the periodic cubic {NLS} below
  ${L}^{{2}}$.
\newblock {\em International Mathematics Research Notices}, 2016.

\bibitem{GPWW}
Zihua Guo, Lizhong Peng, Baoxiang Wang, and Yuzhao Wang.
\newblock Uniform well-posedness and inviscid limit for the
  {B}enjamin–{O}no–{B}urgers equation.
\newblock {\em Advances in Mathematics}, 228(2):647 -- 677, 2011.

\bibitem{Hadac}
Martin Hadac.
\newblock Well-posedness for the kadomtsev-petviashvili ii equation and
  generalisations.
\newblock {\em Transactions of the American Mathematical Society},
  360(12):6555--6572, 2008.

\bibitem{HadacHerrKoch}
Martin Hadac, Sebastian Herr, and Herbert Koch.
\newblock Well-posedness and scattering for the kp-ii equation in a critical
  space.
\newblock {\em Annales de l'Institut Henri Poincare (C) Non Linear Analysis},
  26(3):917 -- 941, 2009.

\bibitem{ionescu2009}
A.D Ionescu and C.E Kenig.
\newblock Local and global well-posedness of periodic {KP-I} equations.
\newblock {\em Mathematical Aspects of Nonlinear Dispersive Equations. Ann.
  Math. Stud}, 163:181--211, 2009.

\bibitem{IonescuKenigTataru2008}
A.D. Ionescu, C.E. Kenig, and D.~Tataru.
\newblock Global well-posedness of the {KP-I} initial-value problem in the
  energy space.
\newblock {\em Inventiones mathematicae}, 173(2):265--304, 2008.

\bibitem{IorioNunes}
Rafael~José Iório and Wagner Vieira~Leite Nunes.
\newblock On equations of {KP}-type.
\newblock {\em Proceedings of the Royal Society of Edinburgh: Section A
  Mathematics}, 128:725--743, 1 1998.

\bibitem{IsazaMejia}
Pedro Isaza and Jorge Mejía.
\newblock Local and global cauchy problems for the kadomtsev–petviashvili
  (kp–ii) equation in sobolev spaces of negative indices.
\newblock {\em Communications in Partial Differential Equations},
  26(5-6):1027--1054, 2001.

\bibitem{KP1970}
B.~B. {Kadomtsev} and V.~I. {Petviashvili}.
\newblock {On the Stability of Solitary Waves in Weakly Dispersing Media}.
\newblock {\em Soviet Physics Doklady}, 15, December 1970.

\bibitem{Kenig2004}
Carlos~E. Kenig.
\newblock On the local and global well-posedness theory for the {KP-I}
  equation.
\newblock {\em Annales de l'Institut Henri Poincare (C) Non Linear Analysis},
  21(6):827 -- 838, 2004.

\bibitem{KenigPilod}
Carlos~E. Kenig and Didier Pilod.
\newblock Well-posedness for the fifth-order {K}d{V} equation in the energy
  space.
\newblock {\em Trans. Amer. Math. Soc.}, 367(4):2551--2612, 2015.

\bibitem{Koch2008}
H.~Koch and N.~Tzvetkov.
\newblock On finite energy solutions of the {KP-I} equation.
\newblock {\em Mathematische Zeitschrift}, 258(1):55--68, 2008.

\bibitem{KochTzvetkov2003BO}
Herbert Koch and Nikolay Tzvetkov.
\newblock On the local well-posedness of the {B}enjamin-{O}no equation in
  ${H}^s(\mathbb{R})$.
\newblock {\em International Mathematics Research Notices},
  2003(26):1449--1464, 2003.

\bibitem{Mizumachi2012}
Tetsu Mizumachi and Nikolay Tzvetkov.
\newblock Stability of the line soliton of the {KP-II} equation under periodic
  transverse perturbations.
\newblock {\em Mathematische Annalen}, 352(3):659--690, 2012.

\bibitem{MST2002}
L.~Molinet, J.-C. Saut, and N.~Tzvetkov.
\newblock Well-posedness and ill-posedness results for the
  {K}adomtsev-{P}etviashvili-{I} equation.
\newblock {\em Duke Math. J.}, 115(2):353--384, 11 2002.

\bibitem{MST2011}
L.~Molinet, J.-C. Saut, and N.~Tzvetkov.
\newblock Global well-posedness for the {KP-II} equation on the background of a
  non-localized solution.
\newblock {\em Annales de l'Institut Henri Poincare (C) Non Linear Analysis},
  28(5):653 -- 676, 2011.

\bibitem{Molinet2007}
Luc Molinet.
\newblock Global well-posedness in the energy space for the {B}enjamin-{O}no
  equation on the circle.
\newblock {\em Mathematische Annalen}, 337(2):353--383, 2007.

\bibitem{Rousset2012}
Frederic Rousset and Nikolay Tzvetkov.
\newblock Stability and instability of the {K}d{V} solitary wave under the
  {KP-I }flow.
\newblock {\em Communications in Mathematical Physics}, 313(1):155--173, 2012.

\bibitem{SautTzvetkov2001}
J.-C. {Saut} and N.~{Tzvetkov}.
\newblock {On Periodic {KP-I} Type Equations}.
\newblock {\em Communications in Mathematical Physics}, 221:451--476, 2001.

\bibitem{Takaoka}
H.~Takaoka and N.~Tzvetkov.
\newblock On the local regularity of the {K}adomtsev-{P}etviashvili-{II}
  equation.
\newblock {\em International Mathematics Research Notices}, 2001(2):77--114,
  2001.

\bibitem{Tom}
Michael~M Tom.
\newblock On a generalized kadomtsev-petviashvili equation.
\newblock {\em Contemporary Mathematics}, 200:193--210, 1996.

\bibitem{Zhang2015}
Yu~Zhang.
\newblock Local well-posedness of {KP-I} initial value problem on torus in the
  {B}esov space.
\newblock {\em Communications in Partial Differential Equations}, pages 1--26,
  2015.

\end{thebibliography}
\clearpage
\end{document}